\documentclass[a4paper,10pt,reqno]{scrartcl}
\usepackage[utf8]{inputenc}
\usepackage{latexsym,amsmath,amsfonts,amscd,amssymb,amsthm}
\usepackage{mathrsfs}
\usepackage{etoolbox}
\usepackage{pifont}
\usepackage{commath}
\usepackage{tabu}
\usepackage{soul}
\usepackage{mathabx}
\usepackage[all]{xy}
\usepackage{graphicx,arydshln,booktabs,array}
\usepackage{tikz}
\usetikzlibrary[patterns]
\usepackage{color}
\definecolor{cobalt}{RGB}{61,89,171}
\definecolor{blue3}{RGB}{0,0,205}
\usepackage[colorlinks,citecolor=blue3,linkcolor=blue3,urlcolor=blue3,pdfpagemode=UseNone,backref = page]{hyperref}

\newcommand{\Lie}{\mathcal{L}}      
\newcommand{\lvec}[1]{\overleftarrow{#1}}

\theoremstyle{plain}
\newtheorem{theorem}{Theorem}[section]
\newtheorem*{theorem*}{Theorem}
\newtheorem{proposition}[theorem]{Proposition}
\newtheorem{lemma}[theorem]{Lemma}
\newtheorem{corollary}[theorem]{Corollary}

\theoremstyle{definition}
\newtheorem{definition}[theorem]{Definition}
\newtheorem{example}[theorem]{Example}
\theoremstyle{remark}
\newtheorem{remark}[theorem]{Remark}

\newcommand{\bN}{\mathbb{N}}
\newcommand{\bR}{\mathbb{R}}

\newcommand{\bZ}{\mathbb{Z}}
\newcommand{\fX}{\mathfrak{X}}

\newcommand{\bbeta}{\bar{\beta}}
\newcommand{\btheta}{\bar{\theta}}
\newcommand{\bphi}{\overline{\phi}}
\newcommand{\tphi}{\widetilde{\phi}}
\newcommand{\tpsi}{\widetilde{\psi}}
\newcommand{\tchi}{\widetilde{\chi}}
\newcommand{\bomega}{\bar{\omega}}

\newcommand{\cB}{\mathcal{B}}

\newcommand{\cF}{\mathcal{F}}
\newcommand{\cL}{\mathcal{L}}
\newcommand{\cM}{\mathcal{M}}
\newcommand{\tcM}{\widetilde{\cM}}
\newcommand{\bcM}{\overline{\cM}}

\newcommand{\fe}{\mathsf{e}}

\renewcommand{\d}{\delta}

\newcommand{\cext}{\circledcirc}

\newcommand{\fg}{\mathfrak{g}}
\newcommand{\fh}{\mathfrak{h}}
\newcommand{\fk}{\mathfrak{k}}

\newcommand{\fs}{\mathfrak{s}}

\newcommand{\la}{\langle}
\newcommand{\ran}{\rangle}

\DeclareMathOperator{\Aut}{Aut}
\DeclareMathOperator{\Id}{Id}

\begin{document}

\title{On locally conformal symplectic manifolds of the first kind}
\author{Giovanni Bazzoni$^*$ \quad Juan Carlos Marrero$^\dagger$
\\[0.5cm]
 \small\texttt{\href{mailto:bazzoni@math.lmu.de}{bazzoni@math.lmu.de}} \qquad
  \small\texttt{\href{mailto:jcmarrer@ull.edu.es}{jcmarrer@ull.edu.es}}\\[0.1cm]
  {\normalsize\slshape $^*$Mathematisches Institut der Ludwig-Maximilians-Universit\"at M\"unchen, }\\[-0.1cm]
  {\normalsize\slshape Theresienstra\ss e 39, 80333 M\"unchen}\\[-0.1cm]
  {\normalsize\slshape $^\dagger$Universidad de La Laguna, Facultad de Ciencias,}\\[-0.1cm]
  {\normalsize\slshape Dpto. de Matem\'aticas, Estad{\'\i}stica e IO}\\[-0.1cm]
  {\normalsize\slshape Avda. Astrof\'isico Francisco S\'anchez s/n. 38071, La Laguna}\\[-0.1cm]
}
\date{}
\maketitle

\begin{abstract}
We present some examples of locally conformal symplectic structures of the first kind on compact nilmanifolds which do not admit Vaisman metrics. One of these examples does not admit locally 
conformal K\"ahler metrics and all the structures come from left-invariant locally conformal symplectic structures on the corresponding nilpotent Lie groups.
Under certain topological restrictions related with the compactness of the canonical foliation, we prove a structure theorem for locally conformal symplectic manifolds of the first kind. In the 
non compact case, we show that they are the product of a real line with a compact contact manifold and, in the compact case, we obtain that they are mapping tori of compact contact manifolds by strict 
contactomorphisms. Motivated by the aforementioned examples, we also study left-invariant locally conformal symplectic structures on Lie groups. In particular, we obtain a complete description of these structures 
(with non-zero Lee $1$-form) on connected simply connected nilpotent Lie groups in terms of locally conformal symplectic extensions and symplectic double extensions of symplectic nilpotent Lie groups. In order to obtain this description, we study locally conformal symplectic structures of the first kind on Lie algebras. 
\end{abstract}

\vskip1pt{\small\textsf{\textbf{MSC classification [2010]}}: 22E25, 22E60, 53C12, 53D05, 
53D10, 53C55}\vskip0.5pt
\vskip1pt{\small\textsf{\textbf{Key words}}: Locally conformal symplectic structures, contact structures, mapping torus, Lie algebras, Lie groups, nilpotent Lie algebras, nilpotent Lie groups, 
compact nilmanifolds, locally conformal K\"ahler metrics, Vaisman metrics.}\vskip10pt

\tableofcontents

%
%
\section{Introduction}

A locally conformal symplectic structure on a manifold $M$ consists of a pair $(\Phi,\omega)$, where $\Phi$ and $\omega$ are a $2$-form and a $1$-form, respectively, with $\Phi$ non degenerate, $\omega$ closed, 
subject to the equation $d\Phi=\omega\wedge\Phi$. $\omega$ is known as the Lee form. This implies that, locally, $\Phi$ is conformal to a genuine symplectic form, hence the name. If $\omega=d f$, 
then the global conformal change $\Phi\mapsto e^{-f}\Phi$ endows $M$ with a symplectic structure. We use the name globally conformal symplectic structure in this case. Notice that a manifold 
endowed with a locally conformal symplectic structure is orientable and almost complex. A locally conformal symplectic manifold is one endowed with a locally conformal symplectic structure.

Locally conformal symplectic structures were introduced by Lee in \cite{Lee} and then studied extensively by Vaisman (see \cite{VaismanF}), Banyaga (see \cite{Banyaga1,Banyaga2,Banyaga3}) 
and many others (see \cite{BK1,GL,LV,MMTP}). Vaisman pointed out in \cite[Section 1]{VaismanF} that locally conformal symplectic structures play an important role in Hamiltonian mechanics, 
generalizing the usual description of the phase space in terms of symplectic geometry.

In this paper we are mainly concerned with locally conformal symplectic structures of the first kind; this means that there exists an automorphism $U$ of $(\Phi,\omega)$ with $\omega(U)=1$; $U$ is 
then called an anti-Lee vector field. This is equivalent to the existence of a 1-form $\eta$, with $d\eta$ of rank $2n-2$, such that $\Phi=d\eta-\omega\wedge\eta$; here $2n=\dim M$. In order 
to obtain examples of locally conformal symplectic manifolds, one can take the product of a contact manifold and an interval; more generally, locally conformal symplectic structures of the first 
kind exist on the suspension (or mapping torus) of a strict contactomorphism\footnote{In this paper we will restrict to contact \emph{forms} rather than contact structures. Consequently, 
our morphisms will be \emph{strict} contactomorphisms, i.e.\ contactomorphisms preserving the contact form - see Definition \ref{contact_isomorphism}.} of a contact manifold. 
We refer to Section \ref{section:Preliminaries} for all the details. Banyaga proved a sort of converse to this result: in \cite[Theorem 2]{Banyaga1}, he showed that a compact manifold endowed with a locally conformal 
symplectic structure of the first kind fibres over the circle and that the fibre inherits a contact structure. This result, which shows an interplay between locally conformal symplectic geometry and contact 
geometry, is a non-metric version of a result of Ornea and Verbitski. In fact, in \cite{OV1}, they proved that compact Vaisman manifolds (see the definition below) are diffeomorphic to mapping tori with Sasakian fiber.

Our first result is the following (see Theorem \ref{global-descrip-lcs}):

\noindent {\bf Theorem A}
\textit{Let $M$ be a connected manifold endowed with a locally conformal symplectic structure of the first kind $(\Phi,\omega)$, where $\Phi=d\eta-\omega\wedge\eta$. Assume that the anti-Lee vector field $U$ 
is complete and that the foliation $\mathcal{F}=\{\omega=0\}$ has a compact leaf $L$. Denote by $\Psi$ the flow of $U$ and write $\eta_L=i^*\eta$, where $i\colon L\to M$ is the canonical inclusion. 
Then we have two possibilities:
\begin{enumerate}
\item $(\Phi,\omega)$ is a globally conformal symplectic structure on $M$ and $\Psi\colon L\times \bR\to M$ is an isomorphism of globally conformal symplectic manifolds;
\item there exist a real number $c>0$ and a strict contactomorphism $\phi\colon L \to L$ such that $\Psi$ induces an isomorphism between the locally conformal symplectic manifold of the first kind 
$L_{(\phi, c)}$ and $M$. In particular, $M$ is compact.
\end{enumerate}
In both cases, each leaf of $\mathcal{F}$ is of the form $\Psi_t(L)$ for some $t\in\bR$ and $\Psi_t\big|_L\colon L\to\Psi_t(L)$ is a strict contactomorphism.}

We refer to the discussion after the proof of Theorem \ref{global-descrip-lcs} (in Section \ref{lcs-can-fol}) for a comparison between Banyaga's result and the previous theorem.

Continuing our parallel between contact and locally conformal symplectic structures of the first kind, a theorem of Martinet (see \cite{Mar}) asserts that every oriented closed manifold of 
dimension $3$ has a contact form. Hence every orientable closed 3-manifold admits a contact structure. We provide a Martinet-type result for locally conformal symplectic structures of 
the first kind on 4-manifolds.

In fact, if $M$ is an oriented connected manifold of dimension $4$, $\omega$ is a closed $1$-form on $M$ without singularities and $L$ is a compact leaf of the foliation $\cF=\{\omega = 0\}$ then we obtain sufficient 
conditions for $M$ to admit a locally conformal symplectic structure of the first kind (see Corollary \ref{Martinet_type}).

Locally conformal symplectic structures also appear in the context of almost Hermitian geometry. An almost Hermitian structure on a manifold
$M$ of dimension $2n$ consists of a Riemannian metric $g$ and a compatible almost complex structure $J$, that is, an endomorphism $J\colon TM\to TM$ with $J^2=-\mathrm{Id}$, such
that $g(JX,JY)=g(X,Y)$ for every $X,Y\in\fX(M)$. If $(g,J)$ is an almost Hermitian structure on $M$, one can consider the following tensors:
\begin{itemize}
 \item a 2-form $\Phi$, the K\"ahler form, defined by $\Phi(X,Y)=g(JX,Y)$, $X,Y\in\fX(M)$;
 \item a 1-form $\omega$, the Lee form, defined by $\omega(X)=\frac{-1}{n-1}\delta\Phi(JX)$, where $\delta$ is the codifferential and $X\in\fX(M)$.
\end{itemize}

In \cite{GH}, Gray and Hervella classified almost Hermitian structures $(g,J)$ by studying the covariant derivative of the K\"ahler form $\Phi$ with respect to the Levi-Civita connection 
$\nabla$ of $g$. Of particular interest for us are the following classes:
\begin{itemize}
 \item the class of K\"ahler structures, for which $\nabla\Phi=0$. In this case, $g$ is said to be a K\"ahler metric. A K\"ahler manifold is a 
manifold endowed with a K\"ahler structure. The Lee form is zero on a K\"ahler manifold. We refer to \cite{Hu} for an introduction to K\"ahler geometry.
 \item the class of locally conformal K\"ahler (lcK) structures, for which $d\Phi=\omega\wedge\Phi$; in this case we call $g$ a locally conformal K\"ahler metric. A locally conformal K\"ahler manifold is a 
manifold endowed with a locally conformal K\"ahler structure. Locally conformal K\"ahler manifolds were introduced by Vaisman in \cite{VaismanA}. A remarkable example of locally conformal K\"ahler manifold is the Hopf surface. 
The study of locally conformal K\"ahler metrics on compact complex surfaces was undertaken in \cite{Belgun}. Homogeneous locally conformal K\"ahler and Vaisman structures have been studied in \cite{ACHK,GMO,HK,HK2}. 
The standard reference for locally conformal K\"ahler geometry is \cite{DO}; see also the recent survey \cite{OV}.
\end{itemize}

In the K\"ahler and locally conformal K\"ahler case, the almost complex structure $J$ is integrable, hence these are complex manifolds. One can interpret K\"ahler structure as a 
``degenerate'' case of locally conformal K\"ahler structures. Indeed, it turns out that the Lee form of a locally conformal K\"ahler structure is closed; if it is exact, then one can show that there 
is a K\"ahler metric in the conformal class of $g$. In such a case, one says that the structure is globally conformal K\"ahler. In general, if a manifold admits a genuine locally conformal K\"ahler structure, 
i.e.\ one for which the Lee form is not exact, then it admits an open cover such that the Lee form is exact on each open set, hence the locally conformal K\"ahler metric is locally conformal 
to a K\"ahler metric.

The non-metric version of K\"ahler manifolds are symplectic manifolds (see \cite{McDS}). A manifold $M^{2n}$ is symplectic if there exists a 2-form $\Phi$ such that
$d\Phi=0$ and $\Phi^n$ is a volume form. Clearly, the K\"ahler form of a K\"ahler structure is symplectic. It is well known that the existence of a K\"ahler metric on a compact manifold deeply 
influences its topology. In fact, suppose $M$ is a compact K\"ahler manifold of dimension $2n$ and let $\Phi$ be the K\"ahler form; then:
\begin{itemize}
 \item the odd Betti numbers of $M$ are even;
 \item the Lefschetz map $H^{p}(M)\to H^{2n-p}(M)$, $[\alpha]\mapsto [\alpha\wedge\Phi^{n-p}]$ is an isomorphism for $0\leq p\leq n$;
 \item $M$ is formal.
\end{itemize}

For a long time, the only known examples of symplectic manifolds came from K\"ahler (or even projective) geometry. In \cite{Th}, Thurston constructed the first example of a compact, symplectic
4-manifold which violates each of the three conditions given above. Since then, the problem of constructing compact symplectic manifolds without
K\"ahler structures has inspired beautiful Mathematics (see, for instance, \cite{FM,Go,McD,OT}).

Following this train of thought, the K\"ahler form and the Lee form of a locally conformal K\"ahler structure define a locally conformal symplectic structure. In contrast to the K\"ahler 
case, however, not much is known about the topology of compact locally conformal K\"ahler manifolds. In \cite[Conjecture 2.1]{DO} it was 
conjectured that a compact locally conformal K\"ahler manifold which satisfies the topological restrictions of a K\"ahler manifold admits a (global) K\"ahler metric. A stronger conjecture (see 
\cite{VaismanD}) is that a compact locally conformal K\"ahler manifold which is not globally conformal K\"ahler must have at least one odd degree Betti number which is odd. It was shown in 
\cite{BMO,KS,VaismanE} that a compact Vaisman manifold, i.e.\ a locally conformal K\"ahler manifold with non-zero parallel Lee form (see \cite{VaismanB}), has odd $b_1$. Hence the stronger conjecture holds for Vaisman 
structures. However, it does not hold in general: in \cite{OeTo}, Oeljeklaus and Toma 
constructed a locally conformal K\"ahler manifold of complex dimension 3 with $b_1=b_5=2$, $b_0=b_2=b_4=b_6=1$ and $b_3=0$.

In \cite{OV}, the authors proposed the following problem:
\begin{quote}
\emph{Is there a compact manifold with locally conformal symplectic structures but no locally conformal K\"ahler metric?}
\end{quote}

A first answer to this question was given by Bande and Kotschik: in \cite{Bande-Kotschick} they described a locally conformal symplectic structure on a 4-manifold of the form $M\times S^1$, which 
does not carry any complex structure, hence no locally conformal K\"ahler metric. 

The second goal of this paper is to give a different answer to the question above. In fact, we prove the following result (see Corollary \ref{Main:3}):

\noindent{\bf Theorem B}
\textit{There exists a compact, 4-dimensional nilmanifold, not diffeomorphic to the product of a compact 3-manifold and a circle, which has a locally conformal symplectic structure but no locally conformal 
K\"ahler metric.}

This result is contained in the preprint \cite{BMA}, of which this paper represents a substantial expansion. The example of Thurston that we mentioned above is also a nilmanifold. This makes the parallel between the 
symplectic and the locally conformal symplectic case particularly transparent. The product of the Heisenberg manifold and the real line admits a compact quotient, which turns out to be a nilmanifold which admits a Vaisman 
structure. It was conjectured by Ugarte in \cite{U} that this is basically the only possibility, i.e.\ that a compact nilmanifold endowed with a locally conformal K\"ahler structure with non-zero Lee form is a compact quotient 
of the Heisenberg group multiplied by $\bR$. This conjecture was proved by Sawai in \cite{Sawai}, assuming that the locally conformal K\"ahler structure has a left-invariant complex structure, but remains open in general. 
In Section \ref{sec:examples} we provide different examples of compact nilmanifolds with locally conformal symplectic structures.

On the other hand, in last years, special attention has been devoted to the study of symplectic Lie algebras (see \cite{BaCo,DaMe,LiMe,MeRe,Ovando}). In particular, in \cite{MeRe} (see also \cite{DaMe}), the authors introduce 
the notion of a symplectic double extension of a symplectic Lie algebra. In fact, if $\fs_1$ is a symplectic Lie algebra of dimension $2n-2$ then, in the presence of a derivation on $\fs_1$ and an element of $\fs_1$ which satisfy 
certain conditions, one may produce a new symplectic Lie algebra of dimension $2n$. In addition, in \cite{MeRe} (see also \cite{DaMe}), the authors prove a very interesting result: {\em every symplectic nilpotent Lie algebra of 
dimension $2n$ may be obtained as a sequence of $(n-1)$ symplectic double extensions from the abelian Lie algebra $\mathbb{R}^2$}.

We remark that a symplectic structure on a Lie algebra $\fs$ induces a left-invariant symplectic structure on a Lie group with Lie algebra $\fs$ and, conversely, a left-invariant symplectic structure on a Lie group induces 
a symplectic structure on its Lie algebra. Furthermore, symplectic structures may be considered as locally conformal symplectic structures with zero Lee $1$-form. In addition, all the locally conformal symplectic 
structures on the examples of compact nilmanifolds in Section \ref{sec:examples} come from left-invariant locally conformal symplectic structures (with non-zero Lee $1$-form) on connected simply connected Lie groups. 

Therefore, the following problem we tackle in this paper is the study of left-invariant locally conformal symplectic structures on Lie groups with non-zero Lee $1$-form and, more precisely, left-invariant locally conformal 
symplectic structures of the first kind. In this direction, our first result relates locally conformal symplectic Lie groups of the first kind with contact Lie groups 
(that is, Lie groups endowed with left-invariant contact structures). In fact, we prove that (see Section \ref{relacion-contact}):

\noindent{\bf Theorem C}
\textit{The extension of a contact Lie group $H$ by a contact representation of the abelian Lie group $\mathbb{R}$ on $H$ is a locally conformal symplectic Lie group of the first kind. Conversely, every connected simply 
connected locally conformal symplectic group of the first kind is an extension of a connected simply connected contact Lie group $H$ by a contact representation of $\mathbb{R}$ on $H$.}

We also introduce the definition of a locally conformal symplectic extension of a symplectic Lie group $S$ by a symplectic $2$-cocycle and a symplectic representation of $\mathbb{R}$ on $S$ and, then, we 
prove the following result (see Section \ref{rel_symplectic}):

\noindent{\bf Theorem D}
\textit{The locally conformal symplectic extension of a symplectic Lie group $S$ of dimension $2n$ is a locally conformal symplectic Lie group of the first kind with bi-invariant Lee vector field and dimension $2n+2$. 
Conversely, every connected simply connected locally conformal symplectic Lie group of the first kind with bi-invariant Lee vector field is the locally conformal symplectic extension of a connected simply connected symplectic 
Lie group.}

The last part of the paper is devoted to the study of locally conformal symplectic nilpotent Lie groups with non-zero Lee $1$-form. We completely describe the nature of these Lie groups in terms of locally conformal symplectic 
extensions and double symplectic extensions. In fact, we prove the following result (see Theorem \ref{structure_nilpotent}):

\noindent{\bf Theorem E}
\textit{Every connected simply connected locally conformal symplectic nilpotent Lie group of dimension $2n+2$ with non-zero Lee $1$-form is the locally conformal symplectic nilpotent extension of a connected simply connected 
symplectic nilpotent Lie group $S$ and, in turn, $S$ may be obtained as a sequence of $(n-1)$ double symplectic nilpotent extensions from the abelian Lie group $\mathbb{R}^2$.}

We show that all the compact locally conformal symplectic nilmanifolds in Section \ref{sec:examples} may be described using the previous result. On the other hand, in order to obtain the above results on locally conformal 
symplectic Lie groups, we discuss Lie algebras endowed with locally conformal symplectic structures.

This paper is organized as follows:
\begin{itemize}
 \item in Section \ref{section:Preliminaries} we recall the main definitions and known results about locally conformal symplectic geometry, Lie algebra cohomology, 
 multiplicative vector fields, central extensions of Lie algebras and groups and compact nilmanifolds;
 \item in Section \ref{sec:examples} we obtain some examples of compact locally conformal symplectic nilmanifolds of the first kind (symplectic or not) which do not admit Vaisman metrics; in particular, we present an example of a 
 compact, 4-dimensional nilmanifold, not diffeomorphic to the product of a compact 3-manifold and a circle, which has a locally conformal symplectic structure but no locally conformal 
 K\"ahler metric;
 \item in Section \ref{sec:foliation} we provide some general results about foliations of codimension 1, we describe the global structure of a connected locally conformal symplectic manifold of the first kind with a compact leaf 
 in its canonical foliation (see \thmref{global-descrip-lcs}) and we deduce some consequences (see Corollary \ref{Martinet_type});
 \item in Section \ref{sec:lcs_Lie_algebras} we study locally conformal symplectic structures on Lie algebras, with a particular emphasis on the nilpotent case. 
 \item in Section \ref{sec:lcs_Lie_groups} we consider left-invariant locally conformal symplectic structures on Lie groups emphasizing again the nilpotent case. We show how to recover the examples of Section 
\ref{sec:examples} from the general description.
\end{itemize}

%
%

\section{Preliminaries}\label{section:Preliminaries}

In this section we review the basics in locally conformal symplectic geometry as well as some definitions, 
constructions and results on Lie algebra cohomology, multiplicative vector fields, central extensions of Lie algebras and groups and compact nilmanifolds.

\subsection{Locally conformal symplectic structures}\label{LCS_Manifolds}

\begin{definition}
A \emph{locally conformal symplectic (lcs)} structure on a smooth manifold $M^{2n}$ ($n \geq 2$) consists of a 2-form $\Phi\in\Omega^2(M)$ 
and a closed 1-form $\omega\in\Omega^1(M)$, called the Lee form, such that $\Phi^n$ is a volume form on $M$ and 
\begin{equation}\label{eq:1}
 d\Phi=\omega\wedge\Phi.
\end{equation}
The lcs structure is \emph{globally conformal symplectic (gcs)} if $\omega$ is exact.
\end{definition}

Notice in particular that if $H^1(M;\bR)=0$, every lcs structure on $M$ is gcs. 
We require $n\geq 2$, since in dimension 2 a 2-form is automatically closed, hence lcs geometry in dimension 2 is nothing but symplectic geometry.

Here is an equivalent definition: a manifold $M^{2n}$ has a lcs structure if there exist a 2-form $\Phi\in\Omega^2(M)$ with 
$\Phi^n\neq 0$, an open cover $M=\cup_\alpha U_\alpha$ and smooth functions $\sigma_\alpha\colon U_\alpha\to \bR$ such that 
$\Phi_\alpha=e^{\sigma_\alpha}\Phi\big|_{U_\alpha}$ satisfies $d\Phi_\alpha=0$. The lcs structure is globally conformal 
symplectic if the domain of $\sigma_\alpha$ can be chosen to be all of $M$. In this case, the global conformal change $\Phi\mapsto \Phi'=e^\sigma\Phi$ produces a \emph{closed} and 
non-degenerate 2-form $\Phi'$, hence $(M,\Phi')$ is a symplectic manifold.

Since $\Phi$ is a non-degenerate 2-form, it provides an isomorphism $\fX(M)\to \Omega^1(M)$ given by
\[
X\mapsto \imath_X\Phi.
\]
We define $V\in\fX(M)$ by the condition $\imath_V\Phi=\omega$. $V$ is the \emph{Lee vector field} of the lcs structure. Clearly $\omega(V)=0$, since $\Phi$ is skew-symmetric.

As usual, a good approach to the study of a geometric structure is to consider its automorphisms. If $(\Phi,\omega)$ is a lcs structure on $M$, an \emph{infinitesimal automorphism} of $(\Phi,\omega)$ is a 
vector field $X\in\fX(M)$ such that $\Lie_X\Phi=0$. Since $n \geq 2$, $\Lie_X\omega =0.$ Infinitesimal automorphisms of $(\Phi,\omega)$ form a Lie subalgebra $\fX_\Phi(M)$ of $\fX(M)$. 
Moreover, if $X\in\fX_\Phi(M)$ then $\omega(X)$ is a constant function whenever $M$ is connected. Hence the \emph{Lee morphism} $\ell\colon \fX_\Phi(M)\to\bR$, $\ell(X)=\omega(X)$, is well defined. Viewing $\bR$ as an abelian
Lie algebra, $\ell$ becomes a Lie algebra morphism (see \cite{VaismanF} or \cite[Proposition 2]{Banyaga3}). Notice that the Lee vector field $V$ is in $\fX_\Phi(M)$ and that $\ell(V)=0$. 
Since the image of $\ell$ has dimension at most 1, the Lee morphism is either surjective or identically zero. The lcs structure $(\Phi,\omega)$ is said to be \emph{of the first kind} if the Lee 
morphism is surjective, \emph{of the second kind} otherwise (see \cite{VaismanF}). 

In this paper, we shall restrict ourselves to lcs structures of the first kind. Also, in general, when we say \emph{lcs structure} we explicitly exclude the possibility that the structure is 
actually globally conformal symplectic.

Let $(\Phi,\omega)$ be a lcs structure of the first kind on $M$. Let $U\in\fX_\Phi(M)$ be a vector field such that $\omega(U)=1$. Define $\eta\in\Omega^1(M)$ by the equation 
$\eta=-\imath_U\Phi$. If $U$ and $\eta$ are as above and $V$ is the Lee vector field then
\[
\eta(U) = 0, \; \; \eta(V) = 1, \; \; 
\Lie_U\eta=\Lie_U\omega=\Lie_V\eta=\Lie_V\omega=0. 
\]
In particular, $\imath_U d\eta=\imath_V d\eta=0$.

The next proposition gives an alternative characterization of lcs structures of the first kind in terms of the existence of a certain 1-form. For a proof, see \cite[Proposition 2.2]{VaismanF}.

\begin{proposition}\label{lcs_1_kind}
Let $M^{2n}$ be manifold endowed with a lcs structure of the first kind $(\Phi,\omega)$. Then there exists a 1-form $\eta\in\Omega^1(M)$ such that
\begin{equation}\label{eq:2}
\Phi= d\eta-\omega\wedge\eta, 
\end{equation}
$d\eta$ has rank $2n-2$ and $\omega \wedge \eta \wedge (d\eta)^{n-1}$ is a volume form.
Conversely, let $M^{2n}$ be a manifold endowed with two nowhere zero 1-forms $\omega$ and $\eta$, with $d\omega=0$, $\textrm{rank}(d\eta)<2n$ and such that $\omega\wedge\eta\wedge(d\eta)^{n-1}$ is a 
volume form. Then $M$ admits a lcs structure of the first kind. 
\end{proposition}

The conditions given in Proposition \ref{lcs_1_kind} imply that there exist two vector fields $U$ and $V$ on $M$ which are characterized by the following conditions
\begin{eqnarray*}
\omega(U) =1, & \eta(U) = 0, & i_Ud\eta = 0,\nonumber \\
\omega(V) =0, & \eta(V) = 1, & i_Vd\eta = 0.
\end{eqnarray*}

Here $V$ is the Lee vector field, $U$ is an \emph{anti-Lee vector field} and $\eta$ is an \emph{anti-Lee $1$-form}.

Note that different anti-Lee vector fields (and, thus, different anti-Lee $1$-forms) may exist for a lcs structure of the first kind. However, from now on, when we refer to a 
lcs structure of the first kind, we will assume that the anti-Lee vector field (and, therefore, the anti-Lee $1$-form) is fixed. Consequently, in what follows, a lcs structure of the first 
kind will be a couple of 1-forms which satisfy the conditions in Proposition \ref{lcs_1_kind}.

Lcs structures can be told apart according to another criterion, which we now review briefly. Let $M$ be a smooth manifold and let $\omega\in\Omega^1(M)$ be a closed 1-form. In 
\cite{GL}, a twisted de Rham 
differential $d_\omega$ was defined: given $\alpha\in\Omega^p(M)$, one sets
\[
d_\omega(\alpha)=d\alpha-\omega\wedge\alpha.
\]
One sees that $(d_\omega)^2=0$, hence the cohomology $H^\bullet_\omega(M)$ of the complex $(\Omega^\bullet(M),d_\omega)$ is defined. $H^\bullet_\omega(M)$ is called the \emph{Lichnerowicz} or \emph{Morse-Novikov} cohomology of $M$ 
relative to $\omega$. Notice that $d_\omega$ is not a derivation with respect to the algebra structure given by the wedge product on $\Omega^\bullet(M)$, i.e.\ it does not satisfy the Leibniz rule. 
If $\omega$ is an exact 1-form, then $(\Omega^\bullet(M),d_\omega)$ is homotopic to the standard de Rham complex $(\Omega^\bullet(M),d)$, hence 
$H^\bullet_\omega(M)\cong H^\bullet(M;\bR)$. But this is not the case if $\omega$ is not exact; for instance, when $M$ is connected, oriented and 
$n$-dimensional, $H^n_\omega(M)= 0$, see \cite[Page 429]{GL}. However, as pointed out in \cite{BK1}, the Euler characteristic of the Lichnerowicz cohomology equals the usual Euler characteristic. Also, 
$H^\bullet_\omega(M)$ is isomorphic to the cohomology of $M$ with values in the sheaf $\mathcal{F}_\omega$ 
of local functions $f$ on $M$ such that $d_\omega f=0$ (see \cite[Proposition 3.1]{VaismanC}).

Let $(\Phi,\omega)$ be a lcs structure on a manifold $M$. Equation \eqref{eq:1} above tells us that $\Phi$ is a 2-cocycle in $(\Omega^\bullet(M),d_\omega)$, hence it defines a cohomology class 
$[\Phi]_\omega\in H^2_\omega(M)$. The lcs structure $(\Phi,\omega)$ is called \emph{exact} if $[\Phi]_\omega=0\in H^2_\omega(M)$, \emph{non exact} otherwise.

Equation \eqref{eq:2} above shows that a lcs structure of the first kind is automatically exact. The converse, however, need not be true. Indeed, by \cite[Proposition 3]{Banyaga3}, out of an 
exact lcs structure $(\Phi,\omega)$ we get a vector field $X$ which satisfies $\cL_X\Phi=(\omega(X)-1)\Phi$ and we can not assure that $\omega(X)=1$. See also Example \ref{solv_ex_not_first_kind} 
below.

According to \cite[Corollary 1]{Banyaga1}, the study of non exact lcs structures and of their automorphisms on a manifold $M$ is strictly related to the study of symplectic structures on the 
minimum cover $\tilde{M}\to M$ on which the Lee form becomes exact and of their corresponding automorphisms. Examples of non exact lcs structures on the 4-dimensional solvmanifold constructed in 
\cite{ACFM} are given in \cite{Banyaga1}.

We describe now two methods to construct manifolds endowed with lcs structures of the first kind from contact manifolds. These examples play a very important role in this paper.
\begin{example}\label{ex:1}
Let $(L,\theta)$ be a contact manifold of dimension $2n-1$; hence $\theta\wedge(d\theta)^{n-1}$ is a volume form and the \emph{Reeb field} $R\in\fX(M)$ is uniquely determined by the 
conditions $\imath_Rd\theta=0$ and $\imath_R\theta=1$. Let $I\subset\bR$ be an open interval (we do not exclude the case $I=\bR$). Then $M=L\times I$ admits a gcs structure, which we describe 
using Proposition \ref{lcs_1_kind}. Let $\pi_i$ denote the projection from $M$ to the $i$-th factor, $i=1,2$. We have two 1-forms $\omega_\theta$ and $\eta_\theta$ on 
$M$,
\[
\omega_\theta=\pi_2^*dt \quad \mathrm{and} \quad \eta_\theta=\pi_1^*\theta,
\]
where $t$ is the standard coordinate on $\bR$. The anti-Lee vector field is $U_\theta=\left(0, \frac{\partial}{\partial t}\right)$ while the Lee vector field is 
$V_\theta=(R, 0)$. Clearly, $M$ is globally conformal symplectic to the standard symplectization of the contact manifold $L$.

In the compact case, one can start with a compact contact manifold $(L,\theta)$ and consider the product $M\coloneq L\times S^1$. The two 1-forms
\[
\omega_\theta=\pi_2^*\tau \quad \mathrm{and} \quad \eta_\theta=\pi_1^*\theta,
\]
where $\tau$ is the angular form on $S^1$, give a lcs structure of the first kind on $M$.
\end{example}

\begin{example}\label{ex:2}
More generally, let $(L,\theta)$ be a contact manifold and let $\phi\colon L\to L$ be a strict contactomorphism: this means that $\phi$ is a diffeomorphism and $\phi^*\theta=\theta$. Let $c$ be a 
positive real number and denote by $M$ the \emph{suspension} of $L$ by $\phi$ and $c > 0$, i.e.
\[
M= L \times_{(\phi, c)} \bR =\frac{L \times \bR }{\sim_{(\phi, c)}},
\]
where $(x,t)\sim_{(\phi, c)}(x',t')$ if and only if $x=\phi^k(x')$ and $t'=t+ck$, $k\in\bZ$. $M$ is also known as the \emph{mapping torus} $L_{(\phi, c)}$ of $L$ by $c$ and $\phi$; it 
can alternatively be described as the quotient space
\[
M=\frac{L \times [0,c]}{(\phi(x), 0)\sim (x, c)}.
\]
The standard gcs structure $(\omega_\theta,\eta_\theta)$ on $ L \times \bR $ induces a lcs structure of the first kind on $M$, which we denote by 
$(\omega_{(\theta,\phi,c)},\eta_{(\theta,\phi,c)})$. We also have a canonical projection $\pi\colon M\to S^1=\bR/c\bZ$ satisfying $\omega_{(\theta,\phi,c)} = 
\pi^*(\tau)$, where $\tau$ is the angular form on $S^1$. Furthermore, if $L$ is compact, so is $M$.
\end{example}

\subsection{Lie algebra cohomology}\label{Lie_alg_cohom}

Let $\fg$ be a real Lie algebra of dimension $n$ and let $W$ be a finite dimensional $\fg$-module. For $1\leq p\leq n$ we consider the space
\[
C^p(\fg;W)=\{f\colon\Lambda^p\fg \to W\};
\]
we also set $C^0(\fg;W)=W$. For $0\leq p\leq n$ we define a map $\d\colon C^p(\fg;W)\to C^{p+1}(\fg;W)$ by
\begin{align*}
 (\d f)(x_1,\ldots,x_{p+1})&=\sum_{i=1}^{p+1}(-1)^{i+1}x_i\cdot f(x_1,\ldots,\hat{x}_i,\ldots,x_{p+1})\\
 &+\sum_{i<j}(-1)^{i+j}f([x_i,x_j],x_1,\ldots,\hat{x}_i,\ldots,\hat{x}_j,\ldots,x_{p+1}),
\end{align*}
where $\cdot$ denotes the $\fg$-module structure on $W$ and $\hat{x}_i$ means that $x_i$ is omitted. We set 
\[
 Z^p(\fg;W)=\ker(\d\colon C^p(\fg;W)\to C^{p+1}(\fg;W)) \quad \mathrm{and} \quad B^p(\fg;W)=\mathrm{im}(\d\colon C^{p-1}(\fg;W)\to C^p(\fg;W)).
\]
Elements in $Z^p(\fg;W)$ (resp.\ $B^p(\fg;W)$) are called \emph{p-cocycles} (resp.\ \emph{p-coboundaries}). One has that $\d^2=0$, hence the degree $p$ cohomology of the complex 
$(C^*(\fg;W),d)$ can be defined as
\[
H^p(\fg;W)=\frac{Z^p(\fg;W)}{B^p(\fg;W)}. 
\]

\begin{example}
When $W=\bR$, the trivial $\fg$-module, then $(C^*(\fg;W),\d)$ is isomorphic to the Chevalley-Eilenberg complex $(\Lambda^\bullet\fg^*,d)$, where $\fg^*=\hom(\fg,\bR)$. In this case, $b_p(\fg)\coloneq\dim H^p(\fg;\bR)$ are 
the \emph{Betti numbers} of $\fg$.
\end{example}

\begin{example}\label{1_dim_rep}
Given a Lie algebra $\fg$, suppose that there exists a non-zero $\omega\in\fg^*$ with $d\omega=0$. Consider the 1-dimensional non-trivial $\fg$-module $W_\omega$ with $\fg$-representation given by
\begin{equation}\label{eq:representation}
 X\cdot w=-\omega(X)w, 
\end{equation}
where $w\in W_\omega$. Since $d\omega=0$, \eqref{eq:representation} is indeed a Lie algebra representation. In this peculiar situation, a computation shows that 
$\delta=d_\omega$, where $d_\omega(\alpha)=d\omega-\omega\wedge\alpha$. Thus the Lie algebra cohomology of $W_\omega$ coincides with the so-called \emph{Lichnerowicz} or \emph{Morse-Novikov} 
cohomology of $\fg$ (see \cite{Millionschikov})
\end{example}

Recall that, for a nilpotent Lie algebra $\fg$, $b_1(\fg)\geq 2$, hence we can always find a non-zero element $\omega\in\fg^*$ with $d\omega=0$. In this case, $W_\omega$ is a non-trivial 1-dimensional $\fg$-module. We need the 
following result of Dixmier:

\begin{theorem}[{\cite[Th\'eor\`eme 1]{Dixmier}}]\label{Dixmier}
Let $\fg$ be a nilpotent $n$-dimensional Lie algebra and let $W$ be a $\fg$-module such that every $\fg$-module contained in $W$ is non-trivial. Then $H^p(\fg;W)=0$ for $0\leq p\leq n$.
\end{theorem}

\begin{corollary}\label{cor:Dixmier}
 Suppose $\fg$ is an $n$-dimensional nilpotent Lie algebra and pick a non-zero $\omega\in\fg^*$ with $d\omega=0$. Then $H^p(\fg;W_\omega)=0$ for $0\leq p\leq n$.
\end{corollary}

\subsection{Multiplicative vector fields on Lie groups}\label{multiplicative}

In this section we review some definitions and constructions on multiplicative vector fields in a Lie group and semidirect product of Lie groups and algebras (for more information see, for instance, \cite{MaXu}).

Let $H$ be a Lie group with Lie algebra $\fh$ and let $\phi\colon \mathbb{R} \to \Aut (H)$ be a representation of the abelian Lie group $\mathbb{R}$ on $H$. $\phi$ is the flow of a multiplicative vector 
field $\cM\colon H \to TH$, that is, $\cM$ satisfies the condition 
\[
\cM(h h') = \cM(h) \cdot \cM(h'), \; \; \mbox{ for } h, h' \in H
\]
where $\cdot$ denotes the multiplication in the Lie group $TH$. In other words,
\[
\cM(h h') = (T_hr_{h'})(\cM(h)) + (T_{h'} \ell_{h})(\cM(h')),
\]
$r_{h'}\colon H \to H$ and $\ell_h\colon H \to H$ being right translation by $h'$ and left translation by $h$, respectively.

If $\fe$ is the identity element in $H$, $\cM(\fe) = 0$ and, in addition, $\cM$ induces a derivation $D\colon\fh \to\fh$ in the Lie algebra $\fh$ which is defined by
\begin{equation}\label{der-induced}
\lvec{DX} = [\lvec{X}, \cM],
\end{equation}
for $X \in\fh$, where $[\cdot, \cdot ]$ is the standard Lie bracket of vector fields in $H$. Here, we use the following notation: if $K$ is a Lie group with Lie algebra $\fk$ and $X \in \fk$ then $\lvec{X}$ is 
the left-invariant vector field on $K$ whose value at $\fe$ is $X$. Recall that a derivation of a Lie algebra $\fh$ is a linear map $D\colon \fh\to\fh$ such that
\[
D([X,Y])=[D(X),Y]+[X,D(Y)] \quad \forall X,Y\in\fh.
\]

The derivations of a Lie algebra $\fh$ form a Lie algebra, denoted $\mathrm{Der}(\fh)$. A derivation is a 2-cocycle in the complex $C^*(\fh,\fh)$, where the $\fh$-module 
structure of $\fh$ is given by the adjoint representation. A derivation $D$ is \emph{inner} if $D=\mathrm{ad}_X$ for some $X\in\fh$.

Using the representation $\phi$, we may consider the semidirect product Lie group $G = H \rtimes_\phi\mathbb{R}$, with multiplication defined by
\[
(h, t) (h', t') = (h\phi_t(h'), t + t'), \; \; \mbox{ for } (h, t), (h', t') \in G.
\]

The Lie algebra $\fg$ of $G$ is the semidirect product $\fg =\fh \rtimes_D \mathbb{R}$. This is simply $\fh\oplus\bR$ with bracket
\begin{equation}\label{extension_derivation}
[(X,a), (Y,b)] = (aD(Y) - bD(X) + [X, Y]_{\fg},0), \; \; \mbox{ for } (X,a), (Y,b) \in \fh\oplus\bR.
\end{equation}

Conversely, let $D\colon\fh\to\fh$ be a derivation of $\fh$ and let $H$ be the connected simply connected Lie group with Lie algebra $\fh$. Then, there exists a unique 
multiplicative vector field $\cM$ on $H$ whose flow induces a representation
\[
\phi\colon \mathbb{R} \to \Aut(H)
\]
of the abelian Lie group $\mathbb{R}$ on $H$ and
\[
\lvec{DX} = [\lvec{X},\cM],
\]
for $X \in \fh$. Hence we can consider the Lie group $G = H \rtimes_\phi \mathbb{R}$ whose Lie algebra $\fg$ is $\fh\rtimes_D\bR$. 

The semidirect product of a group $H$ and $\bR$ with an automorphism 
$\phi\colon \bR\to\Aut(H)$ sits in a short exact sequence $1\to H\to H\rtimes_\phi \bR\to \bR\to 1$ and $H$ is a normal subgroup of $H\rtimes_\phi\bR$.

\subsection{Central extensions of Lie algebras and groups by $\mathbb{R}$}\label{central-ext-alg-group}

In this section we review some constructions on central extensions of Lie algebras (resp.\ groups) by the abelian Lie algebra (resp.\ group) $\mathbb{R}$ (see, for instance, \cite{MaMiOrPeRa,TuWi}).

Let $S$ be a Lie group with Lie algebra $\fs$. A 2-cocycle on $S$ with values in $\mathbb{R}$ is a smooth map $\varphi\colon S \times S \to \mathbb{R}$ such that
\[
\varphi(s, s') - \varphi(s, s's'') + \varphi(ss', s'') - \varphi(s', s'') = 0, \quad \mathrm{for} \ s, s', s'' \in S.
\]

If $\varphi$ is a $2$-cocycle on $S$ with values in $\mathbb{R}$, we can consider the central extension of $S$ by $\varphi$, denoted $\bR\cext_{\varphi}S$. As a set, this is $\bR\times S$, with Lie group structure given
\[
(u,s)(u',s') = (u + u' + \varphi(s, s'),ss') \quad \mathrm{for} \ (u,s), (u',s') \in \bR\times S.
\]
On the other hand, the $2$-cocycle $\varphi$ induces a $2$-cocycle $\sigma$ on $\fs$ (compare with Section \ref{Lie_alg_cohom}), defined by
\begin{equation}\label{sigma}
\sigma(X,Y) = \displaystyle \frac{d}{dt}\Big|_{t=0} \frac{d}{ds}\Big|_{s=0} (\varphi(\exp(tX),\exp(sY)) - \varphi(\exp(sY),\exp(tX))),
\end{equation}
where $\exp\colon\fs \to S$ is the exponential map associated with the Lie group $S$.

Moreover, the Lie algebra of the central extension $\bR\cext_{\varphi}S$ is the central extension $\bR\cext_\sigma\fs$ of $\fs$ by $\sigma$; this is just $\bR\oplus\fs$ with bracket
\begin{equation}\label{extension_central}
[(u,X),(u',X')] = (\sigma(X,X'),[X,X']), \; \; \mbox{ for } (u,X), (u',X') \in \bR\oplus\fs. 
\end{equation}

Conversely, let $S$ be a connected simply connected Lie group with Lie algebra $\fs$ and let $\sigma\in\Lambda^2\fs^*$ be a $2$-cocycle. Then, one may find a $2$-cocycle $\varphi\colon S \times S \to \mathbb{R}$ on $S$ with 
values in $\mathbb{R}$ such that $\varphi$ and $\sigma$ are related by \eqref{sigma}. In addition, the central extension 
$\bR \cext_{\varphi} S$ of $S$ by $\varphi$ is a Lie group with Lie algebra the central extension $\bR\cext_{\sigma}\fs$ of $\fs$ by $\sigma$ (for more details, see \cite{TuWi}). A central extension of a group $S$ by $\bR$, 
given by a 2-cocycle $\varphi\colon S\times S\to\bR$ sits in a short exact sequence $1\to \bR\to \bR\cext_\varphi S\to S\to 1$ and $\bR$ is a normal subgroup of $\bR\cext_\varphi S$.

\subsection{Compact nilmanifolds}
Let $G$ be a connected, simply connected nilpotent Lie group and let $\fg$ be its Lie algebra. It is interesting to know when $G$ contains a discrete, co-compact subgroup (i.e.\ a lattice).

\begin{theorem}[Mal'tsev, \cite{Mal'tsev}]\label{Maltsev}
$G$ contains a lattice if and only if there exists a basis of $\fg$ such that the structure constants of $\fg$ with respect to this basis are 
rational numbers.
\end{theorem}

Assume that $G$ contains a lattice $\Gamma$ and let $N=\Gamma\backslash G$ be the corresponding nilmanifold. One identifies elements of $\Lambda^\bullet\fg^*$ with left-invariant 
forms on $G$, which descend to $N$.
\begin{theorem}[Nomizu, \cite{Nom}]\label{Nomizu}
 The natural inclusion $(\Lambda^\bullet\fg^*,d)\hookrightarrow(\Omega^\bullet(N),d)$ induces an isomorphism on cohomology.
\end{theorem}

\subsection{Notation for Lie algebras}\label{Notation_Lie_algebras}

We will adopt the following notation for Lie algebras, best explained by an example. $\fg=(0,0,0,12)$ means that $\fg^*$ has a basis $\{e^1,\ldots,e^4\}$ such that $de^1=de^2=de^3=0$ and $de^4=e^{12}$, where $d$ is the 
Chevalley-Eilenberg differential and $e^{12}\coloneq e^1\wedge e^2$.

%
%

\section{Some examples of compact lcs manifolds of the first kind}\label{sec:examples}

In this section we use nilmanifolds to construct examples of lcs manifolds of the first kind. In dimension 4, we construct a compact nilmanifold 
endowed with a lcs structure of the first kind which does not carry any lcK metric and, furthermore, is not the product of a 3-manifold and a circle, thus proving \thmref{Main:3}. 
In higher dimension we construct compact nilmanifolds with lcs structures of the first kind, symplectic or not, which do not carry any lcK metric (with left-invariant complex structure) and any Vaisman metric. 
We summarize our examples in Table \ref{table:0} (the superscript \textsuperscript{$\dagger$} means \emph{left-invariant} complex or lcK structure).

\begin{table}[h!]
\caption{Summary of the examples}\label{table:0}
\begin{center}
{\tabulinesep=1.2mm
\begin{tabu}{lcccccccc}
\toprule[1.5pt]
& dimension & symplectic & complex\textsuperscript{$\dagger$} & complex & lcs & lcK\textsuperscript{$\dagger$} & lcK & Vaisman\\
\specialrule{1pt}{0pt}{0pt}
\thmref{first:example} & $4$ & \checkmark & $\times$ & $\times$ & \checkmark & $\times$ & $\times$ & $\times$\\
\specialrule{1pt}{0pt}{0pt}
\thmref{dim:6_first} & $6$ & \checkmark & \checkmark & \checkmark &\checkmark & $\times$ & ? & $\times$\\  
\specialrule{1pt}{0pt}{0pt}
\thmref{dim:6_second} & $6$ & $\times$ & $\times$ & ? & \checkmark & $\times$ & ? & $\times$\\  
\specialrule{1pt}{0pt}{0pt}
\thmref{third:example} & $2n,n\geq 4$ & $\times$ & \checkmark & \checkmark & \checkmark & $\times$ & ? & $\times$\\
\bottomrule[1pt]
\end{tabu}}
\end{center}
\end{table}

\subsection{Dimension 4}\label{Dimension:4}

In this section we present an example of a $4$-dimensional compact nilmanifold endowed with a lcs structure of the first kind, without complex structures, hence no lcK 
structures. We use the general construction in Example \ref{ex:2}, i.e.\ we consider the mapping torus of a compact contact manifold by a strict contactomorphism. The contact manifold 
is a compact quotient of the Heisenberg group of dimension 3. This is the connected simply connected nilpotent Lie group
\[
H=\left\{
\begin{pmatrix} 
1 & y & z \\
0 & 1 & x\\
0 & 0 & 1 
\end{pmatrix} \ | \ x,y,z\in\bR\right\}. 
\]
 We denote a point of $H$ by $(x,y,z)$. In these global coordinates, a basis of left-invariant 1-forms is given by
\[
\alpha=d x,\quad \beta=d y \quad \mathrm{and} \quad \theta=d z-yd x.
\]
It is clear that $\theta$ is a contact 1-form on $H$. We look for a strict contactomorphism $\phi$ of $(H,\theta)$. 
The condition 
\[
d(z\circ\phi) - dz = (y\circ\phi)d(x \circ\phi) - y dx
\]
must hold and if we assume that
\[
z\circ\phi = z + f(y), \; \; y \circ\phi = y, \; \; x\circ\phi = x + g(y),
\]
it follows that
\[
\displaystyle \frac{d f}{dy} = y \frac{d g}{dy}.
\]
In particular, if $t \in \bR$ we have that
\[
\phi_t\colon H\to H, \; \; \; \phi_t(x, y, z) = (x + ty, y, z + t y^2/2)
\]
is a strict contactomorphism. In fact, it is easy to prove that
\begin{lemma}\label{contact_Lie_iso} 
In the above situation, $\{\phi_t\}_{t\in \mathbb{R}}$ is a $1$-parameter group of strict contactomorphisms for $(H, \theta)$ and each $\phi_t$ is a Lie group isomorphism.
\end{lemma}

Now, denote by $\phi$ the strict contactomorphism $\phi_1$. We want to find a lattice $\Gamma\subset H$ such that $\phi(\Gamma) = \Gamma$. It is sufficient to take
\[
\Gamma = \{(m,2n,2p) \in H \ | \ m, n, p \in \mathbb{Z} \}.
\]
Indeed, we have that $\phi(\Gamma)=\Gamma$, hence
\begin{equation}\label{inv-discrete-sub}
\phi_r(\Gamma) = \Gamma, \; \; \; \mbox{ for every } r\in \bZ.
\end{equation}
This implies that $\phi$ induces a diffeomorphism $\bphi\colon L \to L$, where $L$ is the compact nilmanifold $L = \Gamma\backslash H$.

On the other hand, $\theta$ induces a contact 1-form $\btheta$ on $L$ and it is clear that $\bphi\colon L \to L$ is a strict contactomorphism of the contact manifold $(L, 
\btheta)$. Thus, the mapping torus $L_{(\bphi, 1)}$ of $L$ by the couple $(\bphi, 1)$ 
\[
L_{(\bphi, 1)} = \frac{L \times \mathbb{R}}{\sim_{(\bphi, 1)}}
\]
is a 4-dimensional compact lcs manifold of the first kind.

Next, we present a description of $L_{(\bphi, 1)}$ as a compact nilmanifold. In fact, using Lemma \ref{contact_Lie_iso}, we deduce that
\[
\phi\colon \mathbb{R} \to \mathrm{Aut}(H), \quad t \mapsto \phi_t
\]
is a representation of the abelian Lie group $\mathbb{R}$ on $H$. Therefore, we can consider the semidirect product $G = H \rtimes_{\phi} \mathbb{R}$ with multiplication given by
\[
((x, y, z),t)\cdot((x', y', z'),t') = ((x+x'+ty', y+y', z+z'+t(y')^2/2 + yx' + tyy'),t+t').
\]
A basis of left-invariant vector fields on $G$ is $\{U,V,A,B\}$, with
\[
U = \frac{\partial}{\partial t}, \; V = \frac{\partial}{\partial z}, \; A = \frac{\partial}{\partial x} + y \frac{\partial}{\partial z}, \; B = \frac{\partial}{\partial y} + t 
\frac{\partial}{\partial x} + ty \frac{\partial}{\partial z}, 
\]
and we have that
\[
[A,B] = -V, \; \; [U,B] = A,
\]
the rest of the basic Lie brackets being zero. Thus, $G$ is a connected simply connected nilpotent Lie group. The dual basis of left-invariant 1-forms on $G$ is $\{\omega,\eta, \alpha,\beta\}$, with
\[
\omega = dt, \; \; \eta = dz - ydx, \; \; \alpha = dx - t dy, \; \; \beta = dy,
\]
and
\begin{equation}\label{diff-left-inv}
d\omega = d\beta = 0, \; \; d\alpha = \beta\wedge\omega, \; \; d\eta = \alpha \wedge \beta.
\end{equation}
From \eqref{inv-discrete-sub}, it follows that the lattice $\Gamma$ of $H$ is invariant under the restriction to $\mathbb{Z}$ of the representation $\phi$. This implies that
$\Xi=\Gamma\rtimes_\phi\bZ$ is a lattice in $G$ and $M = \Xi\backslash G$ is a compact nilmanifold. 

Now, it is easy to prove that the map
\[
M \to L_{(\bphi,1)}, \; \; \; [((x, y, z), t)] \to [([(x,y, z)], t)]
\]
is a diffeomorphism and, therefore, $L_{(\bphi, 1)}$ is a compact nilmanifold.

\begin{remark}
Note that, under the identification between $L_{(\bphi, 1)}$ and $M$, the lcs structure of the first kind on $L_{(\bphi, 1)}$ is just the couple $(\bomega, \bar{\eta})$, where $\bomega$ and 
$\bar{\eta}$ are the 1-forms on $M$ induced by the left-invariant 1-forms $\omega$ and $\eta$, respectively, on $G$.
\end{remark}

Moreover, we may prove the following result.
\begin{lemma}\label{lemma:1}
$L_{(\bphi, 1)}$ is a compact nilmanifold of dimension $4$ which admits a lcs structure of the first kind and $b_1(L_{(\bphi, 1)}) = 2$. 
\end{lemma}
\begin{proof}
Let $H^1(L_{(\bphi, 1)};\bR)$ be the first de Rham cohomology group of $L_{(\bphi, 1)}$. Then, 
using \eqref{diff-left-inv} and Theorem \ref{Nomizu}, we deduce that
\[
H^1(L_{(\bphi, 1)};\bR) = \langle \bbeta, \bomega \rangle,
\]
where $\bbeta$ is the $1$-form on $M \simeq L_{(\bphi, 1)}$ induced by the left-invariant 1-form $\beta$ on $G$.
\end{proof}

\begin{theorem}\label{first:example}
 $L_{(\bphi, 1)}$ is a compact nilmanifold of dimension $4$ which admits a lcs structure of the first kind; however, $L_{(\bphi, 1)}$ does not carry any lcK metric.
\end{theorem}
\begin{proof}
Assume that $L_{(\bphi, 1)}$ carries a lcK metric. Then $L_{(\bphi, 1)}$ is a compact complex surface. By the Kodaira-Enriques classification, $L_{(\bphi, 1)}$ also carries a K\"ahler metric, since 
its first Betti number is even by Lemma \ref{lemma:1} (also, by \cite[IV. Theorem 3.1]{BHPVdV}, a complex surface with even first Betti number admits a K\"ahler metric). 
Now $L_{(\bphi, 1)}$ is a nilmanifold by Lemma \ref{lemma:1}. It is well known (see \cite{BG,Has}) that a compact K\"ahler nilmanifold is diffeomorphic to a torus. Hence $L_{(\bphi, 1)}$ should be diffeomorphic to the 
4-dimensional torus $T^4$, but this is absurd, since $b_1(T^4)=4$ and $b_1(L_{(\bphi, 1)})=2$.
\end{proof}

We also have the following result:
\begin{proposition}\label{main:prop2}
 The nilmanifold $L_{(\bphi, 1)}$ is not the product of a compact 3-dimensional manifold and a circle.
\end{proposition}
\begin{proof}
By Lemma \ref{lemma:1}, we see that $b_1(L_{(\bphi, 1)})=2$. Assume $L_{(\bphi, 1)}$ is a product $M\times S^1$, where $M$ is a compact 3-dimensional manifold. Since $L_{(\bphi, 1)}$ is a compact 
nilmanifold, $\pi_1(L_{(\bphi, 1)})$ is nilpotent and torsion-free. Then $\pi_1(M)\subset\pi_1(L_{(\bphi, 1)})$ is also nilpotent and torsion-free. In \cite{Mal'tsev}, Mal'tsev showed 
that for such a group $\pi_1(M)$ there exists a real nilpotent Lie group $K$ such that $P=\pi_1(M)\backslash K$ is a nilmanifold. It is well known that a compact nilmanifold has 
first Betti number at least 2, hence $b_1(P)\geq 2$. Now $P$ is an aspherical manifold; in this case, $H^*(P;\bZ)$ is isomorphic to the group cohomology $H^*(\pi_1(M);\bZ)$ (see for instance \cite[page 40]{Br}). Hence $b_1(M) \geq 2$. Now $M$ is also an aspherical space with fundamental group $\pi_1(M)$, hence, 
by the same token, $b_1(\pi_1(M))\geq 2$. However, by the K\"unneth formula applied to $L_{(\bphi, 1)}=M\times S^1$, we get $b_1(M)=1$. Alternatively, we could apply \cite[Theorem 3]{FHT} directly to 
our manifold $M$.
\end{proof}

Now, from \thmref{first:example} and Proposition \ref{main:prop2}, we deduce the following result.

\begin{corollary}\label{Main:3}
There exists a compact, 4-dimensional nilmanifold, not diffeomorphic to the product of a compact 3-manifold and a circle, which has a locally conformal symplectic structure but no locally conformal 
K\"ahler metric.
\end{corollary}

\begin{remark}
 The manifold $L_{(\bphi, 1)}$ also admits a symplectic structure $\sigma$, coming from a left-invariant symplectic structure on $G$. In terms of the basis of $\fg^*$ 
given above, $\sigma=\omega\wedge\alpha+\eta\wedge\beta$. $L_{(\bphi, 1)}$ can also be endowed with a left-invariant, non-integrable almost complex structure, given in terms of the basis 
$\{U,V,A,B\}$ of $\fg$, by $J(U)=A$, $J(V)=B$ and a left-invariant metric $g$ which makes such basis orthonormal. By the Gray-Hervella classification of almost Hermitian manifolds in dimension 4 (see Table II in \cite{GH}) 
there is a line $\mathcal{W}_2$ corresponding to the almost K\"ahler case and a line $\mathcal{W}_4$ corresponding to the complex case, intersecting in the origin, which is the K\"ahler case.
 With the almost Hermitian structure $(g,J,\sigma)$, $L_{(\bphi, 1)}$ lies 
on the line $\mathcal{W}_2$.
\end{remark}

As we remarked in the introduction, Bande and Kotschick (\cite{Bande-Kotschick}) gave a method to construct examples of compact manifolds with lcs structures of the first kind which carry no 
lcK metrics; we describe it briefly. Suppose that $M$ is an oriented compact manifold of dimension $3$. By a result of Martinet, see \cite{Mar}, $M$ admits a contact structure and, then, the 
product manifold $M\times S^1$ is locally conformal symplectic of the first kind (see Example \ref{ex:1}). On the other hand, using some results on compact complex surfaces, the authors show that $M$ may be chosen in 
such a way that the product manifold $M \times S^1$ admits no complex structures (this happens, for instance, if $M$ is hyperbolic) and, in particular, no lcK structures. Moreover, by the results of 
\cite{Friedl_Vidussi}, only for special choices of $M$ (those who fiber over a circle) can $M\times S^1$ have a symplectic structure. Therefore, for suitably chosen $M$, one obtains examples of locally conformal symplectic 
structures on $M\times S^1$, with no symplectic structure. We can then ask:

\begin{quote}
{\em Is there a compact, almost Hermitian lcs 4-manifold, not the product of a 3-manifold and a circle, which admits no complex and no symplectic structure?}
\end{quote}

Concerning this question, we remark that certain Inoue surfaces, discovered by Belgun (see \cite[Theorem 7]{Belgun}), do not carry any lcK metric compatible with the \emph{fixed} 
complex structure. One may, however, fix the underlying smooth manifold 
and vary the complex structure. For example, the smooth manifolds underlying Inoue surfaces are all diffeomorphic to a certain solvmanifold, quotient of the completely solvable Lie group $G$ whose Lie algebra is 
$(0,12,-13,23)$. For particular choices of the complex structure, such manifolds \emph{do} carry lcK metrics.

\subsection{Dimension 6}\label{Dimension:6}

In this section, we describe two 6-dimensional examples of lcs nilmanifolds which were announced in lines 2 and 3 of Table \ref{table:0}.

We start by recalling a result of Sawai (see \cite[Main Theorem]{Sawai}) which completely characterizes lcK nilmanifolds with left-invariant complex structure. For a description of the Lie algebra of the 
Heisenberg group, we refer to the proof of Lemma \ref{Betti_Heisenberg} below.

\begin{theorem}\label{Sawai}
Let $(M^{2n},J)$ be a non-toral compact nilmanifold with a left-invariant complex structure $J$. If $(M,J)$ carries a locally conformal K\"ahler metric, then it is biholomorphic to a 
quotient of $(\mathrm{Heis}_{2n-1}\times\bR,J_0)$, where $\mathrm{Heis}_{2n-1}$ is the Heisenberg group of dimension $2n-1$.
\end{theorem}

We consider the 5-dimensional Lie group 
\[
H=\left\{
\begin{pmatrix} 
1 & y & t & w \\
0 & 1 & x & z\\
0 & 0 & 1 & x\\
0 & 0 & 0 & 1\\
\end{pmatrix} \ | \ x,y,z,t,w\in\bR\right\},
\]
which is diffeomorphic to $\bR^5$ as a manifold; we denote a point of $H$ by $(x,y,z,t,w)$. 

For the first example, we choose the following basis of left-invariant 1-forms:
\[
\alpha=d x,\quad \beta=d y,\quad \gamma= xdx-dz,\quad \delta=d t-ydx \quad \mathrm{and} \quad \eta=d w-yd z+(xy-t)dx
\]
with differentials
\[
d\alpha=0,\quad d\beta=0,\quad d\gamma=0,\quad d\delta=\alpha\wedge\beta \quad \mathrm{and} \quad 
d\eta=\alpha\wedge\delta+\beta\wedge\gamma.
\]
In particular, $(H,\eta)$ is a contact manifold. The Lie algebra of $H$ is $\fh=(0,0,0,12,14+23)$. $\fh$ is nilpotent and isomorphic to $L_{5,3}$ (notation from 
\cite{BM}). Hence $H$ is a connected and simply connected nilpotent Lie group.

The subgroup $\Gamma=\{(m,n,p,q,r)\in H \ | \ m,n,p,q,r\in\bZ\}\subset H$ is a lattice. Since 
$\eta$ is left-invariant, it descends to a contact form $\bar{\eta}$ on the compact nilmanifold $L=\Gamma\backslash H$. By Example \ref{ex:1}, the product $M= L\times S^1$ has a lcs structure of 
the first kind $(\bomega,\bar{\eta})$, where $\bomega$ is the angular form on $S^1$; here we are implicitly identifying $\bar{\eta}\in\Omega^1(L)$ with $\pi_1^*\bar{\eta}\in\Omega^1(M)$ and 
$\bomega\in\Omega^1(S^1)$ with $\pi_2^*\bomega\in\Omega^1(M)$.

\begin{theorem}\label{dim:6_first}
$M$ is a 6-dimensional lcs, complex and symplectic nilmanifold which admits no lcK structures (with left-invariant complex structures). Furthermore, $M$ carries no Vaisman metric.
\end{theorem}

\begin{proof}
$M$ is a nilmanifold, being the quotient of the connected, simply connected nilpotent Lie group $H\times\bR$ by the lattice $\Gamma\times\bZ$. The Lie algebra of $H\times\bR$ is 
$\fh\oplus\bR$, isomorphic to the Lie algebra $\fh_9$ of \cite{CFGU}; $\fh_9$ admits a complex structure by \cite[Theorem 3.3]{Sal} (compare also Table \ref{table:1} below). In particular, $M$ admits 
a left-invariant complex structure. To show that $M$ is symplectic, it is enough to find a left-invariant symplectic form on $H\times\bR$. Consider 
$\rho\coloneq \alpha\wedge\eta+\delta\wedge\gamma+\beta\wedge\omega\in\Omega^2(H\times\bR)$, with $\omega = dt$ the canonical $1$-form on $\mathbb{R}$. $\rho$ is left-invariant, closed and non-degenerate, 
hence descends to a symplectic form on $M$. By Theorem \ref{Nomizu}, $b_1(L)=b_1(\fh)=3$, hence $b_1(M)=4$ by the K\"unneth formula. By Lemma \ref{Betti_Heisenberg}, a quotient of 
$\mathrm{Heis}_5\times\bR$ has $b_1=5$, hence $M$ is not a quotient of $\mathrm{Heis}_5\times\bR$. By Theorem \ref{Sawai}, $M$ does not carry any lcK metric (with left-invariant complex 
structure). If $M$ carried a Vaisman metric, then $b_1(M)$ should be odd by \cite[Corollary 3.6]{VaismanE}.
\end{proof}

\begin{lemma}\label{Betti_Heisenberg}
 A $2n$-dimensional nilmanifold $N$, quotient of $\mathrm{Heis}_{2n-1}\times\bR$, has $b_1(N)=2n-1$.
\end{lemma}

\begin{proof}
The Lie algebra $\mathfrak{heis}_{2n-1}$ of $\mathrm{Heis}_{2n-1}$ is spanned by vectors $\{X_1,Y_1,\ldots,X_{n-1},Y_{n-1},Z\}$ with brackets $[X_i,Y_i]=Z$ for every 
$i=1,\ldots,n-1$. If $\{x_1,y_1,\ldots,x_{n-1},y_{n-1},z\}$ is the dual basis of $\mathfrak{heis}_{2n-1}^*$, then the Chevalley-Eilenberg complex of the Lie algebra 
$\mathfrak{heis}_{2n-1}^*\oplus\bR$ has differentials
\[
d x_i=d y_i=0 \ \forall \ i=1,\ldots,n-1, \quad d w=0 \quad \mathrm{and} \quad d z=-\sum_{i=1}^{n-1}x_i\wedge y_i,
\]
where $w$ generates the $\bR$-factor. This means that $\mathfrak{heis}_{2n-1}\oplus\bR$ has $b_1=2n-1$. By Theorem \ref{Nomizu}, the same is true for every compact quotient of 
$\mathrm{Heis}_{2n-1}\times\bR$. 
\end{proof}

For the second example, we take
\[
\alpha=d x, \quad \beta=d y, \quad \gamma= dz-xdx, \quad \delta=d t-ydx \quad \mathrm{and} \quad \theta=d w-yd z+(xy-t)dx
\]
as a basis of left-invariant 1-forms on $H$; hence,
\[
d\alpha=0,\quad d\beta=0,\quad d\gamma=0,\quad d\delta=\alpha\wedge\beta \quad \mathrm{and} \quad 
d\theta=\alpha\wedge\delta-\beta\wedge\gamma.
\]
As before, $\theta$ is a contact form on $H$. A diffeomorphism $\phi\colon H\to H$ will be a strict 
contactomorphism provided
\begin{equation}\label{strict_contactomorphism}
 d(w\circ\phi)-dw=(y\circ\phi)d(z\circ\phi)-ydz-((x\circ\phi)(y\circ\phi)-(t\circ\phi))d(x\circ\phi)+(xy-t)dx.
\end{equation}
If we assume $y\circ\phi=y+f(x)$, $z\circ\phi=z+g(x)+h(y)$, $t\circ\phi=t+i(x)+j(y)+k(z)$ and $w\circ\phi=w+\ell(x)+m(y)+n(x,y)+o(x,z)$,
then the functions
\begin{align*}
 f(x)  &= x   &   g(x)  & =\frac{x}{2} & h(y) & = y & i(x) & = \frac{x}{6} & j(x) & = \frac{y}{2}\\
 \ell(x) &= \frac{x^2-x^3}{3} & m(y) &= \frac{y^2}{2} & n(x,y) &=xy & o(x,z) &=xz
\end{align*}
verify \eqref{strict_contactomorphism} and $\phi_s\colon H\to H$, mapping $(x,y,z,t,w)$ to
\[
 \left(x,y+sx,z+sy+\frac{1}{2}s^2x,t+sz+\frac{1}{2}s^2y+\frac{1}{6}s^3x,w+sxz+\frac{1}{2}sy^2-\frac{1}{3}sx^3+s^2xy+\frac{1}{3}s^3x^2\right)
\]
is a strict contactomorphism for every $s \in \mathbb{R}$. More precisely:

\begin{lemma}\label{contact-Lie-iso:dim_6}
In the above situation, $\{\phi_s\}_{s\in \bR}$ is a $1$-parameter group of strict contactomorphisms of $(H,\theta)$ and each $\phi_s$ is a Lie group isomorphism.
\end{lemma}

The subgroup
\[
 \Gamma=\{(6m,2n,p,q,r)\in H \ | \ m,n,p,q,r\in\bZ\}\subset H
\]
is a lattice and we have $\phi_s(\Gamma)=\Gamma$, for $s \in \bZ$. Hence $L\coloneq\Gamma\backslash H$ is a compact nilmanifold, $\theta$ induces a contact form $\btheta$ on $L$ 
and $\phi_1$ descends to a strict contactomorphism $\bphi\colon L\to L$. By Example \ref{ex:2}, the mapping torus $L_{(\bphi,1)}$ of $L$ by the pair $(\bphi, 1)$,
\[
 L_{(\bphi, 1)}=\frac{ L \times \bR}{\sim_{(\bphi, 1)}}
\]
is a 6-dimensional compact lcs manifold of the first kind.

We show that $L_{(\bphi, 1)}$ has the structure of a compact nilmanifold. In fact, by Lemma 
\ref{contact-Lie-iso:dim_6}, we obtain a representation $\phi\colon \bR\to\Aut(H)$, $s\mapsto \phi_s$, hence we can form the semidirect product $G= H \rtimes_{\phi}\bR$, whose 
group structure is given by
\[
 \begin{pmatrix}
 x \\ y\\ z\\ t\\ w\\ s
\end{pmatrix}\cdot
 \begin{pmatrix}
 x' \\ y'\\ z'\\ t'\\ w'\\ s'
\end{pmatrix}=
 \begin{pmatrix}
 x+x' \\ y+y'+sx'\\ z+z'+xx'+sy'+\frac{1}{2}s^2x'\\ t+t'+yx'+sz'+\frac{1}{2}s^2y'+\frac{1}{6}s^3x'\\ w+w'+tx'+yz'+s(x'z'+yy'+\frac{1}{2}(y')^2-\frac{1}{3}(x')^3)+s^2(x'y'+\frac{1}{2}yx')+\frac{1}{3}s^3(x')^2\\ s+s'
\end{pmatrix}
\]

A basis for left-invariant 1-forms on $G$ is given by
\begin{itemize}
 \item $\omega=d s$;
 \item $\alpha=d x$;
 \item $\beta=d y-sdx$;
 \item $\gamma= dz-sdy+(\frac{s^2}{2}-x)dx$;
 \item $\delta=d t-sdz+\frac{s^2}{2}dy+(xs-y-\frac{s^3}{6})dx$;
 \item $\eta=d w-ydz-(t-xy)dx$;
\end{itemize}
with
\[
d\omega=0=d\alpha, \quad d\beta=-\omega\wedge\alpha, \quad d\gamma=-\omega\wedge\beta, \quad d\delta=-\omega\wedge\gamma+\alpha\wedge\beta \quad \mathrm{and} \quad 
d\eta=\alpha\wedge\delta-\beta\wedge\gamma.
\]
In particular, the Lie algebra $\fg$ of $G$ is isomorphic to the nilpotent Lie algebra $L_{6,22}$ (resp.\ $\fh_{32}$) in the notation of \cite{BM} (resp.\ \cite{CFGU}). 
Hence $G$ is a connected, simply connected nilpotent Lie group. Moreover, arguing as in Section \ref{Dimension:4}, $\Xi=\Gamma\rtimes_\phi\bZ\subset G$ is a lattice, hence $M=\Xi\backslash G$ is a 
compact nilmanifold, diffeomorphic to the mapping torus $L_{(\bphi,1)}$.

\begin{remark}
Note that, under the identification between $L_{(\bphi, 1)}$ and $M$, the lcs structure of the first kind on $L_{(\bphi, 1)}$ is just the pair $(\bomega, \bar{\eta})$, where $\bomega$ and 
$\bar{\eta}$ are the 1-forms on $M$ induced by the left-invariant 1-forms $\omega$ and $\eta$, respectively, on $G$.
\end{remark}

\begin{theorem}\label{dim:6_second}
$M$ is a compact, complex and lcs nilmanifold which does not admit symplectic, lcK (with left-invariant complex structure) or Vaisman structures.
\end{theorem}
\begin{proof}
The lcs structure of $M$ is clear from the above discussion. By \cite[Theorem 3.3]{Sal} the Lie algebra $\fg$ of $G$ does not admit any complex structure. 
Hence $M$ does not admit any lcK structure with left-invariant complex structure. Next, assume that $M$ admits a symplectic structure. By \cite[Lemma 2]{Has}, $M$ also 
admits a left-invariant symplectic structure, i.e.\ $\fg$ is a symplectic Lie algebra. But this is not the case, according to \cite[Section 7]{BM}. Hence $M$ does not admit any symplectic structure.
Clearly $b_1(\fg)=2$, hence, again by Nomizu Theorem, $b_1(M)$ is also equal to 2. Therefore $M$ does not carry any Vaisman structure  by \cite[Corollary 3.6]{VaismanE}.
\end{proof}

\subsection{A higher dimensional example}\label{Dimension:higher}

In this section we construct, for each $n\geq 3$, a $2n$-dimensional nilmanifold $M_{2n}$. For $n=3$, $M_6$ is isomorphic to the nilmanifold described in \thmref{dim:6_first}. For $n\geq 4$, 
$M_{2n}$ admits a lcs structure of the first kind, a complex structure, but no symplectic structures, no lcK structures (with 
left-invariant complex structure) and no Vaisman structures.

Consider the $(2n-1)$-dimensional Lie group
\[
H_{2n-1}=\left\{
\begin{pmatrix} 
1 & x_{n-2} & x_{n-3} & \ldots & x_1 & y_{n-1} & w \\
0 & 1 & 0 & \ldots & 0 & 0 & y_{n-2}\\
0 & 0 & 1 & \ldots & 0 & 0 & y_{n-3}\\
\vdots & \vdots & \vdots & \ddots & \ldots &\ldots & \vdots\\
0 & \ldots & \ldots & \ldots & 1 & x_{n-1} & y_1\\
0 & \ldots & \ldots & \ldots & 0 & 1 & x_{n-1}\\
0 & \ldots & \ldots & \ldots & 0 & 0 & 1\\
\end{pmatrix} \ | \ x_i,y_i,w\in\bR, i=1,\ldots,n-1\right\}. 
\]
As a manifold, $H_{2n-1}$ is diffeomorphic to $\bR^{2n-1}$. Taking global coordinates $(x_1,y_1,\ldots,x_{n-1},y_{n-1},w)$, the multiplication in $H_{2n-1}$ is given by
\[
\begin{array}{c}
(x_1, y_1, \dots, x_{n-1}, y_{n-1}, w)\cdot( x'_1, y'_1, \dots, x'_{n-1}, y'_{n-1}, w') \\
= (x_1 + x'_1, y_1 + y'_1 + x_{n-1}x'_{n-1}, x_2 + x'_2, y_2 + y'_2, \dots, x_{n-2} + x'_{n-2}, y_{n-2} + y'_{n-2}, \\
x_{n-1} + x'_{n-1}, y_{n-1} + y'_{n-1} + x_1x'_{n-1}, w + w' + x_{n-2}y'_{n-2} + x_{n-3}y'_{n-3} + \dots + x_1y'_1 + y_{n-1}x'_{n-1}).
\end{array}
\]
Moreover, one may compute a basis of left-invariant forms on $H_{2n-1}$:
\begin{itemize}
 \item $\alpha_i=-dx_i$, $i=1,\ldots,n-2$;
 \item $\alpha_{n-1}=dx_{n-1}$;
 \item $\beta_1=dy_1-x_{n-1}dx_{n-1}$;
 \item $\beta_i=dy_i$, $i=2,\ldots,n-2$;
 \item $\beta_{n-1}=dy_{n-1}-x_1dx_{n-1}$;
 \item $\eta=dw-x_1(dy_1-x_{n-1}dx_{n-1})-\sum_{i=2}^{n-2}x_idy_i-y_{n-1}dx_{n-1}$.
\end{itemize}

These satisfy:
\begin{itemize}
 \item $d\alpha_i=0$, $i=1,\ldots,n-1$;
 \item $d\beta_i=0$, $i=1,\ldots,n-2$;
 \item $d\beta_{n-1}=\alpha_1\wedge\alpha_{n-1}$;
 \item $d\eta=\sum_{i=1}^{n-1}\alpha_i\wedge\beta_i$.
\end{itemize}
In particular, $\fh_{2n-1}$, the Lie algebra of $H_{2n-1}$, is a nilpotent Lie algebra. For $n=3$, it is isomorphic to the Lie algebra $L_{5,3}$ which appeared in Section \ref{Dimension:6}.

Notice that $(H_{2n-1},\eta)$ is a contact manifold. The subgroup $\Gamma_{2n-1}=\{(p_1,q_1,\ldots,p_{n-1},q_{n-1},r)\in H_{2n-1} \ | \ p_i,q_i,r\in\bZ\}\subset H_{2n-1}$ is a lattice. Since 
$\eta$ is left-invariant, it descends to a contact form $\bar{\eta}$ on the nilmanifold $L_{2n-1}=\Gamma_{2n-1}\backslash H_{2n-1}$. By Example \ref{ex:1}, the product $M_{2n}= 
L_{2n-1}\times S^1=(\Gamma_{2n-1}\times\bZ)\backslash(H_{2n-1}\times\bR)$ has a lcs structure of the first kind $(\bomega,\bar{\eta})$, where $\bomega$ is the angular form on $S^1$; as before, we are 
implicitly identifying $\bar{\eta}\in\Omega^1(L_{2n-1})$ with $\pi_1^*\bar{\eta}\in\Omega^1(M_{2n})$ and $\bomega\in\Omega^1(S^1)$ with $\pi_2^*\bomega\in\Omega^1(M_{2n})$.

Set $\fg_{2n}=\fh_{2n-1}\oplus\bR$ and let $\omega$ be the dual of a generator of the $\bR$-factor. Notice that $d\omega=0$.
\begin{proposition}\label{complex_str}
 The Lie algebra $\fg_{2n}$ admits a complex structure.
\end{proposition}
\begin{proof}
Let $\{X_1,Y_1,\ldots,X_{n-1},Y_{n-1},Z,T\}$ be the basis of $\fg_{2n}$ dual to $\{\alpha_1,\beta_1,\ldots,\alpha_{n-1},\beta_{n-1},\eta,\omega\}$.
In this basis, the non-zero brackets are
\[
 [X_i,Y_i]=-Z, \ i=1,\ldots,n-1 \quad \mathrm{and} \quad [X_1,X_{n-1}]=-Y_{n-1}.
\]
We define an endomorphism $J\colon\fg_{2n}\to\fg_{2n}$ by setting
\[
 J(X_1)=-X_{n-1}, \ J(Y_1)=-Y_{n-1}, \ J(X_i)=Y_i, i=2,\ldots,n-2, \ \mathrm{and} \ J(Z)=-T,
\]
and imposing $J^2=-\mathrm{Id}$. A straightforward computation shows that the Nijenhuis tensor of $J$ vanishes, hence $J$ is a complex structure on $\fg_{2n}$.
\end{proof}

\begin{proposition}\label{sympl_str}
 For $n\geq 4$, the Lie algebra $\fg_{2n}$ admits no symplectic structure.
\end{proposition}
\begin{proof}
Consider the basis $\{\alpha_1,\beta_1,\ldots,\alpha_{n-1},\beta_{n-1},\eta,\omega\}$ of $\fg_{2n}^*$ and the vector space splitting 
\[
\fg_{2n}^*=\fk^*\oplus \langle\eta\rangle\oplus \langle\omega\rangle, 
\]
where $\fk^*=\langle \alpha_1,\beta_1,\ldots,\alpha_{n-1},\beta_{n-1}\rangle$. Then 
\begin{equation}\label{splitting:Lie}
\Lambda^2\fg_{2n}^*=\Lambda^2\fk^*\oplus\fk^*\wedge\langle\eta\rangle \oplus\fk^*\wedge\langle\omega\rangle\oplus\langle\omega\rangle\wedge\langle\eta\rangle 
\end{equation}
with 
\begin{itemize}
 \item $d(\Lambda^2\fk^*)\subset\langle\alpha_1\wedge\alpha_{n-1}\rangle\wedge\fk^*$
 \item $d(\fk^*\wedge\langle\eta\rangle)\subset\langle\alpha_1\wedge\alpha_{n-1}\rangle\wedge\langle\eta\rangle 
\oplus\fk^*\wedge\langle d\eta\rangle$
 \item $d(\fk^*\wedge\langle\omega\rangle)\subset\langle\alpha_1\wedge\alpha_{n-1}\rangle\wedge\langle\omega\rangle$
 \item $d(\langle\omega\rangle\wedge\langle\eta\rangle)\subset\langle d\eta\rangle\wedge\langle 
\omega \rangle$.
\end{itemize}
Given $\sigma \in \Lambda^2\fg_{2n}^*$, decompose it according to \eqref{splitting:Lie} to get 
\[
\sigma=\sigma_1+\sigma_2\wedge\eta+\sigma_3\wedge\omega+c\omega\wedge\eta.
\]
Notice that $d\sigma_1\in \langle\alpha_1\wedge\alpha_{n-1}\rangle\wedge\fk^*$, $d(\sigma_2 \wedge \eta) \in \langle\alpha_1\wedge\alpha_{n-1}\rangle\wedge\langle\eta\rangle 
\oplus\fk^*\wedge\langle d\eta\rangle $ and
$d(\sigma_3\wedge\omega)\in\langle\alpha_1\wedge\alpha_{n-1}\rangle\wedge\langle\omega\rangle$. We have 
$d(c\omega\wedge\eta)=-c\omega\wedge d\eta$; since $d(c\omega\wedge\eta)$ is the only component of $d\sigma$ which can possibly be a multiple of 
$d\eta\wedge\omega$, $d\sigma=0$ implies $c=0$. By the same token, $d(\sigma_2\wedge\eta)$ is the only component of $d\sigma$ that can possibly lie in $\fk^*\wedge\langle 
d\eta\rangle$. Hence $d\sigma=0$ implies $d(\sigma_2\wedge\eta)=0$. Write $\sigma_2=\sum_{i=1}^{n-1}(a_i\alpha_i+b_i\beta_i)$. The only possible non-zero component of 
$d\sigma$ in $\langle\alpha_1\wedge\alpha_{n-1}\rangle\wedge\langle\eta\rangle$ comes from a term $d(\beta_{n-1}\wedge\eta)$. Hence, $d\sigma=0$ implies $b_{n-1}=0$.
We compute
\begin{multline*}
d\left(\sum_{i=1}^{n-1}a_i\alpha_i\wedge\eta+\sum_{i=1}^{n-2}b_i\beta_i\wedge\eta\right)=\\
 -\sum_{i=1}^{n-1}\sum_{j\neq i}a_i\alpha_i\wedge\alpha_j\wedge\beta_j-\sum_{i=1}^{n-2}\sum_{j\neq i}b_i\beta_i\wedge\alpha_j\wedge\beta_j
 -\sum_{i=1}^{n-2}b_i\beta_i\wedge\alpha_{n-1}\wedge\beta_{n-1};
\end{multline*} 
we see that $d\sigma=0$ implies $b_i=0$ for $i=1,\ldots,n-2$. We are left with $\sigma_2=\sum_{i=1}^{n-1}a_i\alpha_i$. If $n=3$,
\[
d(\sigma_2\wedge\eta)=-a_1\alpha_1\wedge\alpha_2\wedge\beta_2-a_2\alpha_2\wedge\alpha_1\wedge\beta_1\in \langle\alpha_1\wedge\alpha_2\rangle\wedge\fk^*
\]
and choosing $a_1=0$ and $a_2=1$, for instance, we get
\[
 d(\sigma_2\wedge\eta)=\alpha_1\wedge\alpha_2\wedge\beta_1=-d(\beta_1\wedge\beta_2);
\]
this gives the symplectic structure $\sigma=\alpha_1\wedge\omega+\alpha_2\wedge\eta+\beta_1\wedge\beta_2$, which was implicitly used in the proof of 
\thmref{dim:6_first}. Hence $\fg_6$ is symplectic. However, if $n\geq 4$, $\langle\alpha_1\wedge\alpha_{n-1}\rangle\wedge\fk^*\cap \fk^*\wedge\langle d\eta\rangle=\emptyset$ 
and if $d(\sigma_2\wedge\eta)$ has non-zero component in $\fk^*\wedge\langle d\eta\rangle$, then $\sigma$ is not closed. This forces the $a_i$ to vanish. This implies that every 2-cocycle has rank 
$<n$, hence $\fg_{2n}$ has no symplectic structure.
\end{proof}

\begin{theorem}\label{third:example}
 For every $n\geq 4$, $M_{2n}$ is a $2n$-dimensional nilmanifold which is complex, locally conformal symplectic but has no symplectic structure, carries no locally conformal K\"ahler metric (with left-invariant complex 
 structure) and no Vaisman metric.
\end{theorem}
\begin{proof}
By construction, $M_{2n}=(\Gamma_{2n-1}\times\bZ)\backslash (H_{2n-1}\times\bR)$, hence it is a nilmanifold. We have already described the lcs structure on $M_{2n}$. By Proposition 
\ref{complex_str}, $H_{2n-1}\times\bR$ admits a left-invariant complex structure, which therefore endows $M_{2n}$ with a complex structure. Also, since $n\geq 4$, $\fg_{2n}$ admits no symplectic 
structure by Proposition \ref{sympl_str}, hence $H_{2n-1}\times\bR$ has no left-invariant symplectic structure. Therefore, using Lemma 2 in \cite{Has}, we deduce that $M_{2n}$ carries 
no symplectic structures.

One computes readily that $b_1(\fg_{2n})=2n-2$; hence, by Nomizu \thmref{Nomizu}, $b_1(M_{2n})=2n-2$ and $M_{2n}$ is not a quotient of the Heisenberg group multiplied by $\bR$, according to 
Lemma \ref{Betti_Heisenberg}. Again by \thmref{Sawai}, $M_{2n}$ does not carry any lcK structure (with left invariant complex structure). Finally, since $b_1(M)$ 
is even $M_{2n}$ carries no Vaisman structure by \cite[Corollary 3.6]{VaismanE}.
\end{proof}

\begin{remark}
By Sawai's \thmref{Sawai}, if $M_{2n}$ carries a lcK metric, then the corresponding complex structure can not be left-invariant. This raises the question of studying nilmanifolds endowed with 
complex structures that are not left-invariant.
\end{remark}

%
%

\section{Local and global structure of a lcs manifold of the first kind with a compact leaf in its canonical foliation}\label{sec:foliation}

In this section we present a description of the local and global structure of a compact lcs manifold of the first kind with a compact leaf in its canonical foliation. For this purpose, we use some results on a special class 
of foliations of codimension $1$. We will include basic proofs of them in the next subsection.

\subsection{Some general results on a special class of foliations of codimension 1}\label{fol-codim-1}

Let $\omega$ be a closed $1$-form on a smooth manifold $M$, with $\omega(p) \neq 0$, for every $p \in M$. The codimension 1 foliation ${\mathcal F}$, whose characteristic space at $p 
\in M$ is $\cF(p) = \{ v \in T_pM \ | \ \omega(p)(v) = 0 \}$, is said to be \emph{transversely parallelizable complete} if there exists a complete vector field $U$ on $M$ such that $\omega(U) = 1$. 

Then, we may prove the following results:
\begin{theorem} (the local description)\label{local-descrip-general}
Let $\omega$ be a closed $1$-form on a smooth manifold $M$ such that $\omega(p) \neq 0$, for every $p \in M$. Suppose that $L$ is a compact leaf of the foliation ${\mathcal F}=\{\omega = 0\}$ and 
that ${\mathcal F}$ is transversely parallelizable complete, $U$ being a complete vector field on $M$ such that $\omega(U) = 1$. If $\Psi_L\colon L\times\mathbb{R} \to M$ is the restriction to 
$L\times\mathbb{R}$ of the flow $\Psi\colon M \times \mathbb{R} \to M$ of $U$ then there exists $\epsilon > 0$ and an open subset $W$ of $M$, $L \subseteq W$, such that $\omega|_{W}$ is an exact 
$1$-form and
\[
\Psi^{\epsilon}_L \coloneq (\Psi_{L})\big|_{L\times (-\epsilon, \epsilon)}\colon L\times (-\epsilon, \epsilon) \to W
\]
is a diffeomorphism. Moreover,
\begin{equation}\label{exact-omega}
(\Psi^{\epsilon}_L)^*(\omega|_{W}) = pr_2^*(dt), \; \; U|_{W} \circ \Psi_L^{\epsilon} = T\Psi_L^{\epsilon} \circ \frac{\partial}{\partial t}\Big|_{L\times (-\epsilon, \epsilon)}
\end{equation}
$t$ being the standard coordinate on $\mathbb{R}$. In addition, the leaves of ${\mathcal F}$ over points of $W$ are of the form $\Psi_t(L)$, with $t \in (-\epsilon, \epsilon)$. In particular, they 
are contained in $W$ and they are diffeomorphic to $L$. 
\end{theorem}
\begin{theorem} (the global description)\label{global-descrip-general}
Under the same hypotheses as in Theorem \ref{local-descrip-general} and if, in addition, $M$ is connected then all the leaves of the foliation ${\mathcal F}$ are of the form $\Psi_t(L)$ (with $t \in 
\mathbb{R}$) and, therefore, diffeomorphic to $L$. Furthermore, we have two possibilities:
\begin{enumerate}
\item
If $M$ is not compact then the map
$
\Psi_L\colon L\times\mathbb{R} \to M
$
is a diffeomorphism and
\begin{equation}\label{omega-exact-global}
\Psi_{L}^*(\omega) = pr_2^*(dt), \; \; U \circ \Psi_L = T\Psi_L \circ \frac{\partial}{\partial t}, 
\end{equation}
where $T\Psi_L\colon T(L\times\mathbb{R}) \to TM$ is the tangent map of $\Psi_L$.
\item
If $M$ is compact then there exists $c > 0$ such that $\Psi_c\colon L \to L$ is a diffeomorphism and the map $\Psi_L\colon L\times\mathbb{R} \to M$ induces a diffeomorphism between $M$ and the 
mapping torus of $L$ by $\Psi_c$ and $c$.
\item
Under the identification between $M$ and $L\times_{(\Psi_c,c)} \mathbb{R} $, $U$ is the vector field on $L\times_{(\Psi_c,c)}\mathbb{R} $ which is induced by the vector field 
$\frac{\partial}{\partial t}$ on $L\times\mathbb{R} $ and $\omega$ is the $1$-form on $M$ which is induced by the canonical exact $1$-form $dt$ on $L\times\mathbb{R} $. So, if 
$\pi\colon M = L\times_{(\Psi_c,c)}\mathbb{R} \to S^1 = \mathbb{R}/c\mathbb{Z}$ is the canonical projection and $\tau$ is the length element of $S^1$, we have that $\pi^*(\tau) = \omega$.
\end{enumerate}
\end{theorem}
\begin{proof}(of Theorem \ref{local-descrip-general})
Let $p\in L$ be a point. The tangent space to $M$ at $p$ decomposes as $T_pM=T_pL\oplus\langle U(p)\rangle$. Since $\Psi$ is the flow of $U$, one has
\[
(T_{(p,0)}\Psi_L)\left(\frac{\partial}{\partial t}_{(p,0)}\right)=U(p),\quad 
(T_{(p,0)} \Psi_L)(X(p))=X(p)
\]
for every vector $X(p)\in T_pL$. Hence $T_{(p,0)} \Psi_L\colon T_pL\times\bR\to T_pM$ is an isomorphism, thus there exist $\epsilon_p>0$, an open set $W_p^L\subset L$ with 
$p\in W_p^L$ and an open set $W_p\subset M$ with $p\in W_p$ such that
\[
\Psi\big|_{W_p^L\times (-\epsilon_p,\epsilon_p)}\colon W_p^L\times(-\epsilon_p,\epsilon_p)\to W_p
\]
is a diffeomorphism. Clearly $L=\bigcup_{p\in L}W_p^L$, and since $L$ is compact there exist $p_1,\ldots p_k\in L$ such that $L=\bigcup_{i=1}^kW_{p_i}^L$. Set 
$\epsilon=\min\{\epsilon_{p_1},\ldots,\epsilon_{p_k}\}$ and $W=\Psi (L\times(-\epsilon,\epsilon))$. Then $W\subset M$ is an open set, $L\subset W$ and 
\[
\Psi^{\epsilon}_L\coloneq\Psi\big|_{L\times(-\epsilon,\epsilon)}\colon L\times(-\epsilon,\epsilon)\to W
\]
is a diffeomorphism. On the other hand, we have that 
\[
\cL_U\omega= d(\omega(U)) + i_U(d\omega) = 0.
\]
 Hence $\Psi_t^*\omega=\omega$ and, as the pullback of $\omega|_{W}$ to $L$ under the inclusion $L\hookrightarrow W$ is zero, we get 
$(\Psi^{\epsilon}_L)^*(\omega|_{W})=pr_2^*(dt)$. Furthermore, using that $\Psi$ is the flow of $U$, we directly deduce the second relation in \eqref{exact-omega}.

Finally, we show that if $t\in(-\epsilon,\epsilon)$ and $p\in L$, then $\Psi_t(L)$ is the leaf $L'$ of $\cF$ over $\Psi_t(p)$. It is clear that $\Psi_t(L)$ is a submanifold of $M$ which 
is diffeomorphic to $L$. Moreover, since $\Psi_t^*\omega=\omega$, $\Psi_t(L)$ is a connected integral submanifold of $M$, hence $\Psi_t(L)\subset L'$. Thus $L\subset \Psi_{-t}(L')$, and 
$\Psi_{-t}(L')$ is an integral submanifold of $\cF$. But this implies $L=\Psi_{-t}(L')$, and $L'=\Psi_t(L)$.
\end{proof}

\begin{proof} (of Theorem \ref{global-descrip-general})
We will proceed in three steps.

\medskip

\noindent \textbf{First step} We will see that $\Psi_L\colon L\times\bR\to M$ is a surjective local diffeomorphism which maps 
the submanifolds $L\times \{t\}$, with $t\in\bR$, to the leaves of the canonical foliation $\cF$ of $M$. In particular, each leaf of $\cF$ is of the form 
$\Psi_t(L)$, $t\in\bR$, and it is diffeomorphic to $L$. 

In fact, proceeding as in the proof of \thmref{local-descrip-general}, we deduce that \eqref{omega-exact-global} holds
and that if $(p,t)\in L\times\bR$ then $\Psi_t(L)$ is the leaf of $\cF$ over the point $\Psi_t(p)$. This implies that $\Psi_L$ is a local diffeomorphism. Indeed, if $X\in T_pL$ and $\lambda\in\bR$ satisfy
\[
0=(T_{(p,t)}\Psi_L)\left(X+\lambda\frac{\partial}{\partial t}\Big|_t\right)=(T_p\Psi_t)(X)+\lambda U(p)
\]
we have that
\[
0=\omega_p((T_p\Psi_t)(X)+\lambda U(p))=\lambda,
\]
(note that $(T_p\Psi_t)(X)\in \cF(\Psi_t(p))$). Hence $(T_p\Psi_t)(X)=0$, which gives $X=0$. 

We claim that $\Psi_L(L\times\bR)=M$. Since $\Psi_L$ is a local diffeomorphism, it is clear that 
$\Psi_L(L\times\bR)$ is an open subset of $M$. We will show that $\Psi_L(L\times\bR)$ is a closed subset of $M$; the claim will follow from this and from the assumption that $M$ is 
connected. Take $p'\in M - \Psi_L(L\times\bR)$ and denote by $L'$ the leaf of $\cF$ over $p'$. Then there exist $\epsilon'>0$ and $W'$ an open subset of $M$ such that $L' \subseteq W'$ and
\[
\Psi\big|_{L'\times(-\epsilon',\epsilon')}\colon L'\times(-\epsilon',\epsilon')\to W'
\]
is a diffeomorphism. We will show that $W'\subset M - \Psi_L(L\times\bR)$. Indeed, suppose that $q'\in W'$ and 
$q'=\Psi_t(p)$ for some $p\in L$. Then there exist $r'\in L'$ and $t'\in (-\epsilon',\epsilon')$ such that $q'=\Psi_t(p)=\Psi_{t'}(r')$. Thus $r'=\Psi_{t-t'}(p)\in L'\cap \Psi_{t-t'}(L)$.
Since $L$ and $L'$ are leaves of $\cF$, we conclude that $L'=\Psi_{t-t'}(L)$, a contradiction. We have proved so far that $\Psi_L(L\times\bR)=M$. This 
concludes the first step.\\

\noindent \textbf{Second step} We will see that the space of leaves $M/\cF=\widetilde{M}$ of the canonical foliation $\cF$ is a smooth manifold such that the canonical projection 
$\pi\colon M\to\widetilde{M}$ is a submersion. 

Let $p'\in M$ be a point and let $L'$ be the leaf of $\cF$ through $p'$. By Theorem \ref{local-descrip-general}, there exist an open set $W'\subset M$, with $L'\subset W'$, and $\epsilon'>0$ such that
\[
\Psi\big|_{L'\times(-\epsilon',\epsilon')}\colon L'\times(-\epsilon',\epsilon')\to W'
\]
is a diffeomorphism. Take coordinates on some open subset $W'_{L'}\subset L'$ with $p'\in W'_L$. Then $\hat{W}'=\Psi_{L'}( W'_{L'}\times(-\epsilon',\epsilon'))$ is an open subset of $M$
which admits a system of coordinates adapted to the foliation $\cF$, and $p'\in \hat{W}'$. Moreover, if $t',s'\in(-\epsilon',\epsilon')$, $t' \neq s'$, then the plaques $\Psi_{t'}(W'_{L'})$ and 
$\Psi_{s'}(W'_{L'})$ in $\hat{W}'$ are contained in different leaves of $\cF$. This is precisely the condition needed to ensure that the quotient space $\widetilde{M}=M/\cF$ is a smooth
manifold and that the canonical projection $\pi\colon M\to \widetilde{M}$ is a submersion. \\

\noindent \textbf{Third step} Let $L$ be our compact leaf and let $p\in L$ be a point. Set
\[
A_p=\{t\in\bR - \{0\} \ | \ \Psi_t(p)\in L\}\subset\bR - \{0\}.
\]
We will see that:
\begin{itemize}
\item
$A_p = \emptyset$ gives the first possibility of the theorem, and
\item
$A_p\neq\emptyset$ gives the second possibility.
\end{itemize}
 
In fact, suppose that $A_p=\emptyset$. Then, since all the leaves of $\cF$ are of the form $\Psi_t(L)$ for some $t\in \bR$, the map $\Psi_L\colon \bR\times L\to M$ is injective, hence a 
diffeomorphism 
by the first step. 

Now, assume that $A_p\neq\emptyset$. We shall see that there exists $c\in\bR$, $c>0$, such that $A_p=c\bZ$. In fact, set
\[
c=\inf A_p^+, \quad A_p^+=\{t\in\bR^+ \ | \ \Phi_t(p)\in L\}.
\]
Note that $A_p^+\neq\emptyset$. In addition, from Theorem \ref{local-descrip-general} follows that $c>0$. Furthermore, being $L$ a closed submanifold, we see that $c\in A_p^+$. Therefore 
$p\in L\cap\Psi_{-c}(L)$ which implies $L=\Psi_{-c}(L)$ and $L=\Psi_c(L)$. Hence, if $k\in\bZ$, $ck\in A_p$ and $c\bZ\subset A_p$. Conversely, if $t\in A_p$, there exists $k\in\bZ$ such 
that
$ck\leq t <c(k+1)$. If $t>ck$, we have $0<t-ck<c$ and $t-ck\in A_p^+$, which contradicts the fact that $c$ is the infimum of $A_p^+$. Hence $t=ck$ and $A_p=c\bZ$.

To conclude, we will show that $\Psi_L\colon L\times\bR\to M$ induces a diffeomorphism between the manifold $L\times_{(\Psi_c,c)}\bR$ and $M$. For this, it is sufficient to prove that if $(y,t)$ and 
$(y',t')$ in $L\times\bR$, then $\Psi_L(y,t)=\Psi_L(y',t')$ if and only if there exists $k\in\bZ$ such that 
\begin{equation}\label{condition3}
y=\Psi_{ck}(y')\qquad \mathrm{and} \qquad t'=t+ck .
\end{equation}
Clearly, if \eqref{condition3} holds, then $\Psi_L(y,t)=\Psi_L(y',t')$. Conversely, suppose $\Psi_L(y,t)=\Psi_L(y',t')$ and suppose $t'\geq t$. Then 
$y=\Psi_{t'-t}(y')\in L\cap\Psi_{t'-t}(L)$, from which follows as usual $L=\Psi_{t'-t}(L)$. In particular, $\Psi_{t'-t}(p)\in L$, so that $t'-t\in A_p$. Therefore, there exists $k\in\bZ$ 
such that $t'=t+ck$ and from this we see that $y=\Psi_{ck}(y')$. This concludes the proof of the theorem.
\end{proof}




\subsection{The particular case of a lcs manifold of the first kind with a compact leaf in its canonical foliation}\label{lcs-can-fol}

In this section, we will consider the particular case when $M$ has a lcs structure of the first kind and $\omega$ is the Lee $1$-form of $M$.

We recall that if $L$ is a leaf of the canonical foliation ${\mathcal F}=\{\omega = 0\}$, $\eta$ is the anti-Lee $1$-form of $M$ and $i\colon L \to M$ is the canonical inclusion then $\eta_L 
= i^*(\eta)$ is a contact $1$-form on $L$. Thus:
\begin{itemize}
\item
The product manifold $L\times(-\epsilon, \epsilon)$, with $\epsilon > 0$, admits a canonical gcs structure of the first kind and
\item
If $L$ is compact, $\phi\colon L \to L$ is a strict contactomorphism and $c > 0$ then the mapping torus $L \times_{(\phi,c)} \mathbb{R}$ of $L$ by $\phi$ and $c$ is a compact lcs manifold of the 
first kind
\end{itemize}
(see Section \ref{LCS_Manifolds} for more details).

Now, we will introduce two natural definitions which will be useful in the sequel. The first one is well known (see, for instance, \cite{LiMa}).
\begin{definition}\label{contact_isomorphism}
Let $\eta_1$ and $\eta_2$ be contact forms on the manifolds $M_1$ and $M_2$ respectively. A diffeomorphism $\Psi\colon M_1\to M_2$ is said to be a 
\emph{strict contactomorphism} if $\Psi^*\eta_2=\eta_1$.
\end{definition}
\begin{definition}
Let $(\omega_1,\eta_1)$ and $(\omega_2,\eta_2)$, be lcs structures of the first kind on manifolds $M_1$ and $M_2$ respectively. A diffeomorphism $\Psi\colon M_1\to M_2$ is said to be 
\emph{lcs morphism of the first kind} if $\Psi^*\omega_2=\omega_1$ and $\Psi^*\eta_2=\eta_1$.
\end{definition}
\begin{remark}
$\Psi$ being a lcs morphism of the first kind implies that $\Psi^*\Phi_2=\Phi_1$, where 
$\Phi_i=d\eta_i+\eta_i\wedge\omega_i$, $i=1,2$.
\end{remark}
We prove the two following results
\begin{theorem}(the local description)\label{local-descrip-lcs}
Let $(\omega, \eta)$ be a lcs structure of the first kind on a manifold $M$ and let $U$ be the anti-Lee vector field of $M$. Suppose that $U$ is complete and that $L$ is a compact leaf of the canonical 
foliation $\cF=\{\omega = 0\}$ on $M$. If $\Psi_L\colon L\times\bR\to M$ is the restriction to $L\times\mathbb{R}$ of the flow $\Psi\colon M\times\bR\to M$ of $U$ 
then, for every point $p \in L$, there exist an open subset $W \subseteq M$, $L \subseteq W$, and a positive real number $\epsilon > 0$ such that
\[
\Psi^{\epsilon}_L \coloneq\Psi\big|_{L\times(-\epsilon, \epsilon)}\colon L\times(-\epsilon, \epsilon) \to W
\]
is an isomorphism between the gcs manifolds $L\times(-\epsilon, \epsilon)$ and $W$. Moreover, the leaves of $\cF$ over points of $W$ are of the form $\Psi_t(L)$, with $t \in(-\epsilon,\epsilon)$. 
In particular, they are contained in $W$ and they are strict contactomorphic to $L$.
\end{theorem}
\begin{theorem}(the global description)\label{global-descrip-lcs} 
If, in addition to the hypotheses of \thmref{local-descrip-lcs}, $M$ is connected, then all the leaves of the canonical foliation are of the form $\Psi_t(L)$ (with $t \in \mathbb{R}$) 
and, therefore, they are strict contactomorphic to $L$. Furthermore, we have the two following possibilities:
\begin{enumerate}
\item
If $M$ is not compact then the map
$
\Psi_L\colon L\times\mathbb{R}\to M
$
is a gcs isomorphism between the gcs manifolds $L\times\mathbb{R}$ and $M$.
\item
If $M$ is compact then there exists $c > 0$ such that $\Psi_c\colon L \to L$ is a strict contactomorphism and $\Psi_L$ induces a lcs isomorphism of the first kind between $M$ and the mapping torus 
of $L$ by $\Psi_c$ and $c$.
\end{enumerate}
\end{theorem}
In order to prove Theorems \ref{local-descrip-lcs} and \ref{global-descrip-lcs}, we will use the following result:
\begin{proposition}\label{eta-invariant}
Under the same hypotheses as in Theorem \ref{local-descrip-lcs}, we have that
\[
\Psi_L^*(\eta) = pr_1^*(\eta_L),
\]
where $pr_1\colon L\times\mathbb{R}\to L$ is the canonical projection on the first factor and $\eta_L$ is the contact $1$-form on $L$.
\end{proposition}
\begin{proof}
It is clear that
\[
{\mathcal L}_{U}\eta = d(\eta(U)) + i_U(d\eta) = 0,
\]
and, thus,
\begin{equation}\label{invariant-eta}
\Psi_t^*\eta = \eta, \; \; \; \mbox{ for every } t \in \mathbb{R}.
\end{equation}
In addition, using that $\eta(U) = 0$, we also have that
\begin{equation}\label{eta-U-0}
(\Psi_{L}^* \eta)\left(\frac{\partial}{\partial t}\right) = 0.
\end{equation}
Therefore, from \eqref{invariant-eta} and \eqref{eta-U-0}, we conclude that
\[
\Psi_{L}^* \eta = pr_1^*(\eta_L).
\]
\end{proof}
\begin{proof} (of Theorem \ref{local-descrip-lcs})
If follows using Theorem \ref{local-descrip-general} and Proposition \ref{eta-invariant}.
\end{proof}
\begin{proof} (of Theorem \ref{global-descrip-lcs})
It follows using Theorem \ref{global-descrip-general} and Proposition \ref{eta-invariant}.
\end{proof}

The first possibility of \thmref{global-descrip-lcs} means that $M$ is symplectomorphic to the standard symplectization of the leaf. In the compact case, \thmref{global-descrip-lcs} tells 
us that a manifold endowed with a locally conformal symplectic structure of the first kind $(\omega,\eta)$ fibres over a circle, that $\eta$ pulls back to a contact form on the fibre and that the 
gluing map is a strict contactomorphism. This displays the similarity with the result of Banyaga. The main difference between the two results is that Banyaga restricts to the compact case and 
needs to modify the foliation $\mathcal{F}$ to a $\mathscr{C}^{\infty}-$near foliation $\mathcal{F}'$ which he proves to be a fibration, while we work directly with $\mathcal{F}$. In our case, however, we must 
assume the completeness of $U$ (which is automatic if our manifold is compact) as well as the existence of a compact leaf of $\cF$.

\begin{remark} A similar issue appears in the context of \emph{cosymplectic manifolds}. Recall that a manifold $M^{2n+1}$ is 
cosymplectic (in the sense of Libermann, see \cite{Lib}) if there exist 
$\alpha\in\Omega^1(M)$ and $\beta\in\Omega^2(M)$ such that $d\alpha=0=d\beta$ and $\alpha\wedge\beta^n\neq 0$. It was proven by Li in \cite{Li} that, in the compact case, a cosymplectic manifold 
corresponds to a mapping torus of a symplectic manifold and a symplectomorphism. In particular, such manifolds fibre over $S^1$ with fibre a symplectic manifold. However, in order to 
obtain this result (whose proof is similar in spirit to that of Banyaga for the lcs case) one needs to perturb the foliation $\{\alpha=0\}$, so that the original cosymplectic structure is 
destroyed. Both the result of Banyaga and that of Li rely on a theorem of Tischler \cite{Ti}, which asserts that a compact manifold with a closed and nowhere zero 1-form fibres over the 
circle $S^1$. Recently, Guillemin, Miranda and Pires (see \cite{GMP}) have proven a result similar to our \thmref{global-descrip-lcs} in the context of compact cosymplectic manifolds, working 
directly with the foliation $\{\alpha=0\}$.
\end{remark}

\subsection{A Martinet-type result for lcs structures of the first kind}\label{Martinet}
As another application of the results in Section \ref{fol-codim-1}, we obtain a Martinet-type result about the existence of lcs structures of the first kind on a certain type of oriented 
compact manifolds of dimension $4$.

First of all, we recall the classical result of Martinet \cite{Mar}: if $L$ is an oriented closed manifold of dimension $3$ then there exists a contact $1$-form on $L$. There exist some equivariant versions of this result:
\begin{theorem}[\cite{KaTs,Nie}]\label{inv-contact-forms-1}
Let $L$ be an oriented closed manifold of dimension $3$ and an action of $S^1$ on $L$ which preserves the orientation. Then there exists a $S^1$-invariant contact $1$-form on $L$.
\end{theorem}
\begin{theorem}[\cite{Ca}]\label{inv-contact-forms-2}
Let $L$ be an oriented closed manifold of dimension $3$ and suppose that a finite group $\Gamma$ of prime order acts on $L$ preserving the orientation. Then there exists a $\Gamma$-invariant 
contact $1$-form on $L$.
\end{theorem}
Now, let $M$ be an oriented connected manifold of dimension $4$, $\omega$ a closed $1$-form on $M$ without singularities and let $L$ be a compact leaf of the foliation $\cF=\{\omega = 0\}$.

If $M$ is not compact then, using \thmref{global-descrip-general}, it directly follows that $M$ admits a gcs structure of the first kind.

Next, suppose that $M$ is compact. Using again \thmref{global-descrip-general}, we deduce that the global structure of $M$ is completely determined by $L$, a real number $c > 0$ and a 
diffeomorphism $\phi\colon L \to L$. In fact, if $U$ is a vector field on $M$ such that $\omega(U) = 1$ and $\Psi_L\colon L\times\mathbb{R}\to M$ is the restriction to $L\times\mathbb{R}$ of the 
flow of $U$ then $\phi = \Psi_c$ and $\Psi_L$ induces a diffeomorphism between $L_{(\phi,c)} = (L\times\mathbb{R}) / \sim_{(\phi,c)}$ and $M$.

On the other hand, since $M$ is orientable, we can choose a volume form $\nu$ on $L\times\mathbb{R}$ which is invariant under the transformation
\[
L\times\mathbb{R} \to L\times\mathbb{R}, \; \; \; (x,t) \to (\phi(x),t-c).
\]
Thus, for every $t \in \mathbb{R}$, the $3$-form $i(\frac{\partial}{\partial t})\nu$ on $L\times\mathbb{R}$ induces a volume form $\nu_t$ on $L$ in such a way that:
\begin{itemize}
\item
the volume form $\nu_t$ is $\phi$-invariant and
\item
any two of these volume forms define the same orientation on $L$.
\end{itemize}
Using the previous facts and Theorems \ref{global-descrip-general}, \ref{inv-contact-forms-1} and \ref{inv-contact-forms-2}, we conclude
\begin{corollary}\label{Martinet_type}
Let $M$ be an oriented connected manifold of dimension $4$, $\omega$ a closed $1$-form on $M$ without singularities and $L$ a compact leaf of the foliation $\cF=\{\omega = 0\}$.
\begin{enumerate}
\item
If $M$ is not compact then it admits a gcs structure of the first kind. The structure is globally conformal to the symplectization of a leaf.
\item
If $M$ is compact then $M$ may be identified with a mapping torus of $L$ by a real number $c > 0$ and a diffeomorphism $\phi\colon L \to L$. Moreover:
\begin{enumerate}
\item
If there exists an action $\psi\colon S^1 \times L \to L$ which preserves the orientation induced on $L$ and $\phi=\psi_\lambda$, for some $\lambda \in S^1$, then $M$ admits a lcs structure of 
the first kind.
\item
If $\phi\colon L \to L$ preserves the orientation induced on $L$, the discrete subgroup of transformations of $M$
\[
\Gamma = \{ \phi^k \ | \ k \in \mathbb{Z} \}
\]
is finite and its order is prime, then $M$ also admits a lcs structure of the first kind.
\end{enumerate} 
\end{enumerate}
\end{corollary}

%
%

\section{Locally conformal symplectic Lie algebras}\label{sec:lcs_Lie_algebras}

\subsection{Lcs structures on Lie algebras}
In this section we describe locally conformal symplectic structures on Lie algebras.

\begin{definition} Let $\fg$ be a real Lie algebra of dimension $2n$ ($n\geq 2$). A \emph{locally conformal symplectic (lcs) structure} on $\fg$ consists of:
\begin{itemize}
\item $\Phi\in\Lambda^2\fg^*$, non-degenerate, i.e. $\Phi^n\neq 0$;
\item $\omega\in\fg^*$, with $ d\omega=0$, such that $d\Phi=\omega\wedge\Phi$.
\end{itemize}
\end{definition}

Here $d$ is the Chevalley-Eilenberg differential on $\Lambda^\bullet\fg^*$. Locally conformal symplectic Lie algebras (along with contact Lie algebras) have been considered in \cite{IM}, as a special instance of algebraic Jacobi 
structures.

Let $(\fg,\Phi,\omega)$ be a lcs Lie algebra. Since $\Phi$ is non-degenerate, it defines an isomorphism $\fg\to\fg^*$, $X\mapsto \imath_X\Phi$.

The \emph{automorphisms} of the lcs structure $(\Phi,\omega)$, denoted $\fg_\Phi$, are the elements of $\fg$ which preserve the 2-form $\Phi$, that is
\[
\fg_\Phi=\{X\in\fg \ | \ L_X\Phi=0\}.
\]
$\fg_\Phi\subset\fg$ is a Lie subalgebra of $\fg$. Let $\ell\colon \fg\to\bR$ be the map 
which sends $X\in \fg$ to $\omega(X)\in\bR$. Viewing $\bR$ as an abelian Lie algebra, the closedness of $\omega$ implies that $\ell$ is a morphism of Lie algebras. 
Thus, the restriction of $\ell$ to $\fg_{\Phi}$ also is a Lie algebra morphism known as \emph{Lee morphism}. The image of the Lee morphism is 1-dimensional; hence $\ell$ is surjective, if it is non-zero.

\begin{definition} 
The lcs Lie algebra $(\fg,\Phi,\omega)$ is said to be \emph{of the first kind} if $\ell$ is surjective; \emph{of the second kind} if it is zero.
\end{definition}

In this paper we will deal with lcs Lie algebras of the first kind. Let $(\fg,\Phi,\omega)$ be a lcs Lie algebra of the first kind. Pick a vector $U\in\fg_\Phi$ such that 
$\ell(U)=1$. Define $\eta\in\fg^*$ by the equation $\eta=-\imath_U\Phi$; clearly $U\in\ker(\eta)$. Also, define $V \in\fg$ by $\omega=\imath_V\Phi$; notice that $V\in\ker(\omega)$ and that 
$\imath_V\eta=1$. Since $U\in\fg_\Phi$, 
$\Lie_U\Phi=0$ and hence $d\imath_U\Phi=-\imath_U d\Phi$, which implies
\[
d\eta=-d\imath_U\Phi=\imath_U d\Phi=\imath_U(\omega\wedge\Phi)=\Phi+\omega\wedge\eta.
\]
Thus $\Phi=d\eta-\omega\wedge\eta$. Easy computations show that $\imath_U d\eta=\imath_V d\eta=0$. Therefore $d\eta\in\Lambda^2\fg^*$ has two vectors in its kernel, and can not have 
maximal rank $n$. We compute
\[
0\neq\Phi^n=(d\eta+\eta\wedge\omega)^n=n(d\eta)^{n-1}\wedge\eta\wedge\omega.
\]
But then $d\eta$ has rank $2n-2$ and $\eta$ behaves like a contact form on the ideal $\ker(\omega)$, which has dimension $2n-1$.

To sum up, we obtain an algebraic analogue to Proposition \ref{lcs_1_kind}: a lcs structure of the first kind on a Lie algebra $\fg$ of dimension $2n$ is completely determined by two 
1-forms $\omega,\eta \in \fg^*$ such that
\begin{equation}\label{algebraic-l.c.s-first-kind}
d\omega = 0, \quad \mathrm{rank}(d\eta)<2n \quad \mbox{and} \quad \omega\wedge \eta\wedge (d\eta)^{n-1} \neq 0.
\end{equation}
From now on we will use $(\omega, \eta)$ to denote a lcs structure of the first kind on the Lie algebra $\fg$. Clearly, $\Phi=d\eta-\omega\wedge\eta$.
 By \eqref{algebraic-l.c.s-first-kind} there exist $U, V \in \fg$, \emph{the anti-Lee and Lee vectors}, characterized by the conditions
\begin{eqnarray*}
\omega(U) =1, & \eta(U) = 0, & i_Ud\eta = 0,\nonumber \\
\omega(V) =0, & \eta(V) = 1, & i_Vd\eta = 0.
\end{eqnarray*}
Note that the previous conditions imply that
\[
i_{[U, V]}\omega = 0, \; \; i_{[U, V]}\eta = 0 \; \; \mbox{ and } \; \; i_{[U, V]}d\eta = 0,
\]
and, therefore,
\begin{equation}\label{U-V-flat}
[U, V] = 0.
\end{equation}

Let now $(\fg,\Phi,\omega)$ be a lcs Lie algebra. Since $\omega$ is a closed 1-form, we can perform the construction of Example \ref{1_dim_rep}. In particular, $\Phi$ is a 2-cocycle 
for $d_\omega\colon C^2(\fg;W_\omega)\to C^3(\fg;W_\omega)$, hence it defines a cohomology class $[\Phi]\in H^2(\fg;W_\omega)$. We call $(\Phi,\omega)$ \emph{exact} if $[\Phi]=0$, \emph{non-exact} otherwise. Assume that the 
lcs structure is of the first kind. Then there exists $\eta\in\fg^*$ such that $\Phi=d\eta-\omega\wedge\eta$, hence $\Phi=d_\omega\eta$ and $\Phi$ is a coboundary. We obtain:

\begin{proposition}\label{Lichnerowicz}
 Let $\fg$ be a Lie algebra endowed with a lcs structure of the first kind $(\Phi,\omega)$. Then $(\Phi,\omega)$ is exact.
\end{proposition}

The converse is in general false, as the following example shows:

\begin{example}\label{solv_ex_not_first_kind}
 Consider the 4-dimensional solvable Lie algebra $\fg=(12+34,0,-23,0)$, isomorphic to the Lie algebra $\mathfrak{d}_{4,1}$ of the list contained in \cite[Proposition 2.1]{Ovando}. 
Consider the lcs structure on $\fg$ obtained by taking $\Phi=2e^{12}+e^{34}$ and $\omega=e^2$. One checks that $\Phi=d_\omega\eta$, with $\eta=e^1$, hence the structure is exact. 
However, it is not of the first kind. Indeed, the only automorphisms of $(\Phi,\omega)$ are of the form $a e_1$, with $a\in \mathbb{R}$; all of them are sent to zero by the Lee morphism. If 
$G$ denotes the unique simply connected solvable Lie group with Lie algebra $\fg$, then $G$ is endowed with an exact lcs structure which is not of the first kind. Notice that $G$ is not compact.
\end{example}

Consider an exact lcs Lie algebra $(\fg,\Phi,\omega)$ of dimension $2n$ and write $\Phi=d\eta-\omega\wedge\eta$, for some $\eta\in\fg^*$. Then we have
\begin{equation}\label{eq:11}
 0\neq \Phi^n=(d\eta-\omega\wedge\eta)^n=(d\eta)^n+n(d\eta)^{n-1}\wedge\eta\wedge\omega,
\end{equation}
hence $2n-2\leq \mathrm{rank}(d\eta)\leq 2n$. If the rank of $d\eta$ is $2n-2$, then $(\Phi,\omega)$ is of the first kind by the above construction. On the other hand, if the rank of $d\eta$ is $2n$,
the fact that $\wedge^{2n}\fg^*$ is 1-dimensional, together with \eqref{eq:11}, shows that $\Phi^n$ must be a multiple of $(d\eta)^n$. This implies that the volume form is exact. 

Recall that a Lie algebra $\fg$ is \emph{unimodular} if $\mathrm{tr}(\mathrm{ad}_X)=0$ for each $X\in\fg$. An equivalent characterization is given, in terms of the Chevalley-Eilenberg complex 
$(\Lambda^\bullet\fg^*,d)$, by the condition $H^n(\fg;\bR)\neq 0$, where $n=\dim\fg$. Hence, on a unimodular Lie algebra, the volume form can not be exact. Let $G$ be the unique connected simply 
connected Lie group with Lie algebra $\fg$. By \cite[Lemma 6.2]{Miln}, if $G$ admits a discrete subgroup $\Gamma$ with compact quotient, then $\fg$ is unimodular. In particular, if $\fg$ is not 
unimodular, then $G$ does not admit any compact quotient.

Putting all these considerations together, we obtain

\begin{proposition}\label{exact_non_first_kind}
 Let $\fg$ be a unimodular Lie algebra endowed with an exact lcs structure $(\Phi,\omega)$. Then $(\Phi,\omega)$ is of the first kind.
\end{proposition}

In particular, on unimodular Lie algebras a lcs structure is exact if and only if it is of the first kind. Notice that the Lie algebra of Example \ref{solv_ex_not_first_kind} is not unimodular.

%
%

\subsection{Contact structures on Lie algebras}

\begin{definition} Let $\fh$ be a real Lie algebra of dimension $2n-1$. 
A \emph{contact structure} on $\fh$ a 1-form $\theta\in\fh^*$ such that 
\[
(d\theta)^{n-1}\wedge\theta\neq 0.
\]
\end{definition}
Again $d$ denotes the Chevalley-Eilenberg differential on $\Lambda^\bullet\fh^*$. 
If $(\fh, \theta)$ is a contact Lie algebra then there exists a unique vector $R\in \fh$, \emph{the Reeb vector}, which is characterized
by the conditions
\[
i_R(d\theta) = 0 \quad \mathrm{and} \quad i_R(\theta) = 1.
\]

For a contact Lie algebra $(\fh, \theta)$ with Reeb vector $R$, either the center of $\fh$, $\mathcal{Z}(\fh)$, is trivial or $\mathcal{Z}(\fh) = \langle R\rangle$. Contact 
Lie algebras with non-trivial center are in 1-1 correspondence with a particular class of central extensions 
of symplectic Lie algebras. In fact, if $\sigma \in \Lambda^2 \fs^*$ is a symplectic structure on a Lie algebra $\fs$ of dimension $2n-2$, that
is, $\sigma^{n-1}\neq 0$ and $d\sigma = 0$, then one may consider the central extension $\fh = \bR \cext_\sigma \fs$ of $\fs$ by the 2-cocycle $\sigma$. If $\theta \in \fh^*$ is the 1-form on $\fh$ given by
\[
\theta(a, X) = a,
\]
we have that $(\fh, \theta)$ is a contact Lie algebra with Reeb vector $R = (1, 0) \in \fh$.
The converse is also true. Namely, if $\fh$ is a contact Lie algebra with Reeb vector $R$ such that 
$\mathcal{Z}(\fh) = \langle R \rangle$ then the quotient vector space $\fs = \fh/\la R\ran$ is a symplectic Lie algebra and $\fh$ is the central extension of $\fs$ by the symplectic structure (for 
more details, see \cite{Dia1}).

We consider next derivations of contact and symplectic Lie algebras:

\begin{definition} 
Let $(\fh,\theta)$ be a contact Lie algebra and let $D\in\mathrm{Der}(\fh)$ be a derivation. $D$ is called a \emph{contact derivation} if $D^*\theta=0$. 
Let $(\fs,\sigma)$ be a symplectic Lie algebra and let $D\in\mathrm{Der}(\fs)$ be a derivation. $D$ is called a \emph{symplectic derivation} if $\sigma(DX,Y)+\sigma(X,DY)=0$ for every 
$X,Y\in\fs$.
\end{definition}

In order to describe contact derivations on $(\fh,\theta)$, we assume that $\fh$ is the central extension of a symplectic Lie algebra.

\begin{proposition}\label{central-extension}
Let $(\fs, \sigma)$ be a symplectic Lie algebra.
There exists a one-to-one correspondence between symplectic derivations in $\fs$ and contact derivations in the central extension $\fh = \bR \cext_\sigma\fs$. 
In fact, the correspondence is given as follows. If $D_{\fs}\colon \fs \to \fs$ is a symplectic derivation in $\fs$ then $D$ is a contact derivation in $\fh$, where $D$ is defined by
\begin{equation}\label{Def-D}
D(a, X) = (0, D_{\fs}X), \; \; \mbox{ for } (a, X) \in \fh.
\end{equation}
\end{proposition}

\begin{proof} 
Let $D$ be a contact derivation of $\fh = \bR \cext_\sigma\fs$. Then, the condition $D^{*}(1,0) = 0$ implies that
\[
D(a, X) = (0, aZ + D_{\fs}X), \; \; \mbox{ for } (a, X) \in \fh
\]
where $Z$ is a fixed vector of $\fs$ and $D_{\fs}\colon \fs \to \fs$ is a linear map. By \eqref{extension_central}, $[(a, 0), (0, X)]_{\fh}=0$ for every $X\in\fs$, and thus 
\[
0= D[(1, 0), (0, Y)]_{\fh} = [D(1,0), (0, Y)]_{\fh} = (\sigma(Z, Y), [Z, Y]);
\]
from this we deduce that $Z=0$. Thus, $D(a, X) = (0, D_{\fs}X)$. On the other hand, using that
\[
D[(0, X), (0,Y)]_{\fh} = [D(0, X), (0, Y)]_{\fh} + [(0, X), D(0, Y)]_{\fh}, \; \; \mbox{ for } X, Y \in \fs
\]
we conclude that $D_{\fs}$ is a symplectic derivation of the symplectic Lie algebra $(\fs, \sigma)$.

Conversely, suppose that $D_{\fs}\colon \fs \to \fs$ is a symplectic derivation of the symplectic algebra $\fs$ and let $D\colon \fh \to \fh$ be the
linear map given by \eqref{Def-D}. Then, a direct computation, proves that $D$ is a contact derivation of $\fh$.
\end{proof}

%
%
 
\subsection{A correspondence between contact Lie algebras and lcs Lie algebras of the first kind}\label{contact-alg-lcs-alg}

Here we prove that there is a 1-1 correspondence between contact Lie algebras in dimension $2n-1$ endowed with a contact derivation and lcs Lie algebras of the first kind in dimension $2n$. 

\begin{theorem}\label{prop:1}
Let $(\fg,\omega, \eta)$ be a lcs Lie algebra of the first kind of dimension $2n$. Set $\fh=\ker(\omega)$ and 
let $\theta$ be the restriction of $\eta$ to $\fh$. Then $(\fh,\theta)$ is a contact Lie algebra, endowed with a contact derivation $D$ and $\fg$ is isomorphic to the semidirect product $\fh\rtimes_D\bR$. 
In fact, $D$ is induced by the inner derivation $ad_U\colon \fg \to \fg$. Conversely, let $(\fh,\theta)$ be a contact Lie algebra and let $D$ be a contact derivation of $\fh$. 
Then $\fg=\fh\rtimes_D\bR$ is endowed with a lcs structure of the first kind.
\end{theorem}

\begin{proof}
If $X,Y\in\fg$, then
\[
\omega([X,Y])=-d\omega(X,Y)=0,
\]
since $\omega$ is closed. This means that the subalgebra $[\fg,\fg]$ is contained in $\fh$, and $\fh$ is an ideal of $\fg$. Let $U$ be the anti-Lee vector of the lcs structure 
$(\omega, \eta)$ on $\fg$ and denote by $\theta$ the restriction of $\eta$ to $\fh$. Using that $\omega \wedge \eta\wedge (d\theta)^{n-1} \neq 0$ and the fact that $i_U\eta = 0$ and 
$i_U d\eta = 0$, we conclude that
$(\fh,\theta)$ is a contact Lie algebra. Define a linear map $D$ on $\fh$ by
\begin{equation}\label{Def-D1}
D(X)=\mathrm{ad}_U(X).
\end{equation}
Since $\fh$ contains the commutator $[\fg,\fg]$, indeed $D\colon \fh\to\fh$, and $D$ is derivation of $\fh$ by the Jacobi identity in $\fg$. Furthermore, 
$D$ is a contact derivation of $(\fh,\theta)$:
\[
(D^*\theta)(X)=\theta(D(X))=\theta([U,X])=\eta([U,X])=- d\eta(U,X)=-\imath_U d \eta(X)=0.
\]
Consider the linear isomorphism $\varphi\colon \fg \to \fh\rtimes_D\bR$ given by
\[
\varphi (X) = ( X -\omega(X) U,\omega(X)), \mbox{ for } X \in \fg.
\]
Note that 
\[
\varphi^{-1}(X,a) = aU + X, \; \; \mbox{ for } (X,a)\in \fh\rtimes_D\bR.
\]
Therefore, from \eqref{Def-D1}, we conclude that $\varphi$ is a Lie algebra isomorphism between $\fg$ and $\fh\rtimes_D\bR$.

Conversely, let us start with a contact Lie algebra $(\fh,\theta)$ of dimension $2n-1$ and a contact derivation $D$. 
Set $\fg=\fh\rtimes_D\bR$ and write a vector in $\fg$ as $(X,a)$, with $a\in\bR$. Use \eqref{extension_derivation} to define a Lie algebra structure on $\fg$. Take $U = (0,1) \in \fg$. 
Then,
\[
[U,(X,0)]_{\fg}=(D(X),0),
\]
and $D$ can be identified with the adjoint action of $U$ on $\fg$. Define $\omega\in\fg^*$ by $\omega|_{\fh}=0$, $\omega(U)=1$, 
and let $\eta\in\fg^*$ be the extension of the contact form $\theta$ to $\fg$ obtained by setting $\eta(U)=0$. To prove that $\fg$ is lcs of the first kind, it is enough to show that $d\omega=0$ 
and that $\omega\wedge\eta\wedge(d\eta)^{n-1} \neq 0$. Taking vectors $(X,a),(Y,b)\in\fg$, we have
\[
d\omega((X,a),(Y,b))=-\omega([(X,a),(Y,b)]_{\fg})=-\omega((aD(Y)-bD(X)+[X,Y]_{\fh},0))=0,
\]
since $\omega|_{\fh}=0$. We are left with showing that the rank of $d\eta$ is $2n-2$. One has
\begin{align*}
\imath_U d \eta((X,a))&=d\eta((0,1),(X,a))=-\eta([(0,1),(X,a)]_{\fg})=-\eta((D(X),0))=-\theta(D(X))=\\
&=-(D^*\theta)(X)=0,
\end{align*}
since $D$ is a contact derivation. This shows that $d\eta$ has a kernel, hence its rank cannot be $2n$. On the other hand, $d\eta=d\theta$ on $\fh=\ker(\omega)$, hence the rank
of $d\eta$ is indeed $2n-2$. Moreover, if $V$ is the Reeb vector of $\fh$ then
\[
\imath_V\imath_U(\omega\wedge\eta\wedge(d\eta)^{n-1}) = \imath_V(\eta\wedge(d\eta)^{n-1}) = (d\eta)^{n-1} \neq 0
\]
which implies that
\[
\omega\wedge\eta\wedge(d\eta)^{n-1} \neq 0.
\]
The lcs structure of the first kind is then obtained by setting $\Phi=d\eta-\omega\wedge\eta$.
\end{proof}

\begin{remark}
\thmref{prop:1} can be interpreted as an algebraic analogue of \thmref{global-descrip-lcs}.
\end{remark}

\subsection{Symplectic Lie algebras and a particular class of lcs Lie algebras of the first kind}\label{sym-Lie-alg-lcs-Lie}

In this section, we will consider a special class of lcs Lie algebras of the first kind, those with central Lee vector. We will see that they are closely related with symplectic Lie algebras. In fact, we have the following result 
\begin{theorem}\label{central-lcs}
There exists a one-to-one correspondence between lcs Lie algebras of the first kind $(\fg, \omega, \eta)$ of dimension $2n+2$ with central Lee vector and symplectic Lie algebras $(\fs, \sigma)$ of 
dimension $2n$ endowed with a symplectic derivation. In fact, this correspondence is defined as follows. Let $(\fs, \sigma)$ be a symplectic Lie algebra and 
$D_{\fs}\colon \fs \to \fs$ a symplectic derivation in $\fs$. On the vector space $\fg = \bR \oplus \fs\oplus\bR $ we can consider the Lie bracket
\begin{equation}\label{double_extension}
[(a, X, a'), (b, Y, b')]_{\fg} = (\sigma(X, Y),a'D_{\fs}Y - b'D_{\fs}X + [X, Y]_{\fs},0 ) 
\end{equation}
and the 1-forms $\omega,\eta\in\fg^*$ given by
\begin{equation}\label{lcs_structure_double}
\omega(a, X, a') = a', \; \; \eta(a, X, a') = a
\end{equation}
for $(a, X, a'), (b, Y, b') \in \fg$. Then $(\omega,\eta)$ is a lcs structure of the first kind on the Lie algebra $(\fg, [\cdot, \cdot]_\fg)$ with central Lee vector $V = (1, 0, 0) \in 
\bR \oplus \fs\oplus\bR $ and anti-Lee vector $U = (0, 0, 1) \in \bR \oplus \fs\oplus\bR $.
\end{theorem}
\begin{proof}
Let $(\fs, \sigma)$ be a symplectic Lie algebra,
$D_\fs\colon \fs \to \fs$ a symplectic derivation and $[\cdot, \cdot]_\fg$ (resp.\ $\omega$ and $\eta$) the bracket (resp.\ the $1$-forms) on $\fg$ 
defined by \eqref{double_extension} (resp.\ by \eqref{lcs_structure_double}). Then, using Proposition \ref{central-extension} and \thmref{prop:1}, we deduce that $(\fg, 
[\cdot, \cdot]_\fg)$ is a lcs Lie algebra of the first kind with lcs structure $(\omega, \eta)$.
 
Conversely, suppose that $(\omega, \eta)$ is a lcs structure of the first kind on a Lie algebra $\fg$ of dimension $2n+2$ and that the Lee vector $V$ belongs to $\mathcal{Z}(\fg)$. Denote by 
$(\fh,\theta)$ the contact Lie subalgebra associated with $\fg$ and by $D$ the corresponding contact derivation. Then, $V$ is the Reeb vector $R$ of $\fh$ and, therefore, $R \in \mathcal{Z}(\fh)$. 
This implies that the quotient vector space $\fs = \fh / \langle R\rangle $ is a symplectic Lie algebra with symplectic structure $\sigma$ and that $\fh$ is the central extension of $\fs$ by 
$\sigma$. Furthermore, using Proposition \ref{central-extension}, we have that $D$ may be given in terms of a symplectic derivation $D_\fs$ on $\fs$ and, in addition, the Lie algebra structure and 
the lcs structure on $\fg$ are given by \eqref{double_extension} and \eqref{lcs_structure_double}, respectively.
\end{proof}

Now, we may introduce the following definition.
\begin{definition}\label{lcs-extension-algebraic}
The vector space $\fg = \bR \oplus \fs\oplus\bR $ in Theorem \ref{central-lcs} endowed with the Lie algebra structure \eqref{double_extension} and the lcs structure of the first kind 
\eqref{lcs_structure_double} is called the \emph{lcs extension} of the symplectic Lie algebra $(\fs, \sigma)$ by the derivation $D_{\fs}$.
\end{definition}

\begin{remark}\label{lcs_extension}
 With the notation introduced in Sections \ref{multiplicative} and \ref{central-ext-alg-group} the lcs extension of $(\fs, \sigma)$ by $D_{\fs}$ can be denoted $(\bR\cext_\sigma\fs)\rtimes_D\bR$, where $D$ is the contact 
 derivation of $\bR\cext_\sigma\fs$ determined by $D_\fs$.
\end{remark}

In some cases, the symplectic Lie algebra in \thmref{central-lcs} may in turn be obtained as a symplectic double extension of another symplectic Lie algebra whose dimension is $\dim \; \fs - 2$.

We recall the construction of the double extension $(\fs, \sigma)$ of a symplectic Lie algebra $(\fs_1, \sigma_1)$ by a derivation $D_{\fs_1}$ and an element $Z_1 \in \fs_1$ (for more 
details, see \cite{DaMe,MeRe}).

It is clear that the $2$-form $D^*_{\fs_1}\sigma_1$ on $\fs_1$ given by
\begin{equation}\label{D-s1-star}
(D^*_{\fs_1}\sigma_1)(X_1, Y_1) = \sigma_1(D_{\fs_1}X_1, Y_1) + \sigma_1(X_1,D_{\fs_1}Y_1),
\end{equation}
is a $2$-cocycle. So, we can consider the central extension $\fh_1 = \bR\cext_{D^*_{\fs_1}\sigma_1}\fs_1$ with bracket given by \eqref{extension_central}.

Now let $(-i_{Z_1}\sigma_1, -D_{\fs_1})\colon \fh_1 \to \fh_1$ be the linear map given by
\[
(-i_{Z_1}\sigma_1, -D_{\fs_1})(a,X_1) = (-\sigma_1(Z_1, X_1), -D_{\fs_1}X_1);
\]
then $(-i_{Z_1}\sigma_1, -D_{\fs_1})$ is a derivation of $\fh_1$ if and only if
\begin{equation}\label{der-ext-central}
d(i_{Z_1}\sigma_1) = -(D_{\fs_{1}}^*)^2\sigma_1.
\end{equation}
Assuming that $(-i_{Z_1}\sigma_1, -D_{\fs_1})$ is a derivation, we can consider the vector space $\fs = \fh_1\oplus\bR$ with the Lie algebra structure
\begin{align*}
[(a_1, X_1,a'_1), (b_1, Y_1,b'_1)]_{\fs} = \ & 
((D_{\fs_1}^*\sigma_1)(X_1,Y_1)-a'_1\sigma_1(Z_1, Y_1) + b'_1\sigma_1(Z_1, X_1),\\
& -a'_1D_{\fs_1}Y_1 + b'_1D_{\fs_1}X_1 + [X_1, Y_1]_{\fs_1},0),
\end{align*}
that is, $\fs$ is the semidirect product $\fh_1\rtimes_{(-i_{Z_1}\sigma_1, -D_{\fs_1})}\bR$. Moreover, the 2-form $\sigma\colon \fs \times \fs \to \mathbb{R}$ on $\fs$ defined by
\begin{equation*}\label{sym-str-double-ext}
\sigma((a_1, X_1,a'_1), (b_1, Y_1,b'_1)) = a_1b'_1 - a'_1 b_1 + \sigma_1(X_1, Y_1)
\end{equation*}
for $(a_1, X_1,a'_1), (b_1, Y_1,b'_1) \in \fs$, is a symplectic structure on $\fs$. The symplectic Lie algebra $(\fs, \sigma)$ is the \emph{double extension} of $(\fs_1, \sigma_1)$ by $D_{\fs_1}$ 
and $Z_1$ (see \cite{DaMe,MeRe}).

Now, in the presence of a symplectic derivation $D_{\fs}$ of $(\fs, \sigma)$, we can use Theorem \ref{central-lcs}. In fact, we have the following result
\begin{corollary}\label{sym-double-lcs-ext}
Assume that we have the following data:
\begin{enumerate}
\item
a symplectic Lie algebra $(\fs_1, \sigma_1)$ of dimension $2n-2$;
\item
a derivation $D_{\fs_1}$ of $\fs_1$ and an element $Z_1 \in \fs_1$ such that \eqref{der-ext-central} holds and
\item
a symplectic derivation $D_\fs$ of the symplectic double extension $(\fs, \sigma)$ of $(\fs_1, \sigma_1)$ by $D_{\fs_1}$ and $Z_1$.
\end{enumerate}
Then the vector space $\fg =(\mathbb{R} \oplus \fs) \oplus \mathbb{R}$ endowed with the Lie algebra structure given by \eqref{double_extension} and the lcs structure of the first kind given by 
\eqref{lcs_structure_double} is a lcs Lie algebra of the first kind of dimension $2n+2$ with central Lee vector.
\end{corollary}

\subsection{Lcs structures on nilpotent Lie algebras}

In this section we focus on lcs structures on nilpotent Lie algebras. Our first result is:

\begin{theorem}\label{first-descrip-lcs-nilpotent}
Let $\fs$ be a nilpotent Lie algebra endowed with a symplectic structure and let $\fg$ be a nilpotent Lie algebra endowed with a lcs structure whose Lee form is non-zero.
\begin{enumerate}
\item
The lcs extension of $\fs$ by a symplectic nilpotent derivation is a nilpotent Lie algebra endowed with a lcs structure of the first kind with central Lee vector.
\item
The lcs structure on $\fg$ is of the first kind, the Lee vector is central and $\fg$ is the lcs extension of a symplectic nilpotent Lie algebra by a symplectic nilpotent derivation.
\end{enumerate}
\end{theorem}
\begin{proof}
To 1. The central extension of a nilpotent Lie algebra by a $2$-cocycle is again nilpotent. Thus, if $(\fs,\sigma)$ is a symplectic nilpotent Lie algebra, $\fh = \mathbb{R} \cext_\sigma \fs $ is 
a contact nilpotent Lie algebra. Moreover, if $D_\fs$ is a symplectic nilpotent derivation of $\fs$ and $D$ is the corresponding contact derivation of $\fh$ then, from (\ref{Def-D}), it follows that 
$D$ is nilpotent. Therefore, $\fh\rtimes_D\bR$ (that is, the lcs extension of $\fs$ by $D_\fs$), is also a nilpotent Lie algebra. Finally, 
from Theorem \ref{central-lcs}, we have that the Lee vector of the lcs structure on $\fg$ is central.

To 2. From Corollary \ref{cor:Dixmier} and Proposition \ref{exact_non_first_kind}, we deduce that the lcs structure on $\fg$ is of the first kind. Now, let $V$ be the Lee vector of $\fg$. By \thmref{prop:1}, 
the Lie subalgebra $\fh = \ker(\omega)$ is a contact Lie algebra and $V$ is just the Reeb vector $R$ of $\fh$. In addition, since $\fg$ is nilpotent, so is $\fh$ 
and we have that $\mathcal{Z}(\fh) = \langle R \rangle$ (see \cite{Dia,Dia1}). This, together with \eqref{U-V-flat}, implies that $V \in \mathcal{Z}(\fg)$.

Thus, from Theorem \ref{central-lcs}, it follows that $\fg$ is the lcs extension of a symplectic Lie algebra $\fs$ by symplectic derivation $D_\fs$ on $\fs$. In fact, using 
Theorems \ref{prop:1} and \ref{central-lcs}, we deduce that $\fs$ is the quotient Lie algebra $\fh / \langle R \rangle \simeq \fg / \langle U, V\rangle$, with $U$ the anti-Lee vector of 
$\fg$, and $D_\fs$ is the derivation on $\fs$ induced by the operator $\mathrm{ad}_U\colon \fg \to \fg$. So, $\fs$ is nilpotent and, using that the endomorphism $\mathrm{ad}_U$ is nilpotent, we 
conclude that the derivation $D_\fs$ also is nilpotent.
\end{proof}

On the other hand, one may prove that the symplectic double extension of a nilpotent Lie algebra by a nilpotent derivation is a symplectic nilpotent Lie algebra. In fact, in \cite{MeRe} (see also 
\cite{DaMe}), the authors prove that every symplectic nilpotent Lie algebra of dimension $2n$ may be obtained by a sequence of $n-1$ symplectic double extensions by nilpotent derivations from the 
abelian Lie algebra of dimension $2$. Hence, using these facts and Theorem \ref{first-descrip-lcs-nilpotent}, we deduce the following result
\begin{theorem}\label{description-Lie-alg-nil}
\begin{enumerate}
\item
Under the same hypotheses as in Corollary \ref{sym-double-lcs-ext} if, in addition, the derivations $D_{\fs_1}$ and $D_\fs$ on the symplectic nilpotent Lie algebras $\fs_1$ and $\fs$, respectively, 
are nilpotent then the lcs Lie algebra $(\bR\oplus \fs)\oplus\bR$ is also nilpotent.
\item
Every lcs nilpotent Lie algebra of dimension $2n+2$ with non-zero Lee $1$-form may be obtained as the lcs extension of a $2n$-dimensional symplectic nilpotent Lie algebra $\fs$ by a 
symplectic nilpotent derivation and, in turn, the symplectic nilpotent Lie algebra $\fs$ may obtained by a sequence of $n-1$ symplectic double extensions by nilpotent derivations from the abelian Lie 
algebra of dimension $2$.
\end{enumerate}
\end{theorem}

For the next result, we recall the notion of characteristic filtration of a nilpotent Lie algebra (see \cite{BM,Sal}). Let $\fg$ be a nilpotent Lie algebra and let $(\Lambda^\bullet\fg^*,d)$ 
be the associated Chevalley-Eilenberg complex. Consider the following subspaces of $\fg^*$:
\begin{equation}\label{char_fil}
W_1=\ker(d), \qquad W_k=d^{-1}(\wedge^2 W_{k-1}), \ k\geq 2.
\end{equation}
It is immediate to see that $W_{k-1}\subset W_k$, hence $\{W_k\}_k$ is a filtration of $\fg^*$, intrinsically defined. The nilpotency of $\fg$ implies that there exists $m\in\bN$ such that 
$W_m=\fg^*$. 
If $W_m=\fg^*$ but $W_{m-1}\neq\fg^*$, one says that $\fg$ is $m$-step nilpotent. In particular, $1$-step nilpotent Lie algebras are abelian.
\begin{definition}
 Let $\fg$ be a nilpotent Lie algebra. The filtration $\{W_k\}_k$ of $\fg^*$, defined by \eqref{char_fil}, is the \emph{characteristic filtration} of $\fg^*$.
\end{definition}
Let $\fg$ be a nilpotent Lie algebra and let $\{W_k\}_k$ be the characteristic filtration of $\fg^*$. Define
\[
F_1=W_1, \qquad F_k=W_k/W_{k-1}, \ k\geq 2.
\]
Clearly one has $\fg^*=\oplus_kF_k$, but the splitting is not canonical. Nevertheless, the numbers $f_k=\dim(F_k)$ are invariants of $\fg^*$.

\begin{proposition}\label{prop:char_filtration}
Let $\fg$ be an $m$-step nilpotent Lie algebra of dimension $2n$. Assume $\fg$ is endowed with a lcs structure. Then $f_m=1$.
\end{proposition}
\begin{proof}
By \thmref{first-descrip-lcs-nilpotent}, the lcs structure is of the first kind; we denote it $(\omega,\eta)$. It is sufficient to prove that, if 
$f_m\geq 2$, the Lie algebra $\fg$ can not admit any lcs structure of the first kind. Let 
$\mathcal{B}=\langle e^1_1,\ldots,e^{m-1}_{2n-i},e^m_{2n-i+1},\ldots,e^m_{2n-1},e^m_{2n}\rangle$ be a basis of $\fg^*$ adapted to the filtration $\{W_k\}$. By definition, this means that the collection 
$\{e^k_j\}_j$ spans $F_k$. Assume that $f_m\geq 2$; hence, $i\geq 2$. Set for convenience $y=e^m_{2n-1}$ and $z=e^m_{2n}$. 
Since $d\omega=0$, by making a suitable change of variables in $F_1$, we can assume that $\omega$ is one of the generators of $F_1$, indeed we can take $\omega=e^1_1$.
Let $\fh=\ker(\omega)$; then $\fh$ is a nilpotent Lie algebra, the characteristic filtration of $\fh^*$ is $\widetilde{W}_1=W_1/\langle\omega\rangle$, $\widetilde{W}_k=W_k$, $k\geq 2$ and
$\widetilde{\mathcal{B}}=\langle e^1_2,\ldots,e^{m-1}_{2n-i},e^m_{2n-i+1},\ldots,e^m_{2n-2},y,z\rangle$ is a basis of $\fh^*$ adapted to $\{\widetilde {W}_k\}$. By \thmref{prop:1}, $\fh$ is a contact 
Lie algebra, with contact form $\theta$ given by the restriction of $\eta$ to $\fh$. In particular the rank of $d\theta$ is maximal on $\ker(\theta)$. This means that if $X\in\ker(\theta)$, 
$\imath_Xd\theta\neq 0$. Expand $\theta$ on the basis $\widetilde{\mathcal{B}}$,
\[
\theta=a_2e^1_2+\ldots+\ldots+a_{2n-2}e^m_{2n-2}+by+cz, \quad a_j,b,c\in\bR.
\]
Let $\langle X_2,\ldots,X_{2n-2},Y,Z\rangle$ be the basis of $\fh$ dual to $\widetilde{\mathcal{B}}$. The vector $T=cY-bZ$ clearly belongs to $\ker(\theta)$. 
By hypothesis, $dy,dz\in\Lambda^2\widetilde{W}_{m-1}$, and $T\in\ker(e^k_{\ell}) \ \forall k,\ell$, hence $\imath_Td\theta=0$. Thus $\theta\wedge(d\theta)^{n-1}$ has a kernel, contradicting 
the fact that $(\fh,\theta)$ is a contact Lie algebra.
\end{proof}

To conclude this section, we use the classification of nilpotent Lie algebras in dimension 4 and 6 to determine which of them carry a locally conformal symplectic structure. We refer to Section 
\ref{Notation_Lie_algebras} for the notation.

\begin{proposition}\label{4_dim_nilpotent}
Suppose $(\fg,\omega,\eta)$ is a nilpotent Lie algebra of dimension 4 endowed with a lcs structure. Then $\fg$ is isomorphic to one of the following Lie algebras:
\begin{itemize}
\item $\fg_1=(0,0,0,12)$, $\omega=e^3$, $\eta=e^4$;
\item $\fg_2=(0,0,12,13)$, $\omega=e^2$, $\eta=e^4$.
\end{itemize}
\end{proposition}
\begin{proof}
There are three isomorphism classes of nilpotent Lie algebras in dimension 4 (see \cite{BM} for instance): the two listed above and the abelian one, which is clearly not lcs.
\end{proof}

\begin{remark}
Both Lie algebras in Proposition \ref{4_dim_nilpotent} are symplectic. $\fg_1$ admits a complex structure and is the Lie algebra of the so called \emph{Kodaira-Thurston} nilmanifold, that was mentioned in the 
introduction. $\fg_2$ does not admit any complex structure (see \cite{Sal}) and is the Lie algebra of the 4-dimensional nilpotent Lie group $G$ considered in Section \ref{Dimension:4}.
\end{remark}

\begin{proposition}\label{6_dim_nilpotent}
Suppose $(\fg,\omega,\eta)$ is a nilpotent Lie algebra of dimension 6 endowed with a lcs structure. Then $\fg$ is isomorphic to one of the Lie algebras contained in 
Table \ref{table:1}.
\end{proposition}
The third and fourth column in Table \ref{table:1} contain two labellings of nilpotent Lie algebras, according to \cite{BM} and \cite{CFGU} respectively. We also 
compare lcs structures with other structures on 6-dimensional nilpotent Lie algebras, namely symplectic and complex structures.

\begin{table}[h!]
\centering
\caption{Locally conformal symplectic nilpotent Lie algebras in dimension 6}\label{table:1}
{\tabulinesep=1.2mm
\begin{tabu}{lcccccccc}
\toprule[1.5pt]
Lie algebra & $\omega$ & $\eta$ & \cite{BM} & \cite{CFGU} & Symplectic & Complex \\
\specialrule{1pt}{0pt}{0pt}
$(0,0,0,0,0,12+34)$ & $e^5$ & $e^6$ & $L_{5,1}\oplus A_1$ & $\fh_3$ & $\times$ & \checkmark\\
\specialrule{1pt}{0pt}{0pt}
$(0,0,0,0,12,15+34)$ & $e^2$ & $e^6$ & $L_{6,3}$ & $\fh_{20}$ & $\times$ & $\times$\\
\specialrule{1pt}{0pt}{0pt}
$(0,0,0,0,12,15+23)$ & $e^4$ & $e^6$ & $L_{5,3}\oplus A_1$ & $\fh_9$ & \checkmark & \checkmark\\
\specialrule{1pt}{0pt}{0pt}
$(0,0,0,12,13,15+24)$ & $e^3$ & $e^6$ & $L_{6,7}$ & $\fh_{18}$ & $\times$ & $\times$\\
\specialrule{1pt}{0pt}{0pt}
$(0,0,0,12,13,24+35)$ & $e^1$ &$e^6$ & $L_{6,8}^+$ & $\fh_{19}^-$ & $\times$ & \checkmark\\
\specialrule{1pt}{0pt}{0pt}
$(0,0,0,12,13,24-35)$ & $e^1$ & $e^6$ & $L_{6,8}^-$ & $\fh_{19}^+$ & $\times$ & $\times$ \\
\specialrule{1pt}{0pt}{0pt}
$(0,0,0,12,14,15+24)$ & $e^3$ & $e^6$ & $L_{5,6}\oplus A_1$ & $\fh_{22}$ & \checkmark & $\times$\\
\specialrule{1pt}{0pt}{0pt}
$(0,0,0,12,14,15+23+24)$ & $e^3$ & $e^6$ & $L_{6,14}$ & $\fh_{24}$ & \checkmark & $\times$\\
\specialrule{1pt}{0pt}{0pt}
$(0,0,0,12,14+23,15-34)$ & $e^2$ & $e^6$ & $L_{6,15}$ & $\fh_{27}$ & \checkmark & $\times$\\
\specialrule{1pt}{0pt}{0pt}
$(0,0,12,13,14,25-34)$ & $e^1$ & $e^6$ & $L_{6,20}$ & $\fh_{31}$ & $\times$ & $\times$ \\
\specialrule{1pt}{0pt}{0pt}
$(0,0,12,13,14+23,25-34)$ & $e^1$ & $e^6$ & $L_{6,22}$ & $\fh_{32}$ & $\times$ & $\times$ \\
\bottomrule[1pt]
\end{tabu}}
\end{table}

\begin{proof}
By Proposition \ref{prop:char_filtration}, we can restrict to those Lie algebras whose characteristic filtration ends with a one-dimensional space. 
It is then enough to show that if $\fg$ is one of the 6-dimensional nilpotent Lie algebras which does not appear in the above list, there exists no contact ideal of $\fg$ and
then apply \thmref{prop:1}. 

As an example, consider the Lie algebra $L_{6,13}=(0,0,0,12,14,15+23)$ (in the notation of \cite{BM}). Then $L_{6,13}$ has a basis $\cB=\langle e^1,\ldots,e^6\rangle$ whose closed 
elements are $e^1$, $e^2$ and $e^3$. The generic closed element has the form $\omega=a_1e^1+a_2e^2+a_3e^3$, hence $\fh=\ker(\omega)$ has a basis 
\[
\langle a_2e_1-a_1e_2, a_3e_1-a_1e_3,e_4,e_5,e_6\rangle\eqcolon\langle f_1,f_2,f_3,f_4,f_5\rangle, 
\]
where $\langle e_1,\ldots,e_6\rangle$ is the basis dual to $\cB$. A computation shows that, with respect to the basis $\langle f_1,f_2,f_3,f_4,f_5\rangle$ basis and its dual basis, 
the Chevalley-Eilenberg differential in $\fh^*$ is given by
\[
 df^1=0=df^2, \quad df^3=a_1a_3 f^{12}, \quad df^4=a_2 f^{13}+a_3f^{23} \quad \mathrm{and} \quad df^5=a_1a_2 f^{12}+a_2f^{14}+a_3f^{24}.
\]
We proceed case by case:
\begin{itemize}
 \item if $a_1=0$ or $a_3=0$, a change of basis shows that $\fh$ is isomorphic to $(0,0,0,12,14)$, which is not a contact algebra;
 \item if $a_2=0$, a change of basis shows that $\fh$ is isomorphic to $(0,0,12,13,14)$, which is not a contact algebra;
 \item if $a_1a_2a_3 \neq 0$, a change of basis shows that $\fh$ is isomorphic to $(0,0,12,13,14)$ as well.
\end{itemize}
We conclude that $L_{6,13}$ has no suitable contact ideal and, by \thmref{prop:1}, it is not locally conformal symplectic.
\end{proof}

According to \cite{Sal}, the five nilpotent Lie algebras in Table \ref{table:1} that admit neither a symplectic nor a complex structure are the only such examples in dimension 6. 
Since all of them admit a lcs structure, we obtain the following result:

\begin{corollary}
A 6-dimensional nilpotent Lie algebra carries at least one among symplectic, complex and locally conformal symplectic structures.
\end{corollary}

%
%

\section{Locally conformal symplectic Lie groups}\label{sec:lcs_Lie_groups}

First of all, we will introduce, in a natural way, the notion of a locally conformal symplectic (lcs) Lie group.
\begin{definition}
A Lie group $G$ of dimension $2n$ ($n \geq 2$) is said to be a \emph{locally conformal symplectic (lcs) Lie group} if it admits a closed left-invariant $1$-form $\lvec{\omega}$ and a non-degenerate left-invariant $2$-form 
$\lvec{\Phi}$ such that
\[
d\lvec{\Phi} = \lvec{\omega} \wedge \lvec{\Phi}.
\]
The lcs structure is said to be \emph{of the first kind} if $G$ admits a left-invariant $1$-form $\lvec{\eta}$ such that
$\Phi = d\lvec{\eta} + \lvec{\eta} \wedge \lvec{\omega}$ and $\lvec{\omega} \wedge \lvec{\eta} \wedge (d\lvec{\eta})^{n-1}$ is a volume form on $G$.
\end{definition}
If $\fe$ is the identity element of a Lie group $G$ and $\fg = T_{\fe}G$ is the Lie algebra of $G$ then it is clear that $G$ is a lcs Lie group (resp.\ lcs Lie group of the first kind) if and only if 
$\fg$ is a lcs Lie algebra (resp. lcs Lie algebra of the first kind). In fact, if $(\lvec{\omega}, \lvec{\Phi})$ is a lcs structure on $G$ then $(\omega, \Phi)$ is a lcs structure on $\fg$, with
\[
\omega = \lvec{\omega}(\fe), \; \; \Phi = \lvec{\Phi}(\fe).
\]
In a similar way, if $(\lvec{\omega}, \lvec{\eta})$ is a lcs structure of the first kind on $G$ then $(\omega, \eta)$ is a lcs structure of the first kind on $\fg$, with
\[
\omega = \lvec{\omega}(\fe), \; \; \eta = \lvec{\eta}(\fe).
\]

Next, we will restrict our attention to lcs Lie groups of the first kind. More precisely,
we will discuss the relation between lcs Lie groups of the first kind and contact (resp.\ symplectic) Lie groups. 

\subsection{Relation with contact Lie groups}\label{relacion-contact}

It is well-known that a Lie group $H$ of dimension $2n-1$ is a contact Lie group if it admits a contact left-invariant $1$-form $\lvec{\eta}$, that is, $\lvec{\eta} \wedge (d\lvec{\eta})^{n-1}$ is a 
volume form on $H$ (see, for instance, \cite{Dia,Dia1}).

As in the lcs case, a Lie group $H$ with Lie algebra $\fh$ is a contact Lie group if and only if $\fh$ is a contact Lie algebra. In fact, if $\lvec{\eta}$ is a left-invariant contact $1$-form on $H$ 
then $\eta = \lvec{\eta}(\fe)$ is a contact structure on $\fh$.

Next, using contact Lie groups, we will present the typical example of a lcs Lie group of the first kind.

\begin{example}\label{lcs-Lie-group-first-kind}
Let $(H, \lvec{\eta})$ be a contact Lie group of dimension $2n-1$ and $\phi\colon \mathbb{R} \to \Aut(H)$ the flow of a contact multiplicative vector field $\cM$ on $H$. In other words,
\[
\phi\colon\mathbb{R} \to \Aut(H)
\]
is a representation of the abelian Lie group $\mathbb{R}$ on $H$ and
\begin{equation}\label{contact-flow}
\phi_t^*(\lvec{\eta}) = \lvec{\eta}, \; \; \mbox{ for } t \in \mathbb{R}.
\end{equation}
Then, we can consider the semidirect product $G = H \rtimes_\phi\mathbb{R}$ whose Lie algebra $\fg$ is the semidirect product $\fg = \fh \rtimes_{D} \mathbb{R}$, 
with $D\colon\fh\to\fh$ the derivation in $\fh$ induced by the representation $\phi$ (see Section \ref{multiplicative}). Denote by $\eta$ the contact structure on the Lie algebra $\fh$, that is, $\eta = \lvec{\eta}(\fe)$, 
${\fe}$ being the identity element in $H$. Then, from \eqref{contact-flow}, it follows that ${\mathcal L}_{\mathcal M}\lvec{\eta} = 0$, which implies that
\[
D^*\eta = 0,
\]
that is, $D$ is a contact derivation. Thus, if $\omega = (0, 1) \in \fh^* \oplus\mathbb{R} = \fg^*$ then, using \thmref{prop:1}, we deduce that the couple $(\omega, \eta)$ is a lcs structure of the first kind on $\fg$. 
Therefore, $(\lvec{\omega}, \lvec{\eta})$ is a lcs structure of the first kind on $G = H \rtimes_\phi\mathbb{R}$. 
\end{example}
\begin{remark}
Note that if $\lvec{R}$ is the Reeb vector field of $H$, with $R\in \fh$, then $(\lvec{R}, 0)$ is just the Lee vector field of $G = H \rtimes_\phi\mathbb{R}$.
\end{remark}
The previous example is the model of a connected simply connected lcs Lie group of the first kind. More precisely, two lcs Lie groups of the first kind $(G, \omega, \eta)$ and $(G',\omega',\eta')$ are isomorphic if 
there exists a Lie group isomorphism $\Psi\colon G \to G'$ such that
\[
\Psi^*\omega' = \omega \; \; \; \mbox{ and } \; \; \; \Psi^*\eta' = \eta.
\]
Then, one may prove the following result
\begin{theorem}\label{Universal-covering-lcs}
Let $G$ be a connected simply connected lcs Lie group of the first kind. If $\fg$ is the Lie algebra of $G$ then $\fg$ is isomorphic to a semidirect product of a contact Lie algebra $\fh$ with $\mathbb{R}$. Moreover, 
if $H$ is a connected simply connected Lie group with Lie algebra $\fh$, then $G$ is isomorphic to a semidirect product $H \rtimes_\phi\mathbb{R}$, where $\phi\colon\mathbb{R} \to\Aut (H)$ is a contact representation 
of the abelian Lie group $\mathbb{R}$ on $H$ and the lcs structure on $H \rtimes_\phi\mathbb{R}$ is given as in 
Example \ref{lcs-Lie-group-first-kind}.
\end{theorem}
\begin{proof}
Let $\fg$ be the Lie algebra of $G$ and $(\omega, \eta)$ the lcs structure of the first kind on $\fg$. Then, using \thmref{prop:1}, we deduce the following facts:
\begin{itemize}
\item $\eta$ induces a contact structure on the ideal $\fh=\ker(\omega)\subset\fg$;
\item there exists a contact derivation $D\colon\fh \to \fh$ such that $\fg$ is isomorphic to the semidirect product $\fh\rtimes_{D} \mathbb{R}$ and
\item under the previous isomorphism $\omega$ is the $1$-form $(0, 1)$ in $\fh^* \oplus\mathbb{R}$.
\end{itemize}
Now, let $H$ be a connected simply connected Lie group with Lie algebra $\fh$. Then, $H$ is a contact Lie group. In addition, the derivation $D\colon\fh\to\fh$ induces a representation
\[
\phi\colon \mathbb{R} \to\Aut(H).
\]
In fact, if $\cM$ is the multiplicative vector field on $H$ whose flow is $\phi$ then, using \eqref{der-induced} and the fact that $D$ is a contact derivation, we deduce that
\[
{\mathcal L}_{\cM}\lvec{\eta} = 0,
\]
which implies that $\phi$ is a contact representation, that is, 
\[
\phi_t^*\lvec{\eta} = \lvec{\eta}, \; \; \; \mbox{ for } t \in \mathbb{R}.
\]
On the other hand, it is clear that the Lie algebra of the connected simply connected Lie group $H \rtimes_\phi\mathbb{R}$ is $\fh\rtimes_{D} \mathbb{R} \simeq \fg$. This ends the proof of the result.
\end{proof}
\begin{remark}
Let $\Gamma$ be a lattice in the contact Lie group $H$ and let $r\bZ$ be an integer lattice of $\mathbb{R}$ such that the restriction to $r\bZ$ of 
the representation $\phi$ (in Theorem \ref{Universal-covering-lcs}) takes values in $\Aut(\Gamma)$. Then $\Gamma \rtimes_\phi(r\bZ)$ is a lattice in $H \rtimes_\phi\mathbb{R}$ 
and $M = (\Gamma \rtimes_\phi(r\bZ)) \backslash (H \rtimes_\phi\mathbb{R})$ is a compact manifold. Furthermore, since the lcs structure of the first kind on $H \rtimes_\phi\mathbb{R}$ is 
left-invariant, it induces a lcs structure of the first kind on $M$. Thus, $M$ is a compact lcs manifold of the first kind.
This construction can also be described as a mapping torus. More precisely, notice that the compact quotient $N\coloneq\Gamma\backslash H$ carries a contact structure; moreover, $\bar{\phi}\coloneq\phi(r)$ gives a 
strict contactomorphism $\bar{\phi}\colon H\to H$ which preserves $\Gamma$, hence descends to a strict contactomorphism of $N$, denoted again $\bar{\phi}$. We can form its mapping torus 
\[
N_{\bar{\phi}}=N\times_{(\bar{\phi},r)}\bR.
\]
The map $M\to N_{\bar{\phi}}$ given by $[(h,t)]\mapsto [([h],t)]$ is a diffeomorphism.
\end{remark}

\subsection{Relation with symplectic Lie groups}\label{rel_symplectic}

A Lie group $S$ of dimension $2n$ is said to be a symplectic Lie group if it admits a left-invariant $2$-form $\lvec{\sigma}$ which is closed and non-degenerate (see, for instance, \cite{BaCo,DaMe,LiMe}). 

Let $S$ be a Lie group with identity $\fe$ and let $\fs=T_{\fe}S$ be its Lie algebra; it is clear that $S$ is a symplectic Lie group if and only if $\fs$ is a symplectic Lie algebra. 
In fact, if $\lvec{\sigma}$ is a left-invariant symplectic structure on $S$ then $\sigma = \lvec{\sigma}(\fe) \in \Lambda^2{\fs}^*$ is a symplectic structure on $\fs$.

Next, we will discuss the relation between symplectic Lie groups and a particular class of lcs Lie groups of the first kind.

\subsubsection{Lcs Lie groups of the first kind with bi-invariant Lee vector field}

 First of all, we present the following example.
\begin{example}
\label{symplectic-lcs-first-kind}
Let $S$ be a Lie group of dimension $2n$ with Lie algebra $\fs$. Moreover, suppose given:
\begin{itemize}
\item a $2$-cocycle $\varphi\colon S \times S \to \mathbb{R}$ on $S$ with values in $\mathbb{R}$ such that the corresponding $2$-cocycle $\sigma\colon \fs\times\fs \to \mathbb{R}$ on $\fs$ with values in $\mathbb{R}$ given by 
(\ref{sigma}) is non-degenerate;
\item a representation $\phi\colon\mathbb{R} \to\Aut(S)$ of the abelian Lie group $\mathbb{R}$ on $S$ such that
\[
\phi_u^*\lvec{\sigma} = \lvec{\sigma}, \; \; \mbox{ for every } u \in \mathbb{R}.
\]
\end{itemize}
Note that the first condition implies that $(\fs,\sigma)$ is a symplectic Lie algebra, hence $S$ is a symplectic Lie group with symplectic form $\lvec{\sigma}$.

Therefore, the central extension $\bR\cext_\sigma\fs$ is a contact Lie algebra with contact structure $\eta = (1,0) \in\bR\oplus\fs^*$ and central Reeb vector 
$R = (1,0) \in \bR\cext_\sigma\fs$. This implies that the central extension $\bR\cext_{\varphi} S$ is a contact Lie group with contact structure $\lvec{\eta}$ and 
bi-invariant Reeb vector field $\lvec{R}$. Note that the left-invariant vector field $\lvec{R}$ is bi-invariant since $R$ belongs to the center of $\bR\cext_{\sigma} \fs$.

Now, denote by $D_{\fs}\colon\fs \to \fs$ the symplectic derivation induced by $\phi$. Then, from Proposition \ref{central-extension}, it follows that the linear map 
$D\colon\bR\cext_{\sigma}\fs \to\bR\cext_{\sigma}\fs$ given by
\begin{equation} \label{D-derivation}
D(u,X) = (0,D_{\fs}X), \; \; \; \mbox{ for } u \in \mathbb{R} \mbox{ and } X \in \fs
\end{equation}
is a contact derivation. Thus, it induces a contact representation $\tphi\colon\mathbb{R} \to\Aut(\bR\cext_{\varphi} S)$ and the semidirect product 
$G = (\bR\cext_{\varphi} S) \rtimes_{\tphi} \mathbb{R}$ is a lcs Lie group of the first kind with bi-invariant Lee vector field $(\lvec{R}, 0)$ (see Example \ref{lcs-Lie-group-first-kind}).

We remark the following facts:
\begin{itemize}
\item $\fg$, the Lie algebra of $G$, is the lcs extension $(\bR\cext_{\sigma}\fs) \rtimes_{D}\mathbb{R}$ of the symplectic Lie algebra $(\fs, \sigma)$ by the derivation $D_{\fs}$ (compare with Remark \ref{lcs_extension});
\item the element $(R, 0)\in\fg$ is central; this implies that the vector field $(\lvec{R}, 0)$ on $G$ is bi-invariant.
\end{itemize}

Next, we will present a more explicit description of the contact representation $\tphi\colon\mathbb{R} \to \Aut(\mathbb{R} \cext_{\varphi} S)$ in terms of the symplectic representation $\phi\colon\mathbb{R} \to \Aut(S)$. 

As we know, the multiplication in $\mathbb{R} \cext_{\varphi} S$ is given by
\begin{equation}\label{Mult-Central-Ext}
(u, s)(u', s') = (u + u' + \varphi(s, s'), ss'), \; \; \mbox{ for } (u, s), (u', s') \in \mathbb{R} \cext_{\varphi} S.
\end{equation}
Thus, we have the following expressions of the left-invariant vector fields on $\mathbb{R} \cext_{\varphi} S$,
\begin{equation}\label{Left-invariant-0}
\lvec{(1, 0)} = \displaystyle \frac{\partial}{\partial u}, \; \; \; \lvec{(0, X)}(u, s) = \displaystyle \frac{d}{dr}\Big|_{r=0} \varphi(s,\exp(rX)) \frac{\partial}{\partial u}\Big|_{(u, s)} + \lvec{X}(s),
\end{equation}
for $X \in \fs$ and $(u, s) \in \mathbb{R} \cext_{\varphi}S$.

Now, let $\tcM$ be the multiplicative vector field on $\mathbb{R} \cext_{\varphi} S$ whose flow is $\tphi$,
\[
\tcM(u, s) = \displaystyle \tpsi(u, s) \frac{\partial}{\partial u}\Big|_{(u, s)} + \bcM(u, s), \; \; \mbox{ for } (u, s) \in \mathbb{R} \cext_{\varphi} S,
\]
with $\tpsi \in\mathscr{C}^{\infty}(\mathbb{R} \cext_{\varphi} S)$ and $\bcM\colon \mathbb{R} \cext_{\varphi} S \to TS$ a time-dependent vector field on $S$.

From \eqref{der-induced}, \eqref{D-derivation} and \eqref{Left-invariant-0}, we deduce that
\[
0 = \displaystyle \lvec{D(1,0)} = \left[\frac{\partial}{\partial u}, \tcM\right]
\]
which implies that
\[
\tcM(u, s) = \tpsi(s) \frac{\partial}{\partial u}\Big|_{(u, s)} + \bcM(s), \; \; \mbox{ for } (u, s) \in \mathbb{R} \cext_{\varphi} S,
\]
with $\tpsi \in\mathscr{C}^{\infty}(S)$ and $\bcM \in \fX(S)$.

So, using \eqref{Mult-Central-Ext} and the fact that $\tcM$ is multiplicative, we conclude that $\bcM$ also is multiplicative. Moreover, from \eqref{D-derivation} and \eqref{Left-invariant-0}, 
it follows that the derivation on $\fs$ associated with $\bcM$ is just $D_{\fs}\colon \fs \to \fs$. Thus, $\bcM = \cM$, with $\cM$ the multiplicative vector field on $S$ whose flow 
is $\phi$, and
\[
\tcM(u, s) = \tpsi(s) \frac{\partial}{\partial u}\Big|_{(u, s)} + \cM(s).
\]
Therefore, the flow $\tphi\colon \mathbb{R} \to \Aut(\mathbb{R} \cext_{\varphi} S)$ of $\tcM$ has the form
\[
\tphi_t(u, s) = (u + \tchi_t(s), \phi_t(s))
\]
and
\begin{equation}\label{M-tilde}
\tcM(u, s) = \displaystyle \frac{d}{dt}\Big|_{t=0}(\tchi_t(s)) \frac{\partial}{\partial u}\Big|_{(u, s)} + {\mathcal M}(s)
\end{equation}
with $\tchi\in\mathscr{C}^{\infty}(\bR \times S)$. Since $\tphi$ is an action of $\mathbb{R}$ on $\mathbb{R} \cext_{\varphi} S$, we deduce that $\tchi$ satisfies the relation
\begin{equation}\label{Chi-Tilde-1}
\tchi_{t+t'}(s) = \tchi_t(s) + \tchi_{t'}(\phi_t(s)), \; \; \mbox{ for } t, t' \in \mathbb{R} \mbox{ and } s \in S.
\end{equation}
Moreover, using \eqref{Mult-Central-Ext} and the fact that $\tchi_t \in\Aut(\mathbb{R} \cext_{\varphi} S)$, for every $t \in \mathbb{R}$, we obtain that
\begin{equation}\label{Chi-Tilde-2}
\displaystyle \tchi_t(ss') - \tchi_t(s) - \tchi_t(s') = \varphi(\phi_t(s), \phi_t(s')) - \varphi (s, s'). 
\end{equation}
In addition, from \eqref{Left-invariant-0} and \eqref{M-tilde}, it follows that
\[
\lvec{D(0, X)} = [\lvec{(0, X)}, \tcM](-\varphi(\fe, \fe), \fe) = \displaystyle \frac{d}{dt}\Big|_{t=0} (d_{\fe}\tchi_t)(X) \frac{\partial}{\partial u}\Big|_{(-\varphi(\fe, \fe), \fe)} + D_{\fs}X,
\]
which, using (\ref{D-derivation}), implies that
\begin{equation}\label{Chi-Tilde-3}
\displaystyle \frac{d}{dt}\Big|_{t=0} (d_{\fe}\tchi_t) = 0.
\end{equation}
\end{example}
\begin{definition}
The Lie group $(\bR\cext_{\varphi} S) \rtimes_{\tphi} \bR$ endowed with the previous lcs structure of the first kind is called the \emph{lcs extension of the Lie group $S$ by the symplectic $2$-cocycle $\varphi$ and 
the symplectic representation $\phi\colon\mathbb{R}\to\Aut(S)$}.
\end{definition}

Lcs extensions of symplectic Lie groups by symplectic $2$-cocycles and symplectic representations are the models of connected simply connected lcs Lie groups of the first kind with bi-invariant Lee vector field. 
In fact, we can prove the following result.
\begin{theorem}\label{symplectic-lcs-Lie-group}
Let $G$ be a connected simply connected lcs Lie group of the first kind with bi-invariant Lee vector field. If $\fg$ is the Lie algebra of $G$ then $\fg$ is a lcs extension of a symplectic Lie algebra $\fs$ and if $S$ is a 
connected simply connected Lie group with Lie algebra $\fs$, we have that $G$ is isomorphic to a lcs extension of $S$. 
\end{theorem}
\begin{proof}
Let $\fg$ be the Lie algebra of $G$, let $(\omega, \eta)$ be the lcs structure of the first kind on $\fg$ and let $V \in \fg$ be the Lee vector. Using that the Lee vector field $\lvec{V}$ of $G$ is bi-invariant, 
we have that $V$ belongs to the center of $\fg$. Thus, from Theorem \ref{central-lcs} and Definition \ref{lcs-extension-algebraic}, we deduce that $\fg$ may be identified with the lcs extension of a symplectic 
Lie algebra $(\fs, \sigma)$ by a symplectic derivation $D_{\fs}\colon\fs \to\fs$. Under this identification, $\omega = ((0, 0), 1)\in (\bR\oplus\fs^*)\oplus\mathbb{R}$, 
$\eta = ((1,0), 0)\in(\bR\oplus\fs^*)\oplus\mathbb{R}$ and $V = ((1,0), 0) \in (\bR\cext_\sigma\mathbb{R})\rtimes_D\mathbb{R}$ (compare with Remark \ref{lcs_extension}). 

Now, let $S$ be a connected simply connected Lie group with Lie algebra $\fs$ and $\varphi\colon S \times S \to \mathbb{R}$ a $2$-cocycle on $S$ with values in $\mathbb{R}$ such that the corresponding $2$-cocycle 
on $\fs$ with values in $\mathbb{R}$ is just $\sigma$ (see Section \ref{central-ext-alg-group}). Then, the central extension $H = \bR\cext_{\varphi}S$ is a connected simply connected contact Lie group with 
Lie algebra $\bR\cext_{\sigma}\fs$, contact structure $\lvec{\eta} = \lvec{(1,0)}$ and bi-invariant Reeb vector field $\lvec{V} = \lvec{(1,0)}$.

Next, proceeding as in the proof of Theorem \ref{Universal-covering-lcs}, we can take a contact representation
\[
\tphi\colon \mathbb{R} \to\Aut(\bR\cext_{\varphi}S)
\]
of $\mathbb{R}$ on the contact Lie group $\bR\cext_{\varphi}S$ such that the corresponding contact derivation $D\colon\bR\cext_\sigma \fs \to \bR\cext_\sigma \fs$ is given by
\[
D(u,X) = (0,D_{\fs}(X)).
\]
Then, the semidirect product $(\bR\cext_{\varphi}S)\rtimes_{\tphi}\mathbb{R}$ is a connected simply connected Lie group with Lie algebra the lcs extension $(\bR\cext_{\sigma}\fs) \rtimes_{D}\mathbb{R}$ of 
$(\fs, \sigma)$ by $D_{\fs}$. Thus, since $\fg$ is isomorphic to $(\bR\cext_{\sigma}\fs) \rtimes_{D}\mathbb{R}$, there exists a Lie group isomorphism between $G$ and 
$(\bR\cext_{\varphi} S) \rtimes_{\tphi}\mathbb{R}$ and, under this isomorphism, the Lee 1-form of $G$ is just the left-invariant $1$-form $\lvec{((0,0), 1)}$, with 
$((0,0), 1) \in \bR\oplus\fs^*\oplus\mathbb{R}$.

Now, let $\phi\colon\mathbb{R} \to\Aut(S)$ be a representation of the abelian Lie group $\mathbb{R}$ on $S$ such that the corresponding derivation on $\fs$ is just $D_{\fs}\colon\fs\to\fs$. Denote by ${\mathcal M}$ the 
multiplicative vector field on $S$ whose flow is $\{\phi_t\}_{t \in \mathbb{R}}$. Then, using \eqref{der-induced} and the fact that $D_{\fs}$ is a symplectic derivation, we deduce that
\[
{\mathcal L}_{\mathcal M}\lvec{\sigma} = 0,
\]
which implies that $\phi$ is a symplectic representation, i.e., 
\[
\phi_t^*\lvec{\sigma} = \lvec{\sigma}, \; \; \mbox{ for every } t \in \mathbb{R}.
\]
Finally, proceeding as in Example \ref{symplectic-lcs-first-kind}, one may see that
\[
\tphi_t(u,s) = ( u + \tchi_t(s),\phi_t(s)), \; \; \mbox{ for } t, u \in \mathbb{R} \mbox{ and } s \in S,
\]
with $\tchi\colon \bR\times S \to \mathbb{R}$ a smooth map satisfying (\ref{Chi-Tilde-1}), (\ref{Chi-Tilde-2}) and (\ref{Chi-Tilde-3}).
\end{proof}
\begin{remark}\label{co-compact-discrete}
Let $S$ be a connected simply connected symplectic Lie group, $\varphi\colon S \times S \to \mathbb{R}$ a symplectic $2$-cocycle on $S$ with values in $\mathbb{R}$ and $\phi\colon\mathbb{R} \to\Aut(S)$ a symplectic 
representation of $\mathbb{R}$ on $S$. Suppose that $\Gamma_S$ is a lattice in $S$ and $p, q$ are real numbers such that:
\begin{itemize}
\item
the restriction of $\varphi$ to $\Gamma_S \times \Gamma_S$ takes values in the integer lattice $p\bZ$ and
\item
the restriction to the integer lattice $q\bZ$ of the corresponding contact representation $\tphi\colon\mathbb{R} \to\Aut(\bR\cext_{\varphi} S)$ takes values in $\Aut(p\bZ\cext_{\varphi} \Gamma_S)$. 
\end{itemize}
Then, $(p\bZ\cext_{\varphi} \Gamma_S )\rtimes_{\tphi} q\bZ$ is a lattice in the lcs Lie group of the first kind $(\bR\cext_{\varphi} S) \rtimes_{\tphi}\mathbb{R}$ and 
$M = ((p\bZ\cext_{\varphi} \Gamma_S )\rtimes_{\tphi} q\bZ)\backslash ((\bR\cext_{\varphi} S) \rtimes_{\tphi}\mathbb{R})$ is a compact manifold. 
Furthermore, since the lcs structure of the first kind on $(\bR\cext_{\varphi} S) \rtimes_{\tphi}\mathbb{R}$ is left-invariant, it induces a lcs structure of the first kind on $M$. 
Thus, $M$ is a compact lcs manifold of the first kind.
\end{remark}

\subsubsection{Symplectic extensions of symplectic Lie groups and lcs Lie groups}\label{symp_ext_symp_lcs_groups}

In some cases, the symplectic Lie group $S$ of dimension $2n$ (in the previous section) may in turn be obtained from a symplectic Lie group $S_1$ of dimension $2n-2$.

In fact, {\em an integrated version} of the symplectic double extension of a symplectic Lie algebra (see Section \ref{contact-alg-lcs-alg}) produces $S$ from $S_1$.
This process may be described as follows.

Let $S_1$ be a Lie group of dimension $2n-2$ with Lie algebra $\fs_{1}$ and suppose that:
\begin{itemize}
\item there exists a $2$-cocycle on $S_1$ with values in $\mathbb{R}$, $\varphi_1\colon S_1 \times S_1 \to \mathbb{R}$, such that the corresponding $2$-cocycle on $\fs_{1}$ with values in $\mathbb{R}$, 
$\sigma_1\colon \fs_{1} \times \fs_{1} \to \mathbb{R}$ is non-degenerate and
\item there exists a representation of the abelian group $\mathbb{R}$ on $S_1$, $\phi_1\colon \mathbb{R} \to\Aut(S_1)$.
\end{itemize}

Under these conditions, the left-invariant $2$-form $\lvec{\sigma_1}$ is a symplectic structure on $S_1$.

Now, let $D_{\fs_{1}}\colon \fs_{1} \to \fs_{1}$ be the derivation on $\fs_{1}$ associated with the representation $\phi_1\colon\mathbb{R} \to\Aut(S_1)$ and 
$D_{\fs_{1}}^*\sigma_1\colon \fs_{1} \times \fs_1 \to \mathbb{R}$ the $2$-cocycle on $\fs_1$ with values in $\mathbb{R}$ given by \eqref{D-s1-star}. Then, we will define a $2$-cocycle on $S_1$ with values in
$\mathbb{R}$ such that the corresponding $2$-cocycle on $\fs_{1}$ is just $D_{\fs_{1}}^*\sigma_{1}$.

For this purpose, we will consider the map $(\varphi_{1}, \phi_{1})\colon S_1 \times S_1 \to \mathbb{R}$ given by
\[
(\varphi_1, \phi_1)(s_1, s'_1) = \displaystyle \frac{d}{dr}\Big|_{r =0}\varphi_1(\phi_1(r)(s_1), \phi_1(r)(s'_1)), \; \; \mbox{ for } s_1, s'_1 \in S_1.
\]
Using that $\varphi_1\colon S_1 \times S_1 \to \mathbb{R}$ is a $2$-cocycle on $S_1$ and the fact that $\phi_1\colon \mathbb{R} \to\Aut(S_1)$ is a representation, we have that $(\varphi_1, \phi_1)$ also is a $2$-cocycle. 
In addition, if $(\sigma_{1}, T_{\fe_1}\phi_{1})\colon \fs_{1} \times \fs_{1} \to \mathbb{R}$ is the $2$-cocycle on $\fs_{1}$ associated with $(\varphi_1, \phi_1)$ then, from \eqref{sigma}, it follows that
\[
(\sigma_1, T_{\fe_1}\phi_1)(X_1, X'_1) = \displaystyle \frac{d}{dr}\Big|_{r = 0} \sigma_1((T_{\fe_1}\phi_1(r))(X_1), (T_{\fe_1}\phi_1(r))(X'_1)).
\]
On the other hand, using \eqref{der-induced}, we obtain that
\[
\lvec{D_{\fs_{1}}X_1} = \displaystyle \frac{d}{dr}\Big|_{r=0}(T_{\fe_1}\phi_1(r))(X_1).
\]

This implies that
\[
(\sigma_1, T_{\fe_1}\phi_1)(X_1, X'_1) = \sigma_1(D_{\fs_{1}}X_1, X'_1) + \sigma_1(X_1, D_{\fs_{1}}X'_1).
\]
In other words, $(\sigma_1, T_{\fe_1}\phi_1)$ is just the $2$-cocycle $D^*_{\fs_{1}}\sigma_1$ and, thus, $(\varphi_1, \phi_1)$ is a $2$-cocycle on $S_1$ associated with $D^*_{\fs_{1}}\sigma_1$. Therefore, we may consider the 
central extension $H_1 = \mathbb{R} \cext_{(\varphi_1, \phi_1)}S_1$. We remark that the multiplication in $H_1$ is given by
\begin{equation}\label{mult-H1}
(r, s_1)(r', s'_1) = (r + r' + (\varphi_1, \phi_1)(s_1, s'_1), s_1 s'_1).
\end{equation}
We will denote by $\fh_{1}$ the Lie algebra of $H_1$. It follows that $\fh_1 = \mathbb{R} \cext_{D_{\fs_1}^*\sigma_1}\fs_1$ and, moreover,
\begin{equation}\label{left-invariant}
\lvec{(1, 0)} = \displaystyle \frac{\partial}{\partial r}, \; \; \; \lvec{(0, X_1)}(r, s_1) = \displaystyle \frac{d}{du}\Big|_{u =0}(\varphi_1, \phi_1)(s_1,\exp(uX_1))\frac{\partial}{\partial r}\Big|_{(r, s_1)} + \lvec{X_1}(s_1)
\end{equation}
for $X_1 \in \fs_1$ and $(r, s_1) \in \mathbb{R} \cext_{(\varphi_1, \phi_1)}S_1 = H_1$.

Now, as in Section \ref{sym-Lie-alg-lcs-Lie}, suppose that $Z_1 \in \fs_1$ and that the map $(-i_{Z_1}\sigma_1, -D_{\fs_1}) \colon \fh_1 \to \fh_1$ given by
\begin{equation}\label{Dh1}
(-i_{Z_1}\sigma_1, -D_{\fs_1})(r, X_1) = (-\sigma_1(Z_1, X_1), -D_{\fs_1}X_1), \; \; \mbox{ for } r \in \mathbb{R} \mbox{ and } X_1 \in \fs_1,
\end{equation}
is a derivation on the Lie algebra $\fh_1$. Denote by 
\[
\phi_{H_1}\colon \mathbb{R} \to\Aut(H_1) =\Aut(\mathbb{R} \cext_{(\varphi_1, \phi_1)}S_1)
\]
the representation of the abelian Lie group $\mathbb{R}$ on $H_1$ associated with the derivation $(-i_{Z_1}\sigma_1, D_{\fs_1})\colon \fh_1 \to \fh_1$ and by ${\mathcal M}_{H_1} \in \fX(H_1)$ the 
multiplicative vector field on $H_1$ whose flow is $\phi_{H_1}$. Then,
\[
{\mathcal M}_{H_1}(r, s_1) = \displaystyle \psi(r, s_1) \frac{\partial}{\partial r}\Big|_{(r, s_1)} + {\mathcal M}_{S_1}(r, s_1), \; \; \mbox{ for } (r, s_1) \in \mathbb{R} \cext_{(\varphi_1, \phi_1)} S_1 = H_1,
\]
with $\psi \in\mathscr{C}^{\infty}(H_1)$ and ${\mathcal M}_{S_1}\colon H_1 \to TS_1$ a time-dependent vector field on $S_1$. Since
\[
0 = \lvec{(Z_1, D_{\fs_1})(1, 0)} = [\lvec{(1, 0)}, {\mathcal M}_{H_1}],
\]
we conclude from \eqref{left-invariant} that
\[
{\mathcal M}_{H_1}(r, s_1) = \displaystyle \psi(s_1) \frac{\partial}{\partial r}\Big|_{(r, s_1)} + {\mathcal M}_{S_1}(s_1),
\]
with $\psi \in\mathscr{C}^{\infty}(S_1)$ and ${\mathcal M}_{S_1} \in \fX(S_1)$. Now, using \eqref{mult-H1} and the fact that ${\mathcal M}_{H_1}$ is multiplicative, 
it follows that ${\mathcal M}_{S_1}$ also is multiplicative. In addition, from \eqref{left-invariant} and \eqref{Dh1}, 
we deduce that the derivation on $\fs_1$ associated with ${\mathcal M}_{S_1}$ is just $-D_{\fs_1}\colon \fs_1 \to \fs_1$. Therefore, ${\mathcal M}_{S_1} = -{\mathcal M}_1$, with ${\mathcal M}_1$ the multiplicative vector 
field on $S_1$ whose flow is $\phi_1$, and
\[
{\mathcal M}_{H_1}(r, s_1) = \psi(s_1) \frac{\partial}{\partial r}\Big|_{(r, s_1)} - {\mathcal M}_1(s_1).
\]
This implies that the flow $\phi_{H_1}\colon \mathbb{R} \to\Aut(H_1)$
of ${\mathcal M}_{H_1}$ has the form
\[
\phi_{H_1}(t)(r, s_1) = (r + \chi_t(s_1), \phi_1(-t)(s_1))
\]
and
\begin{equation}\label{MH1}
{\mathcal M}_{H_1}(r, s_1) = \displaystyle \frac{d}{dt}\Big|_{t=0}(\chi_t(s_1)) \frac{\partial}{\partial r}\Big|_{(r, s_1)} - {\mathcal M}_1(s_1)
\end{equation}
with $\chi \in \mathscr{C}^{\infty}(\mathbb{R} \times S_1)$.

In what follows, we will denote by $(\chi, \phi_1)$ the representation $\phi_{H_1}$. Since $(\chi, \phi_1)$ is an action of $\mathbb{R}$ on $H_1$, we deduce that $\chi$ satisfies the relation
\begin{equation}\label{double-symp-1}
\chi_{t+t'}(s_1) = \chi_t(s_1) + \chi_{t'}(\phi_1(-t)s_1), \; \; \mbox{ for } t, t' \in \mathbb{R} \mbox{ and } s_1 \in S_1.
\end{equation}
Moreover, using \eqref{mult-H1} and the fact that $(\chi, \phi_1)(t) \in\Aut(H_1)$, for every $t \in \mathbb{R}$, we obtain that
\begin{equation}\label{double-symp-2}
\displaystyle \chi_t(s_1s'_1) - \chi_t(s_1) - \chi_t(s'_1) = \frac{d}{dr}\Big|_{r=-t} \varphi_1(\phi_1(r)s_1, \phi_1(r)s'_1) - \frac{d}{dr}\Big|_{r=0} \varphi_1(\phi_1(r)s_1, \phi_1(r)s'_1). 
\end{equation}
In addition, from \eqref{left-invariant} and \eqref{MH1}, it follows that
\[
[\lvec{(0, X_1)}, {\mathcal M}_{H_1}](0, \fe_1) = \displaystyle \frac{d}{dt}\Big|_{t=0} (d_{\fe_1}\chi_t)(X_1) \frac{\partial}{\partial r}\Big|_{(0, \fe_1)} - D_{\fs_1}X_1,
\]
which, using \eqref{Dh1}, implies that
\begin{equation}\label{double-symp-3}
\displaystyle \frac{d}{dt}\Big|_{t=0} (d_{\fe_1}\chi_t)(X_1) = -\sigma_1(Z_1, X_1), \; \; \mbox{ for } X_1 \in \fs_1.
\end{equation}
Next, we consider the Lie group
\[
S = H_1 \rtimes_{(\chi, \phi_1)} \mathbb{R} = (\mathbb{R} \cext_{(\varphi_1, \phi_1)}S_1) \rtimes_{(\chi, \phi_{1})} \mathbb{R}
\]
with Lie algebra
\[
\fs = \fh_1 \rtimes_{(-i_{Z_1}\sigma_1, -D_{\fs_1})} \mathbb{R} = (\mathbb{R} \cext_{D^*_{\fs_1}\sigma_1}\fs_1) \rtimes_{(-i_{Z_1}\sigma_1, -D_{\fs_1})}\mathbb{R}.
\]
Then, following the construction in Section \ref{contact-alg-lcs-alg}, we have that the $2$-cocycle $\sigma$ on $\fs$ given by \eqref{sym-str-double-ext} is non-degenerate and it defines a 
left-invariant symplectic $2$-form $\lvec{\sigma}$ on $S$. Thus, $S$ is a symplectic Lie group.

Using a similar terminology as in the Lie algebra case, we introduce the following definition.
\begin{definition} 
The symplectic Lie group $(S, \lvec{\sigma})$ is the \emph{double extension} of the symplectic Lie group $(S_1, \lvec{\sigma_1})$ by the symplectic $2$-cocycle $\varphi_1\colon S_1 \times S_1 \to \mathbb{R}$ on $S_1$, 
the representation $\phi_1\colon\mathbb{R} \to\Aut(S_1)$ of $\mathbb{R}$ on $S_1$ and the smooth map $\chi\colon\mathbb{R} \times S_1 \to \mathbb{R}$, 
the latter satisfying \eqref{double-symp-1}, \eqref{double-symp-2} and \eqref{double-symp-3}.
\end{definition}
We remark that the discussion in this section proves the following result
\begin{proposition}\label{sequence-sym-Lie-group}
Let $(S, \lvec{\sigma})$ be a connected simply connected Lie group with Lie algebra $\fs$ and suppose that $\fs$ is the double extension of the symplectic Lie 
algebra $(\fs_1, \sigma_1)$ by a derivation $D_{\fs_1}\colon \fs_1 \to \fs_1$ and an element $Z_1 \in \fs_1$ satisfying \eqref{der-ext-central}. Then, we can choose the following objects:
\begin{itemize}
\item
a connected, simply connected Lie group $S_1$ with Lie algebra $\fs_1$;
\item
a $2$-cocycle $\varphi_1\colon S_1 \times S_1 \to \mathbb{R}$ on $S_1$ such that the corresponding $2$-cocycle on $\fs_1$ is just $\sigma_1$;
\item
a representation $\phi_1\colon\mathbb{R} \to\Aut(S_1)$ of $\mathbb{R}$ on $S_1$ such that the corresponding derivation on $\fs_1$ is $D_{\fs_1}$ and
\item
a smooth map $\chi\colon\mathbb{R} \times S_1 \to \mathbb{R}$ which satisfies \eqref{double-symp-1}, \eqref{double-symp-2} and \eqref{double-symp-3}.
\end{itemize}
In addition, $(S, \lvec{\sigma})$ is the double extension of the symplectic Lie group $(S_1, \lvec{\sigma_1})$ by $\varphi_1$, $\phi_1$ and $\chi$.
\end{proposition}
\begin{remark}\label{co-compact-discrete-symplectic}
Let $(S, \lvec{\sigma})$ be a connected simply connected symplectic Lie group which is the double extension of the connected simply connected symplectic Lie group $(S_1, \lvec{\sigma_1})$ by 
$\varphi_1\colon S_1 \times S_1 \to \mathbb{R}$, $\phi_1\colon \mathbb{R} \to\Aut(S_1)$ and $\chi\colon\mathbb{R} \times S_1 \to \mathbb{R}$. Suppose that $\Gamma_{S_1}$ is a lattice in $S_1$ and $p_1, q_1$ are 
real numbers such that:
\begin{itemize}
\item
the restriction of $(\varphi_1, \phi_1)$ to $\Gamma_{S_1} \times \Gamma_{S_1}$ takes values in the integer lattice $p_1\bZ$ and
\item
the restriction to the integer lattice $q_1\bZ$ of the representation $(\chi, \phi_1)\colon\mathbb{R} \to\Aut(\bR\cext_{(\varphi_1, \phi_1)} S_1)$ takes values in $\Aut(p_1\bZ\cext_{(\varphi_1, \phi_1)} \Gamma_{S_1})$. 
\end{itemize}
Then $\Gamma=(p_1\bZ\cext_{(\varphi_1, \phi_1)} \Gamma_{S_1})\rtimes_{(\chi, \phi_1)} q_1\bZ$ is a lattice in the symplectic Lie group $S = (\bR\cext_{(\varphi_1, \phi_1)} S_1) \rtimes_{(\chi, \phi_1)}\mathbb{R}$ and 
$M = \Gamma\backslash S$ is a compact manifold. 
Furthermore, since the symplectic structure on $S$ is left-invariant, it induces a symplectic structure on $M$. 
Thus, $M$ is a compact symplectic manifold.
\end{remark}
Now, let $(S, \lvec{\sigma})$ be the double extension of the symplectic Lie group $(S_1, \lvec{\sigma_1})$ by the $2$-cocycle $\varphi_1\colon S_1 \times S_1 \to \mathbb{R}$, the representation $\phi_1\colon \mathbb{R} \to
\Aut(S_1)$ and the smooth map $\chi\colon\mathbb{R} \times S_1 \to \mathbb{R}$. Moreover, 
suppose that:
\begin{itemize}
\item $\varphi\colon S \times S \to \mathbb{R}$ is a $2$-cocycle on $S$ with values in $\mathbb{R}$ such that the corresponding $2$-cocycle on $\fs$ is just $\sigma$ and
\item $\phi\colon\mathbb{R} \to \Aut(S)$ is a symplectic representation of $\mathbb{R}$ on $S$.
\end{itemize}
Then, one may consider the lcs extension
\[
(\mathbb{R} \cext_{\varphi} S) \rtimes_{\tphi} \mathbb{R} = \mathbb{R} \cext_{\varphi} ((\mathbb{R} \cext_{(\varphi_1, \phi_1)}S_1) \rtimes_{(\chi, \phi_1)} \mathbb{R}\rtimes_{\tphi} \mathbb{R}
\]
of $S$ by $\varphi$ and $\phi$ which is a lcs Lie group of the first kind with bi-invariant Lee vector field by Example \ref{symplectic-lcs-first-kind}.

These results will be useful in the next section. In fact, we will see that every connected simply connected lcs nilpotent Lie group with non-zero Lee $1$-form of dimension $2n$ is obtained as the lcs extension by a nilpotent 
representation of a connected simply connected symplectic nilpotent Lie group $S$ of dimension $2n-2$ and, in turn, $S$ may be obtained by a sequence of $n-1$ symplectic double extensions by nilpotent representations 
from the abelian Lie group $\mathbb{R}^2$.

\subsection{Lcs nilpotent Lie groups}

In this section, we discuss lcs structures (with non-zero Lee $1$-form) on nilpotent Lie groups. 

\begin{example}
Let $(S, \lvec{\sigma})$ be the double extension of the symplectic Lie group $(S_1, \lvec{\sigma_1})$ by the symplectic $2$-cocycle $\varphi_1\colon S_1 \times S_1 \to \mathbb{R}$, 
the representation $\phi_1\colon\mathbb{R} \to \Aut(S_1)$ and the smooth map $\chi\colon \mathbb{R} \times S_1 \to \mathbb{R}$ satisfying \eqref{double-symp-1}, \eqref{double-symp-2} and \eqref{double-symp-3}.

Now, suppose that $S_1$ is nilpotent and that $\phi_1$ also is. Then, the corresponding derivation $D_{\fs_1}$ on the Lie algebra $\fs_1$ of $S_1$ is nilpotent. 
Furthermore, the Lie group $H_1 = \mathbb{R} \cext_{(\varphi_1, \phi_1)} S_1$ is nilpotent and the derivation $(-i_{Z_1}\sigma_1, -D_{\fs_1})$ on the Lie algebra of $H_1$ given by \eqref{Dh1} also is nilpotent. So, using that 
the Lie algebra $\fs$ of $S$ is $\fs = \fh_1 \rtimes_{(-i_{Z_1}\sigma_1, -D_{\fs_1})} \mathbb{R}$, we deduce that $(S, \lvec{\sigma})$ is a nilpotent symplectic Lie group, the nilpotent double extension of $(S_1, \lvec{\sigma_1})$ 
by $\varphi_1$, $\phi_1$ and $\chi$.

Next, let $\varphi\colon S \times S \to \mathbb{R}$ be a symplectic $2$-cocycle on $S$ and $\phi\colon\mathbb{R} \to\Aut(S)$ a symplectic representation of $\mathbb{R}$ on $S$. Denote by $D_{\fs}\colon 
\fs \to \fs$ the symplectic derivation associated with $\phi$. Then, we can consider the lcs extension
\[
G = (\mathbb{R} \cext_{\varphi} S) \rtimes_{\tphi} \mathbb{R} = \mathbb{R} \cext_{\varphi} ((\mathbb{R} \cext_{(\varphi_1, \phi_1)}S_1) \rtimes_{(\chi, \phi_1)} \mathbb{R}) \rtimes_{\tphi} \mathbb{R}
\]
It is a lcs Lie group of the first kind with bi-invariant Lee vector field and with Lie algebra the lcs extension of $\fs$ by $\sigma$ and $D$,
\[
\fg = (\mathbb{R} \cext_{\sigma} \fs) \rtimes_{D} \mathbb{R},
\]
$D$ being the contact derivation on $\mathbb{R} \cext_{\sigma} \fs$ associated with the contact representation $\tphi\colon\bR\to\mathbb{R} \cext_{\varphi} S$ lifting $\phi$.

Assume that $\phi$ is nilpotent. Then, the derivation $D_{\fs}$ is nilpotent and, using Theorem \ref{first-descrip-lcs-nilpotent}, it follows that $\fg$ is a lcs nilpotent Lie algebra with non-zero central Lee vector. 
Therefore, $G$ is a lcs nilpotent Lie group with non-zero bi-invariant Lee vector field.
\end{example} 
Next, we will see that the Lie group $G$ in the previous example is the model of a connected simply connected lcs nilpotent Lie group with non-zero Lee $1$-form. In fact, we will prove the following result.
\begin{theorem}\label{structure_nilpotent}
Let $G$ be a connected simply connected lcs nilpotent Lie group of dimension $2n+2$ with non-zero Lee $1$-form. Then, $G$ is isomorphic to a lcs extension of a connected simply connected symplectic nilpotent Lie group $S$ and, 
in turn, $S$ is isomorphic to a symplectic Lie group which may be obtained as a sequence of $(n-1)$ symplectic nilpotent double extensions from the abelian Lie group $\mathbb{R}^2$.
\end{theorem}
\begin{proof}
Let $\fg$ be the Lie algebra of $G$. Then, using Theorem \ref{first-descrip-lcs-nilpotent}, we deduce that $\fg$ is a lcs Lie algebra of the first kind with non-zero central Lee vector. Thus, $G$ is lcs Lie group of the first 
kind with non-zero bi-invariant Lee vector field.

On the other hand, from Theorem \ref{description-Lie-alg-nil}, it follows that $\fg$ is the lcs extension of a symplectic nilpotent Lie algebra $\fs$ (of dimension $2n$) by a symplectic nilpotent derivation and, in turn, $\fs$ 
may be obtained as a sequence of $(n-1)$ symplectic double extensions by nilpotent derivations from the abelian Lie algebra $\mathbb{R}^2$.

Denote by $\fs_1, \dots ,\fs_{n-2}, \fs_{n-1} = \mathbb{R}^2$ the corresponding sequence of symplectic nilpotent Lie algebras and by $S$ (respectively, $S_i$, with $i = 1, \dots, n-1$) a connected simply connected nilpotent Lie 
group with Lie algebra $\fs$ (respectively, $\fs_i$, with $i = 1, \dots, n-1$).

Then, using Theorem \ref{symplectic-lcs-Lie-group} and Proposition \ref{sequence-sym-Lie-group}, we deduce that $G$ is isomorphic to a lcs extension of $S$ and, in turn, $S$ is isomorphic to a symplectic nilpotent Lie group which 
may be obtained as a sequence of $(n-1)$ symplectic nilpotent double extensions from the abelian Lie group $\mathbb{R}^2$. The corresponding sequence of $(n-1)$ connected simply connected nilpotent symplectic Lie groups is 
$S_1, \dots, S_{n-2}, S_{n-1} = \mathbb{R}^2$. 
\end{proof}

\begin{remark}\label{discrete-subgroup-lcs-nilpotent_1}
Let $S$ be a connected simply connected symplectic nilpotent Lie group with Lie algebra $\fs$, let $\varphi\colon S \times S \to \mathbb{R}$ be a symplectic $2$-cocycle on $S$ with values in $\mathbb{R}$ and 
let $\phi\colon\mathbb{R} \to\Aut(S)$ be a symplectic nilpotent representation of $\mathbb{R}$ on $S$. Moreover, suppose that the structure constants of $\fs$ with respect to a basis are rational numbers. 
It follows from Theorem \ref{Maltsev} that $S$ admits a lattice $\Gamma_S$. In addition, we will assume that we can choose two real numbers $p$ and $q$ such that $\Gamma_S$, $p$ and $q$ satisfy the conditions in Remark 
\ref{co-compact-discrete}. Under these conditions,
\[
M = ((p\bZ\cext_{\varphi} \Gamma_S )\rtimes_{\tilde{\phi}} q\bZ)\backslash ((\bR\cext_{\varphi} S) \rtimes_{\tilde{\phi}}\mathbb{R})
\] 
is a compact lcs nilmanifold of the first kind.
\end{remark}

\begin{remark}\label{discrete-subgroup-lcs-nilpotent_2}
A particular case of the previous situation is the following one. Suppose that $S$ is the nilpotent double extension of a connected simply connected symplectic nilpotent Lie group $S_1$ with Lie algebra $\fs_1$. 
Moreover, suppose that the structure constants of $\fs_1$ with respect to a basis are rational numbers. Then it follows from Theorem \ref{Maltsev} that $S_1$ admits a lattice $\Gamma_{S_1}$. 

In addition, we will assume that we can choose two real numbers $p_1$ and $q_1$ such that $\Gamma_{S_1}$, $p_1$ and $q_1$ satisfy the conditions in Remark \ref{co-compact-discrete-symplectic}. 
Under these conditions, we have that
\[
\Gamma_S = (p_1\bZ\cext_{(\varphi_1, \phi_1)} \Gamma_{S_1})\rtimes_{(\chi, \phi_1)} q_1\bZ
\]
is a lattice of $S$. Finally, if $p$ and $q$ are real numbers as in Remark \ref{discrete-subgroup-lcs-nilpotent_1}, we deduce that
\[
M = (p \bZ \cext_{\varphi}((p_1\bZ\cext_{(\varphi_1, \phi_1)} \Gamma_{S_1})\rtimes_{(\chi, \phi_1)} q_1\bZ) \rtimes_{\tphi} q\bZ)\backslash 
((\bR\cext_{\varphi} ((\bR\cext_{(\varphi_1, \phi_1)} S_1) \rtimes_{(\chi, \phi_1)}\mathbb{R}) \rtimes_{\tphi}\mathbb{R})
\]
is a compact lcs nilmanifold of the first kind.
\end{remark}

To conclude, we show how to recover the examples of Section \ref{sec:examples} in the framework of \thmref{structure_nilpotent} and Remarks \ref{discrete-subgroup-lcs-nilpotent_1} and \ref{discrete-subgroup-lcs-nilpotent_2}.

We start with the 4-dimensional nilpotent Lie group $G$ constructed in Section \ref{Dimension:4}. 
In this case, $S$ is $\bR^2$ with its structure of abelian Lie group and $H$ is the central extension $\bR\cext_\varphi\bR^2$ with respect to the 2-cocycle $\varphi\colon\bR^2\times\bR^2\to\bR$, 
$((x,y),(x',y'))\mapsto yx'$. Using \eqref{sigma}, $\varphi$ determines the 2-cocycle $\sigma$ on $\bR^2$, which is the standard symplectic form $\sigma=dx\wedge dy$ on $\bR^2$. Further, one can see that the 
symplectic nilpotent representation $\phi\colon\bR\to\Aut(\bR^2)$ is given by $\phi_t(x,y)=(x+ty,y)$. In the notation of Section \ref{rel_symplectic}, the function $\tchi\colon\bR\times\bR^2\to\bR$ determining the lift of the 
symplectic representation $\phi_t$ to a contact representation $\tphi\colon\bR\to\Aut(\bR\cext_\varphi\bR^2)$ is $\tchi_t(x,y)=t\frac{y^2}{2}$, which verifies properties \eqref{Chi-Tilde-1}-\eqref{Chi-Tilde-3}. Thus 
$G=(\bR\cext_\varphi\bR^2)\rtimes_{\tphi}\bR$ is the lcs extension of $S$ by $\varphi$ and $\phi$.
Concerning the lattice, we consider $\bZ\times 2\bZ\subset\bR^2$; one sees that $\varphi|_{\bZ\times 2\bZ}$ takes values in $2\bZ$. Then $\Xi=(2\bZ\cext_\varphi(\bZ\times 2\bZ))\rtimes_{\tphi}\bZ\subset G$ is a lattice.

Let us now move to the examples of Section \ref{Dimension:6}. In both of them, $H=\bR\cext_\varphi S$ where:
\begin{itemize}
 \item $S$ is the 4-dimensional nilpotent Lie group with multiplication
 \[
  (x,y,z,t)\cdot(x',y',z',t')=(x+x',y+y',z+z'+xx',t+t'+yx')
 \]
\item $\varphi\colon S\times S\to\bR$ is the 2-cocycle
\[
 \varphi((x,y,z,t),(x',y',z',t'))=yz'+tx'.
\]
\end{itemize}

First of all, we describe $S$ as a symplectic double extension of $S_1=\bR^2$. We consider $\bR^2$ with global coordinates $(x,z)$ and symplectic form $\sigma_1=dx\wedge dz$. The data associated to the symplectic 
double extension of $\fs_1=\bR^2$ are, in the notation of Section \ref{sym-Lie-alg-lcs-Lie}, the vector $Z_1= (0, 1)$ and the trivial derivation $D_{\fs_1}\colon\bR^2\to\bR^2$. Hence 
$\fs=(\bR\oplus\bR^2)\rtimes_{(-i_{Z_1}\sigma_1,0)}\bR$. As for the group structure of Section 
\ref{symp_ext_symp_lcs_groups}, we see that the 2-cocycle $(\varphi_1,\phi_1)$ is trivial, hence $S=(\bR\times\bR^2)\rtimes_{(\chi,\Id)}\bR$, where $\chi_y(x,z)=yx$, which verifies \eqref{double-symp-1}, \eqref{double-symp-2} 
and \eqref{double-symp-3}.

In the first example we chose the trivial symplectic representation on $S$ and lifted 
it to the trivial contact representation on $H=\bR\cext_\varphi S$. Hence, the nilpotent Lie group we work with is $G=(\bR\cext_\varphi S)\times\bR$. 
Concerning the lattice, we start with $\bZ^2\subset \bR^2$; then $\Gamma_S=(\bZ\times\bZ^2)\rtimes_{(\chi,\Id)}\bZ \subset S$ is a lattice.
The lattice we take in $G$ is $(\bZ\cext_\varphi \Gamma_S)\times \bZ$.

In the second example, the symplectic representation $\phi\colon\bR\to\Aut(S)$ is non-trivial and given by 
$\phi_s(x,y,z,t)=(x,y+sx,z+sy+\frac{1}{2}s^2x,t+sz+\frac{1}{2}s^2y+\frac{1}{6}s^3x)$; the function $\tchi\colon\bR\times S\to\bR$ which determines the lift of $\phi$ to a contact representation 
$\tphi\colon\bR\to\Aut(\bR\cext_\varphi S)$ is 
\[
\tchi_s(x,y,z,t)=s\left(xz+\frac{1}{2}y^2-\frac{1}{3}x^3\right)+s^2xy+\frac{1}{3}s^3x^2
\]
which verifies properties \eqref{Chi-Tilde-1}, \eqref{Chi-Tilde-2} and \eqref{Chi-Tilde-3}. For the lattice, we start with $\bZ\times 6\bZ\times \bZ\subset\bR\times \bR^2$
and $\Lambda_S=(\bZ\times 6\bZ\times \bZ)\rtimes_{(\chi,\Id)}2\bZ\subset S$; the restriction of $\varphi$ to such subgroup takes values in $2\bZ$. Then
$(2\bZ\cext_\varphi\Lambda_S)\rtimes_{\tphi}\bZ$ is the desired lattice.

Finally, we consider the example of Section \ref{Dimension:higher}. There one starts with the nilpotent Lie group $S_{2n-2}$ of dimension $2n-2$ ($n\geq 4$), whose group operation is given by
\[
\begin{array}{c}
(x_1, y_1, \dots, x_{n-1}, y_{n-1})\cdot(x'_1, y'_1, \dots, x'_{n-1}, y'_{n-1}) \\
= (x_1 + x'_1, y_1 + y'_1 + x_{n-1}x'_{n-1}, x_2 + x'_2, y_2 + y'_2, \dots, x_{n-2} + x'_{n-2}, y_{n-2} + y'_{n-2}, \\
x_{n-1} + x'_{n-1}, y_{n-1} + y'_{n-1} + x_1x'_{n-1})
\end{array}
\]
in terms of a global system of coordinates $(x_1, y_1, \dots, x_{n-1}, y_{n-1})$. The 2-cocycle $\varphi\colon S_{2n-2}\times S_{2n-2}\to\bR$ is given by
\[
 \varphi((x_1, y_1, \dots, x_{n-1}, y_{n-1}),(x'_1, y'_1, \dots, x'_{n-1}, y'_{n-1}))=\sum_{j=1}^{n-2}x_jy'_j+y_{n-1}x'_{n-1}
\]
and $H_{2n-1}=\bR\cext_\varphi S_{2n-2}$. 

We describe $S_{2n-2}$ as a symplectic double extension; in the notation of Section \ref{Dimension:higher}, we consider 
the basis $\{\alpha_1,\ldots,\alpha_{n-1},\beta_1,\ldots,\beta_{n-1}\}$ of $\fs^*_{2n-2}$ with structure equations
\[
 d\alpha_i=0, \ i=1,\ldots,n-1, \quad d\beta_i=0, \ i=1,\ldots,n-2 \quad \mathrm{and} \quad d\beta_{n-1}=\alpha_1\wedge\alpha_{n-1}.
\]
Hence $\fs_{2n-2}$ is the direct sum of the abelian Lie algebra $\bR^{2n-6}$ with a 4-dimensional Lie algebra, whose dual is spanned by $\{\alpha_1,\alpha_{n-1},\beta_1,\beta_{n-1}\}$, which is the direct sum of 
the 3-dimensional Heisenberg algebra and a 1-dimensional factor. We explained in the discussion of the 6-dimensional examples how this 4-dimensional Lie algebra can be described as a symplectic double extension of $\bR^2$.

In this case as well we have chosen the trivial symplectic representation on $S_{2n-2}$ and lifted it to the trivial contact representation on $H_{2n-1}$. Hence $G=(\bR\cext_\varphi S_{2n-2})\times\bR$ in this case. 
There is a lattice $\Xi_{2n-2}\subset S_{2n-2}$, consisting of points with integer coordinates. $\varphi$ restricted to $\Xi_{2n-2}\times\Xi_{2n-2}$ takes values into $\bZ$, hence 
$\Gamma_{2n-1}=\bZ\cext_\varphi\Xi_{2n-2}\subset H_{2n-1}$ is a lattice. The lattice in $G$ is $(\bZ\cext_\varphi\Xi_{2n-2})\times\bZ$.

\section*{Acknowledgements}
This work was partially supported by SFB 701 - Spectral Structures and
Topological Methods in Mathematics (Universit\"at Bielefeld) (G.\ B.) and MICINN (Spain) grant MTM2012-34478 (J.\ C.\ M.). The authors would like to thank 
Daniele Angella (Universit\`a di Firenze, Italy), Marisa Fern\'andez (EHU, Spain), Dieter Kotschick (LMU, Germany) and Luis Ugarte (Universidad de Zaragoza, Spain) for discussion on locally conformal symplectic structures, 
David Mart{\'\i}nez Torres (PUC-Rio, Brazil) for his comments on equivariant versions of Martinet's result in contact geometry and 
Jesus Alvarez L\'opez (USC, Spain) for his comments on foliations of codimension one.

\end{document}